%% file: main.tex
\documentclass[11pt]{amsart}
\usepackage{amsmath}
\usepackage{bbm}
\usepackage{mathrsfs}
\usepackage[english]{babel}
\usepackage[utf8x]{inputenc}
\usepackage[T1]{fontenc}

\usepackage[top=3cm,bottom=2cm,left=3cm,right=3cm,marginparwidth=1.75cm]{geometry}

\usepackage{amsmath}
\usepackage{graphicx}

\usepackage{hyperref}
\hypersetup{colorlinks,allcolors=blue}

\usepackage{xypic}
\usepackage{tabu}
\usepackage{caption}
\usepackage{comment}
\usepackage{stmaryrd} 
\usepackage{ytableau} 
\usepackage{xcolor}
\usepackage{float} 
\restylefloat{table}

\usepackage[new]{old-arrows}

\usepackage{pgf,tikz,pgfplots} 
\pgfplotsset{compat=1.14}
\usetikzlibrary{arrows}
\usepackage{amssymb}
\usetikzlibrary{cd}
\usepackage{relsize} 
\usetikzlibrary{fit}
\newcommand\addvmargin[1]{
  \node[fit=(current bounding box),inner ysep=#1,inner xsep=0]{};
}
\definecolor{mygray}{gray}{0.8}
\definecolor{mycyan}{rgb}{0.4, 0.5, 0.9}
\newcommand\comm[1]{\textbf{\textcolor{blue}{[#1]}}}
\newcommand\commM[1]{\textbf{\textcolor{orange}{[#1]}}}
\newcommand\commAN[1]{\textbf{\textcolor{red}{[#1]}}}

\newcommand{\norm}[1]{\left\lVert#1\right\rVert}
\newcommand{\be}{\begin{equation}}
\newcommand{\ee}{\end{equation}}

\newcommand{\R}{\mathbb{R}}

\usepackage{mathtools} 
\mathtoolsset{showonlyrefs,showmanualtags}
\numberwithin{equation}{section}
\usepackage{amsmath}

\newtheorem{theorem}{Theorem}
\newtheorem{lemma}[theorem]{Lemma}
\newtheorem{proposition}[theorem]{Proposition}
\newtheorem{corollary}[theorem]{Corollary}

\theoremstyle{definition}
\newtheorem{definition}[theorem]{Definition}

\newtheorem{example}[theorem]{Example}
\newtheorem*{warning*}{Warning!}

\theoremstyle{remark} 
\newtheorem{remark}[theorem]{Remark}

\title{Graph invariants from the topology of rigid isotopy classes}
\author{Mara Belotti}
\address{SISSA (Trieste)}
\email{marabelotti96@gmail.com}
\author{Antonio Lerario}
\date{}
\address{SISSA (Trieste)}
\email{lerario@sissa.it}
\author{Andrew Newman}\thanks{A.N. was supported by Deutsche Forschungsgemeinschaft (DFG, German Research 
Foundation) Graduiertenkolleg 2434 "Facets of Complexity".}
\date{}
\address{A.N. Technische Universit\"at Berlin, Chair of Discrete Mathematics/Geometry, Strasse des 17. Juni 136, 10623 Berlin}
\email{newman@math.tu-berlin.de}

\begin{document}

\makeatletter
\tagsleft@false
\makeatother

\maketitle 

\begin{abstract}We define a new family of graph invariants, studying the topology of the moduli space of their geometric realizations in Euclidean spaces, using a limiting procedure reminiscent of Floer homology. 


Given a labeled graph $G$ on $n$ vertices and $d \geq 1$, $W_{G, d} \subseteq \R^{d \times n}$ denotes the space of nondegenerate realizations of $G$ in $\R^d$. For example if $G$ is the empty graph then $W_{G, d}$ is homotopy equivalent to the configuration space of $n$ points in $\R^d$. Questions about when a certain graph $G$ exists as a geometric in $\R^d$ have been considered in the literature and in our notation have to do with deciding when $W_{G, d}$ is nonempty. However $W_{G, d}$ need not be connected, even when it is nonempty, and we refer to the connected components of $W_{G, d}$ as \emph{rigid isotopy classes} of $G$ in $\R^d$. We study the topology of these rigid isotopy classes. First, regarding the connectivity of $W_{G, d}$, we generalize a result of Maehara that $W_{G, d}$ is nonempty for $d \geq n$ to show that $W_{G, d}$ is $k$-connected for $d \geq n + k + 1$, and so $W_{G, \infty}$ is always contractible.

While $\pi_k(W_{G, d}) = 0$ for $G$, $k$ fixed and $d$ large enough, we also prove that, in spite of this, when $d\to \infty$ the structure of the nonvanishing homology of $W_{G, d}$ exhibits a stabilization phenomenon. The nonzero part of its homology is concentrated in at most $(n-1)$-many equally spaced clusters in degrees between $d-n$ and $(n-1)(d-1)$, and whose structure does not depend on $d$, for $d$ large enough. This leads to the definition of a family of graph invariants, capturing the asymptotic structure of the homology of the rigid isotopy class. For instance, the sum of the Betti numbers of $W_{G,d}$ does not depend on $d$, for $d$ large enough; we call this number the \emph{Floer number} of the graph $G$. This terminology comes by analogy with Floer theory, because of the shifting phenomenon in the degrees of positive Betti numbers of $W_{G, d}$ as $d$ tends to infinity.

Finally, we give asymptotic estimates on the number of rigid isotopy classes of $\R^d$--geometric graphs on $n$ vertices for $d$ fixed and $n$ tending to infinity. When $d=1$ we show that asymptotically as $n\to \infty$ each isomorphism class corresponds to a constant number of rigid isotopy classes, on average. For  $d>1$ we prove a similar statement at the logarithmic scale.
    \end{abstract}
\input{Introduction}

\input{Preliminaries}
\input{Increasing_d}

\input{Increasing_n}

\input{Examples}

\bibliographystyle{alpha}
\bibliography{GG.bib}
\end{document}

%% file: Introduction.tex
\section{Introduction}

Let $P=(p_1, \ldots, p_n)$ be a point in $\R^{d\times n}.$ The geometric graph associated to $P$ is the labeled graph\footnote{From now on, unless differently specified, the word ``graph'' stands for ``labeled graph''.
} $G(P)$ whose vertices and edges are, respectively:  
\be\ V(G(P))=\{(1, p_1), \ldots, (n, p_n)\}\quad\textrm{and}\quad E(G(P))= \{\textrm{$((i, p_i), (j, p_j)) \, |\, i<j,\, \|p_i-p_j\|^2<1$}\}.\ee 


If a graph $G$ on $n$ vertices is isomorphic to a geometric graph $G(P)$, as a labeled graph, for some $P\in \R^{d\times n}$ we say it is realizable as an $\R^d$--geometric graph on $n$ vertices. It was proved by Maehara in \cite{Maehara} that when $d\geq n$ every graph on $n$ vertices is realizable as an $\R^d$-geometric graph. In particular, if we denote by $\#_{d,n}$ the number of isomorphims classes of labeled $\R^d$--geometric graphs on $n$ vertices, then for $d\geq n$ we have
\be\label{eq:Maheara} \#_{d,n}=2^{\binom{n}{2}}.\ee

This statement can be rephrased using the theory of \textit{discriminants} from real algebraic geometry. To explain this idea let us first introduce the notion of nondegenerate geometric graph: the $\R^d$--geometric graph $G(P)$ is called \textit{nondegenerate} if there is no pair of indices $1\leq i<j\leq n$ such that $\|p_i-p_j\|^2=1.$ Studying nondegenerate graphs is not an actual restriction, since the set of isomorphism classes of labeled nondegenerate $\R^d$--geometric graphs coincides with the set of all possible isomorphism classes of labeled $\R^d$--geometric graphs (see Lemma \ref{lem:broad} below). Moreover, nondegenerate geometric graphs are simpler to study, because of their stability under small perturbations of the defining points.

In this setting the discriminant consists of the set of \textit{degenerate} $\R^d$--geometric graphs:
\be \Delta_{d, n}=\{P\in \R^{d\times n}\,|\, \textrm{there exist $1\leq i<j\leq n$ such that $\|p_i-p_j\|^2=1$}\}\subset \R^{d\times n}.\ee
This discriminant partitions $\R^{d\times n} \setminus \Delta_{d, n}$ into many disjoint, connected open sets, which we will call \textit{chambers}. If two points $P_0$ and $P_1$ belong to the same chamber in $\R^{d \times n} \setminus \Delta_{d, n}$ then clearly $G(P_0)$ and $G(P_1)$ are isomorphic, but the reverse implication does not hold in general, leading to the following definition.
\begin{definition}If two points $P_0, P_1\in \R^{d\times n}\backslash \Delta_{d,n}$ belong to the same chamber, that is if there is a continuous curve $P:[0, 1]\to \R^{d\times n}\backslash \Delta_{d, n}$ with $P(0)=P_0$ and $P(1)=P_1$, we will say that the geometric graphs $G(P_0)$ and $G(P_1)$ are \emph{rigidly isotopic}. 
\end{definition}
As an example of $\R^d$-geometric graphs which are isomorphic but not rigidly isotopic, consider $P_0=(-2, 0)$ and $P_1=(0, -2)$: the $\R$--geometric graphs $G(P_0)$ and $G(P_1)$ are isomorphic; they are both the graph on $2$ vertices with no edges, but they are not rigid isotopic since any curve $P(t)\in \R^{1\times 2}$ with $P(0)=P_0$ and $P(1)=P_1$ must intersect the discriminant. Another example is depicted in Figure \ref{fig:exrig}.

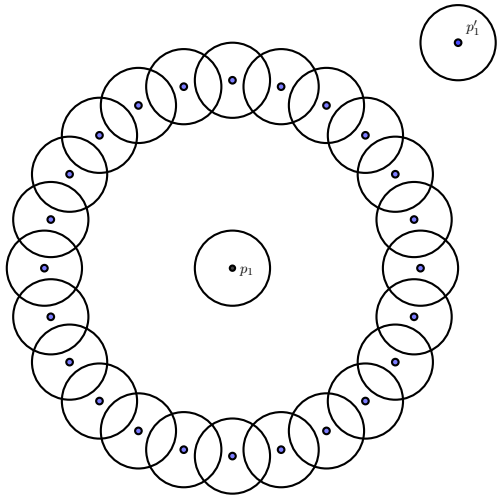
\begin{figure}
\centering
\captionsetup{margin=1.5cm}
\definecolor{ududff}{rgb}{0.30196078431372547,0.30196078431372547,1.}
\definecolor{uuuuuu}{rgb}{0.26666666666666666,0.26666666666666666,0.26666666666666666}
\definecolor{xdxdff}{rgb}{0.49019607843137253,0.49019607843137253,1.}
\begin{tikzpicture}[thick,scale=0.5, every node/.style={scale=0.5}][line cap=round,line join=round,>=triangle 45,x=1.0cm,y=1.0cm]
\clip(-7.,-7.) rectangle (7.5,7.5);
\draw [line width=0.8pt] (0.,5.) circle (1.cm);
\draw [line width=0.8pt] (1.2940952255126077,4.829629131445338) circle (1.cm);
\draw [line width=0.8pt] (2.5,4.330127018922189) circle (1.cm);
\draw [line width=0.8pt] (3.5355339059327386,3.535533905932733) circle (1.cm);
\draw [line width=0.8pt] (4.330127018922193,2.5) circle (1.cm);
\draw [line width=0.8pt] (4.829629131445341,1.2940952255126) circle (1.cm);
\draw [line width=0.8pt] (5.,0.) circle (1.cm);
\draw [line width=0.8pt] (4.829629131445339,-1.2940952255126066) circle (1.cm);
\draw [line width=0.8pt] (4.33012701892219,-2.5) circle (1.cm);
\draw [line width=0.8pt] (3.535533905932734,-3.535533905932738) circle (1.cm);
\draw [line width=0.8pt] (2.5,-4.330127018922193) circle (1.cm);
\draw [line width=0.8pt] (1.294095225512601,-4.829629131445341) circle (1.cm);
\draw [line width=0.8pt] (0.,-5.) circle (1.cm);
\draw [line width=0.8pt] (-1.2940952255126055,-4.82962913144534) circle (1.cm);
\draw [line width=0.8pt] (-2.5,-4.330127018922191) circle (1.cm);
\draw [line width=0.8pt] (-3.5355339059327378,-3.5355339059327355) circle (1.cm);
\draw [line width=0.8pt] (-4.330127018922193,-2.5) circle (1.cm);
\draw [line width=0.8pt] (-4.829629131445341,-1.2940952255126026) circle (1.cm);
\draw [line width=0.8pt] (-5.,0.) circle (1.cm);
\draw [line width=0.8pt] (-4.829629131445341,1.2940952255126041) circle (1.cm);
\draw [line width=0.8pt] (-4.330127018922193,2.5) circle (1.cm);
\draw [line width=0.8pt] (-3.5355339059327373,3.5355339059327373) circle (1.cm);
\draw [line width=0.8pt] (-2.5,4.330127018922193) circle (1.cm);
\draw [line width=0.8pt] (-1.2940952255126037,4.8296291314453415) circle (1.cm);
\draw [line width=0.8pt] (0.,0.) circle (1.cm);
\draw [line width=0.8pt] (6.,6.) circle (1.cm);
\draw (0.06,0.18) node[anchor=north west] {$p_1$};
\draw (6.08,6.74) node[anchor=north west] {$p_1'$};
\begin{scriptsize}
\draw [fill=xdxdff] (0.,5.) circle (2.5pt);
\draw [fill=uuuuuu] (0.,0.) circle (2.0pt);
\draw [fill=xdxdff] (-1.2940952255126037,4.8296291314453415) circle (2.5pt);
\draw [fill=xdxdff] (-2.5,4.330127018922193) circle (2.5pt);
\draw [fill=xdxdff] (-3.5355339059327373,3.5355339059327373) circle (2.5pt);
\draw [fill=xdxdff] (-4.330127018922193,2.5) circle (2.5pt);
\draw [fill=xdxdff] (-4.829629131445341,1.2940952255126041) circle (2.5pt);
\draw [fill=xdxdff] (-5.,0.) circle (2.5pt);
\draw [fill=xdxdff] (-4.829629131445341,-1.2940952255126026) circle (2.5pt);
\draw [fill=xdxdff] (-4.330127018922193,-2.5) circle (2.5pt);
\draw [fill=xdxdff] (-3.5355339059327378,-3.5355339059327355) circle (2.5pt);
\draw [fill=xdxdff] (-2.5,-4.330127018922191) circle (2.5pt);
\draw [fill=xdxdff] (-1.2940952255126055,-4.82962913144534) circle (2.5pt);
\draw [fill=xdxdff] (0.,-5.) circle (2.5pt);
\draw [fill=xdxdff] (1.294095225512601,-4.829629131445341) circle (2.5pt);
\draw [fill=xdxdff] (2.5,-4.330127018922193) circle (2.5pt);
\draw [fill=xdxdff] (3.535533905932734,-3.535533905932738) circle (2.5pt);
\draw [fill=xdxdff] (4.33012701892219,-2.5) circle (2.5pt);
\draw [fill=xdxdff] (4.829629131445339,-1.2940952255126066) circle (2.5pt);
\draw [fill=xdxdff] (5.,0.) circle (2.5pt);
\draw [fill=xdxdff] (4.829629131445341,1.2940952255126) circle (2.5pt);
\draw [fill=xdxdff] (4.330127018922193,2.5) circle (2.5pt);
\draw [fill=xdxdff] (3.5355339059327386,3.535533905932733) circle (2.5pt);
\draw [fill=xdxdff] (2.5,4.330127018922189) circle (2.5pt);
\draw [fill=xdxdff] (1.2940952255126077,4.829629131445338) circle (2.5pt);
\draw [fill=ududff] (6.,6.) circle (2.5pt);
\end{scriptsize}

\end{tikzpicture}
     \caption{Here we are drawing points in $\R^2$ together with the circles centered at those points and with radius $\frac{1}{\sqrt{2}}$. Now, let us define $P$ and $P'$ points in $\R^{2\times 25}$ in such a way that $p_1$ is the point inside the big circle and $p_1'$ is the point outside the big circle while $p_i=p_i'$ for $i>1$ and they are the points on the big circle. Then, the two geometric graphs $G(P)$ and $G(P')$ are isomorphic but not rigidly isotopic. }
\label{fig:exrig}
\end{figure}

For $n$ and $d$ the number of rigid isotopy classes of geometric graphs on $n$ vertices in $\R^d$ is exactly given by $b_0(\R^{d \times n} \setminus \Delta_{d, n})$, and we always clearly have 
\[b_0(\R^{d \times n} \setminus \Delta_{d, n}) \geq \#_{d, n}.\]
One natural question therefore is for what values of $n$ and $d$ do the two notions coincide. Moreover, one could consider higher-dimensional notions of connectivity of $\R^{d \times n} \setminus \Delta_{d, n}$ and study the higher homology and the homotopy groups of the space of its connected components. As we will see, this study will lead us to a definition of a new graph invariant, which reminds of Floer homology, as well as precise asymptotics for the enumeration of rigid isotopy and isomorphism classes of geometric graphs.



\subsection{The case $d\to \infty$} As we will prove in Corollary \ref{cor:conn1} below, for $d\geq {n+1}$, the two notions of isomorphic and rigidly isotopic coincide and $b_0(\R^{d\times n}\backslash \Delta_{d, n})=\#_{d, n}$. Therefore, adopting this language we can reformulate\footnote{Here and below, for a topological space $X$ we will denote by $b_k(X)=\mathrm{dim}_{\mathbb{Z}_2}(H^k(X; \mathbb{Z}_2))$ its $k$--th Betti number and by $b(X)=\sum_{k=0}^{\infty} b_k(X)$, its total Betti number, whenever these numbers are defined. This will happen for all the spaces that we will consider in this paper: they will all be homotopy equivalent to finite CW--complexes.} the identity in \eqref{eq:Maheara} as:
\be b_0(\R^{d\times n}\backslash \Delta_{d, n})=2^{\binom{n}{2}},\ee
which is true for $d\geq {n+1}.$ The realizability result of Maehara \cite{Maehara} and the fact that for large $d$ ``rigid isotopy'' and ``isomorphism'' are the same notion, seem to settle all relevant questions related to the study of the asymptotics for the number of chambers of $b_0(\R^{d \times n} \setminus \Delta_{d, n})$ for fixed $n$ and large $d$. However, as we will see, the topology of the chambers of the complement of the discriminant is extremely rich and some unexpected structure emerges as $d\to \infty.$

In order to explain this phenomenon, let us label the chambers of $\R^{d\times n}\backslash \Delta_{d,n}$ with the corresponding isomorphism class of labeled geometric graphs: given a graph $G$ on $n$ vertices we define
\be W_{G, d}=\{P\in \R^{d\times n}\backslash \Delta_{d,n}\,|\, G(P)\cong G\}\subset \R^{d\times n}.\ee
In other words, $W_{G,d}$ consists of all the points $P\in \R^{d\times n}$ not on the discriminant whose associated geometric graph is isomorphic to $G$. For small $d$ this set could be a union of several chambers, but for large $d$ it is an actual chamber, that is a connected open set. This can be rephrased by saying that for every graph $G$ on $n$ vertices and for large enough $d$, the homotopy group $\pi_0(W_{G,d})$ consists of a single element. In fact, as we will show, the same statement is true for all the homotopy groups, once the group is fixed and $d$ becomes large enough.

\begin{theorem}\label{thm:zerohomotopy}For every $k\geq 0$ and for $d\geq k+n+1$ we have $\pi_k(W_{G,d})=0.$ 
\end{theorem}

Theorem \ref{thm:zerohomotopy} in fact generalizes the result of Maehara \cite{Maehara}. Taking the standard convention that a topological space is said to be $(-1)$-connected provided it is nonempty, Theorem \ref{thm:zerohomotopy} for $k = -1$ is Maehara's result that every graph on $n$ vertices can be realized as a geometric graph in $\R^d$, for $d \geq n$.

Notice that there is a natural sequence of inclusions:
\be\label{eq:inclusionsWg}\cdots \longhookrightarrow W_{G, d}\longhookrightarrow W_{G, d+1}\longhookrightarrow\cdots \ee
obtained by simply including $\R^{d\times n}$ into $\R^{(d+1)\times n}$ by appending a list of zeroes to the coordinates of $P$. 
The proof of the Theorem \ref{thm:zerohomotopy} goes through two intermediate steps, which are of independent interest: we first prove that for every $k\geq 0$ and for $d\geq k+n+1$ the inclusion $W_{G, d}\longhookrightarrow W_{G, d+1}$ induces an injection on the homotopy classes of maps, and then we prove that the inclusion $W_{G, d}\longhookrightarrow W_{G, d+n}$ is homotopic to a constant map.

\begin{example}[Homotopy groups of the configuration space of $n$ points in $\R^d$] \label{ex:confhomology} Let us consider the graph $G$ consisting of $n$ vertices and no edges. It is easy to see that the corresponding chamber $W_{G, d}$ is homotopy equivalent\footnote{Here and below, for two topological spaces $X$ and $Y$ we use the symbol $X\simeq Y$ to denote that they are homeomorphic and $X\sim Y$ to denote that they are homotopy equivalent.}  to the configuration space of $n$ distinct points in $\R^d$:
\be W_{G,d}\sim \mathrm{Conf}_n(\R^d).\ee 
In this case one can compute exactly the homotopy groups of $W_{G,d}$: for every $k\geq 0$ and for $d\geq 3$ we have (see \cite[Chapter 2, Theorem 1.1]{FadellHusseini})
\be \pi_k(W_{G, d})\simeq \pi_k(\mathrm{Conf}_n(\R^d))\simeq \bigoplus_{j=1}^{n-1}\pi_k\bigg(\underbrace{S^{d-1}\vee\cdots\vee S^{d-1}}_{\textrm{bouquet of $j$ spheres}}\bigg).\ee
Since $\pi_k(S^{d-1}\vee\cdots\vee S^{d-1})=0$ for $d\geq k+2$, in this case we immediately see that also $\pi_k(W_{G, d})=0$ for $d\geq k+2.$
\end{example}

It is natural at this point to put the sequence of inclusions \eqref{eq:inclusionsWg} into the infinite dimensional space $\R^{\infty\times n}$ of $n$--tuples of sequences $(p_1, \ldots, p_n)$ such that for every $j=1, \ldots, n$ all but finitely many elements in the sequence $p_j$ are zero:
\be \R^{\infty\times n}=\varinjlim \R^{d\times n}.\ee
The definition of geometric graph and discriminant also makes sense in this infinite dimensional space, see Section \ref{sec:infinite}. The chambers, are now defined as follows: for a given graph $G$ on $n$ vertices, we set
\be W_{G, \infty}=\left\{P=(p_1, \ldots, p_n)\in \R^{\infty\times n}\backslash \Delta_{\infty, n}\,\big|\, G(P)\cong G\right\}.\ee
From Theorem \ref{thm:zerohomotopy} we deduce the following.
\begin{theorem}\label{thm:contractible} For every graph $G$, the set $W_{G, \infty}=\varinjlim W_{G, d}$ is contractible.
\end{theorem}
\subsection{Floer homology of a graph}Summarizing the picture so far: as $d\to \infty$ each $W_{G, d}$ becomes eventually $k$--connected and its direct limit $W_{G, \infty}$ has no homotopy. Moreover, by Hurewicz Theorem, each fixed reduced Betti number of $W_{G,d}$ vanishes for $d$ large enough: more precisely, for every $k>0$ there exists $d(k)>0$ such that
\be b_k(W_{G, d})=0\quad \forall d\geq d(k).\ee
But this is not the whole story. Before we continue let us discuss one more, likely familiar, example.
\begin{example}[The infinite dimensional sphere]\label{ex:sphere}Let $G$ be the graph consisting of two disjoint points. Then $W_{G, d}\simeq \R^d\times S^{d-1}\sim S^{d-1}$ and $W_{G, \infty}\simeq S^{\infty}\times \R^{d\times (n-1)}\sim S^{\infty}.$ In this case \eqref{eq:inclusionsWg} become:
\be \cdots \longhookrightarrow S^{d-1}\longhookrightarrow S^{d}\longhookrightarrow \cdots \hookrightarrow S^{\infty}=\varinjlim S^d.\ee
If now we look at the Betti numbers of $S^{d}$ we see that there is a hole in dimension $d$ that moves to infinity as $d\to \infty$ and it disappears when $d=\infty$. The sphere $S^{\infty}$ has no cohomology except in dimension zero, but still it has cohomology every time we cut it with a finite--dimensional space.
\end{example}

The phenomenon described in Example \ref{ex:sphere} can be interpreted using some extraordinary cohomology theory, in the context of the Leray--Schauder degree and, more generally, of Floer homology theories (see \cite{szulkin, gebagranas, abbondandolo}). This behaviour has also been observed for non--holonomic loop spaces in Carnot groups, see \cite{AgrachevGentileLerario}. In all these examples we are dealing with a sequence of spaces $X_d$ whose direct limit $X_\infty$ is contractible, but for every $d$ large enough each space carries the same amount of cohomology, just shifted in its dimension. A more general family of examples where this occurs is the iterated suspension.

\begin{example}[The iterated suspension]Let $X_0$ be a CW-complex and define $X_d=S X_{d-1}$, where $S X$ is the suspension of $X$:
\be SX=(X\times I)/\sim,\ee 
and the equivalence relation ``$\sim$'' is given by: $(x_1, 0)\sim (x_2, 0)$ and $(x_1, 1)\sim (x_2, 1)$ for all $x_1, x_2\in X$. We have a natural sequence of inclusions
\be \label{eq:suspension}\cdots \longhookrightarrow X_d\longhookrightarrow X_{d+1}\longhookrightarrow \cdots \ee
given by mapping $X_d\to SX_{d}$ homeomorphically to $X_d\times \{\frac{1}{2}\}.$ We denote by $X_\infty=\varinjlim X_d$ the direct limit of the sequence of inclusions \eqref{eq:suspension}. If $X_0=\{x_1, x_2\}$, then $X_d=S^d$ and $X_\infty=S^{\infty}.$ For every $k$ the space $X_d$ becomes eventually $k$--connected when $d$ is large enough, and $X_\infty$ is contractible. If one looks at the homology of $X_d$, this is made by a cluster of holes that shifts to infinity as $d\to \infty$. This can be expressed, for example, by looking at the Poincar\'e polynomial of $X_d$:
\be P_{X_d}(t)=1+t^d(P_{X_0}(t)-1).\ee
These holes are not present when $d=\infty$, but the sum of the Betti numbers of $X_d$ is constant: 
\be b(X_d)=P_{X_d}(1)\equiv P_{X_0}(1)= b(X_0).\ee
\end{example}

We will prove that a similar phenomenon happens, for all the spaces $W_{G, d}$: their reduced cohomology is made of ``clusters of holes'' that ``shift'' to infinity as $d\to \infty$, see Figure \ref{fig:chamberfloer}. In fact we show also that for $G$ on $n$ vertices there are at most $n-1$ such clusters. More precisely, we have the following result.

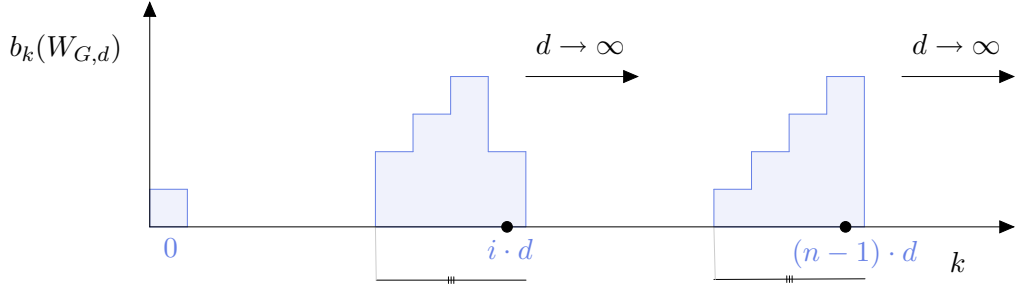
\begin{figure}
\centering
\begin{tikzpicture}[line cap=round,line join=round,>=triangle 45,x=1.0cm,y=1.0cm]
\clip(-2.1,-1.) rectangle (12.,3.);
\fill[line width=0.3pt,color=mycyan,fill=mycyan,fill opacity=0.10000000149011612] (0.,0.) -- (0.,0.5) -- (0.5,0.5) -- (0.5,0.) -- cycle;
\fill[line width=0.3pt,color=mycyan,fill=mycyan,fill opacity=0.10000000149011612] (3.,0.) -- (3.,1.) -- (3.5,1.) -- (3.5,1.5) -- (4.,1.5) -- (4.,2.) -- (4.5,2.) -- (4.5,1.) -- (5.,1.) -- (5.,0.) -- cycle;
\fill[line width=0.3pt,color=mycyan,fill=mycyan,fill opacity=0.10000000149011612] (7.5,0.) -- (7.5,0.5) -- (8.,0.5) -- (8.,1.) -- (8.5,1.) -- (8.5,1.5) -- (9.,1.5) -- (9.,2.) -- (9.5,2.) -- (9.5,0.) -- cycle;
\draw [line width=0.3pt,color=mycyan] (0.,0.)-- (0.,0.5);
\draw [line width=0.3pt,color=mycyan] (0.,0.5)-- (0.5,0.5);
\draw [line width=0.3pt,color=mycyan] (0.5,0.5)-- (0.5,0.);
\draw [line width=0.3pt,color=mycyan] (0.5,0.)-- (0.,0.);
\draw [line width=0.3pt,color=mycyan] (3.,0.)-- (3.,1.);
\draw [line width=0.3pt,color=mycyan] (3.,1.)-- (3.5,1.);
\draw [line width=0.3pt,color=mycyan] (3.5,1.)-- (3.5,1.5);
\draw [line width=0.3pt,color=mycyan] (3.5,1.5)-- (4.,1.5);
\draw [line width=0.3pt,color=mycyan] (4.,1.5)-- (4.,2.);
\draw [line width=0.3pt,color=mycyan] (4.,2.)-- (4.5,2.);
\draw [line width=0.3pt,color=mycyan] (4.5,2.)-- (4.5,1.);
\draw [line width=0.3pt,color=mycyan] (4.5,1.)-- (5.,1.);
\draw [line width=0.3pt,color=mycyan] (5.,1.)-- (5.,0.);
\draw [line width=0.3pt,color=mycyan] (5.,0.)-- (3.,0.);
\draw [line width=0.3pt,color=mycyan] (7.5,0.)-- (7.5,0.5);
\draw [line width=0.3pt,color=mycyan] (7.5,0.5)-- (8.,0.5);
\draw [line width=0.3pt,color=mycyan] (8.,0.5)-- (8.,1.);
\draw [line width=0.3pt,color=mycyan] (8.,1.)-- (8.5,1.);
\draw [line width=0.3pt,color=mycyan] (8.5,1.)-- (8.5,1.5);
\draw [line width=0.3pt,color=mycyan] (8.5,1.5)-- (9.,1.5);
\draw [line width=0.3pt,color=mycyan] (9.,1.5)-- (9.,2.);
\draw [line width=0.3pt,color=mycyan] (9.,2.)-- (9.5,2.);
\draw [line width=0.3pt,color=mycyan] (9.5,2.)-- (9.5,0.);
\draw [line width=0.3pt,color=mycyan] (9.5,0.)-- (7.5,0.);
\draw (0.05,-0.025) node[anchor=north west] {\textcolor{mycyan}{$0$}};
\draw (4.35,-0.025) node[anchor=north west] {\textcolor{mycyan}{$i\cdot d$}};
\draw (8.4,-0.025) node[anchor=north west] {\textcolor{mycyan}{$(n-1)\cdot d$}};
\draw (5.0,2.7) node[anchor=north west] {$d\to\infty$};
\draw (10.0,2.7) node[anchor=north west] {$d\to\infty$};
\draw [line width=0.3pt] (3.0111111111111115,-0.70888888888889)-- (5.0022222222222235,-0.70888888888889);
\draw [line width=0.3pt] (3.971111111111112,-0.6644444444444456) -- (3.971111111111112,-0.7533333333333345);
\draw [line width=0.3pt] (4.006666666666668,-0.6644444444444456) -- (4.006666666666668,-0.7533333333333345);
\draw [line width=0.3pt] (4.042222222222223,-0.6644444444444456) -- (4.042222222222223,-0.7533333333333345);
\draw [line width=0.3pt] (7.5,-0.7)-- (9.50888888888889,-0.6911111111111122);
\draw [line width=0.3pt] (8.468692582040255,-0.6512688701131852) -- (8.469085891857684,-0.7401568888518706);
\draw [line width=0.3pt] (8.504247789535729,-0.651111546186214) -- (8.504641099353158,-0.7399995649248994);
\draw [line width=0.3pt] (8.539802997031204,-0.6509542222592428) -- (8.540196306848634,-0.7398422409979282);
\draw [line width=0.3pt,color=mygray] (3.,0.)-- (3.0111111111111115,-0.70888888888889);
\draw [line width=0.3pt,color=mygray] (7.5,0.)-- (7.517816760316126,-0.6999211647773622);
\draw [->,line width=0.3pt] (0.,0.) -- (0.,3.);
\draw [->,line width=0.3pt] (0.,0.) -- (11.5,0.);
\draw (-2.,2.7) node[anchor=north west] {$b_k(W_{G,d})$};
\draw [->,line width=0.3pt] (5.,2.) -- (6.5,2.);
\draw [->,line width=0.3pt] (10.,2.) -- (11.5,2.);
\draw (10.5,-0.2) node[anchor=north west] {$k$};
\begin{scriptsize}
\draw [fill=black] (4.75,0.) circle (2.0pt);
\draw [fill=black] (9.25,0.) circle (2.0pt);
\end{scriptsize}
\end{tikzpicture}
	\caption{A plot of the Betti numbers of $W_{G, d}$. The width of each non zero cluster of holes is $\binom{n}{2}+1$, which is a constant. Each of these clusters is placed at a multiple of $d$ and as $d\to \infty$ they shift to infinity. The total Betti number of $W_{G, d}$, i.e. the blue area, becomes constant for $d$ large enough.} \label{fig:chamberfloer}
\end{figure}

\begin{theorem}\label{theorem:floer}For every graph $G$ on $n$ vertices there exist polynomials\footnote{These polynomials only depend on $G$ and not on $d$.} $Q_{G,1}, \ldots, Q_{G,{n-1}}$ each of degree at most $\binom{n}{2}+1$ such that for $d\geq {n\choose 2}+2$  the Poincar\'e polynomial of $W_{G,d}$ is 
\be P_{W_{G, d}}(t)=1+t^{d-{n\choose 2}-1}Q_{G, 1}(t)+\cdots +t^{md-{n\choose 2}-1}Q_{G, m}(t)+\cdots+t^{(n-1)d-{n\choose2 }-1}Q_{G, n-1}(t).\ee
In particular, there exists $\beta(G)>0$ such that
\be b(W_{G, d})=P_{W_{G, d}}(1)\equiv 1+\sum_{\ell=1}^{n-1}Q_{G, \ell}(1)=\beta(G)\ee for $d$ large enough (i.e. the sum of the Betti numbers of $W_{G, d}$ becomes a constant, which depends on $G$ only).
\end{theorem}

Each polynomial $Q_{G, m}$ corresponds to one of the clusters above and keeps track of the Betti numbers $b_k(W_{G, d})$ with $0\leq dm-k\leq {n\choose 2}$, i.e. with index $k$ located ``near'' $d m$ (remember that ${n\choose 2}$ is a constant in this asymptotic regime). These clusters are the ``Floer homologies'' of the graph.

The proof of Theorem \ref{theorem:floer} uses a spectral sequence argument: each $W_{G,d}$ can be described as a system of quadratic inequalities and one can use the technique developed in \cite{Agrachev, AgrachevLerario} for the study of its Betti numbers. In this case, as $d\to \infty$, the spectral sequence that we need to consider converges at the second  step (i.e. $E_2=E_\infty$). One can prove that both $E_2 $ and the second differential $d_2$ have an asymptotic stable shape: as $d\to \infty$ the nonzero part of the spectral sequence looks like a skinny table where there is no room for higher differentials and the nonzero elements of this table are located near the rows which are labeled by indices which are multiples of $d$ (see Figure \ref{fig:spectral} below). 

An interesting question arising from Theorem \ref{theorem:floer} is, for a given graph $G$ on $n$ vertices, to compute the number
\be \beta(G)=\lim_{d\to\infty}b(W_{G, d}).\ee
We call this number the \emph{Floer number} of the graph $G$. This number is a graph invariant, as well as the polynomials from Theorem \ref{theorem:floer}. Table \ref{tbl:PoincarePolynomials} shows the value of this number for all the possible graphs on $4$ vertices, but the general case is still mysterious. There is more discussion about some of the details of Table \ref{tbl:PoincarePolynomials} in Example \ref{ex:4vertices}.
\begin{table}
\begin{tabular}{c|c|c|c}
Graph & Poincar\'{e} polynomial & $\beta(G)$ &Labeled copies \\ \hline 
\centering
\begin{tikzpicture}[baseline=0, scale = 0.2]
\draw(-1, -1) node[circle, draw, fill=black, inner sep=0pt, minimum width=2pt]{};
\draw(1, -1) node[circle, draw, fill=black, inner sep=0pt, minimum width=2pt]{};
\draw(1, 1) node[circle, draw, fill=black, inner sep=0pt, minimum width=2pt]{};
\draw(-1, 1) node[circle, draw, fill=black, inner sep=0pt, minimum width=2pt]{};
\addvmargin{1mm};
\end{tikzpicture}
& $1 + 6t^{d -1} + 11t^{2d - 2} + 6t^{3d - 3}$ &  24&1 \\ \hline
\centering

\begin{tikzpicture}[baseline=0, scale = 0.2]
\draw(-1, -1) node[circle, draw, fill=black, inner sep=0pt, minimum width=2pt]{};
\draw(1, -1) node[circle, draw, fill=black, inner sep=0pt, minimum width=2pt]{};
\draw(1, 1) node[circle, draw, fill=black, inner sep=0pt, minimum width=2pt]{};
\draw(-1, 1) node[circle, draw, fill=black, inner sep=0pt, minimum width=2pt]{};
\draw(-1, -1) -- (-1, 1);
\addvmargin{1mm};
\end{tikzpicture} & $1 + 3t^{d -1} + 2t^{2d - 2}$ & 6& 6 \\ \hline

\begin{tikzpicture}[baseline=0, scale = 0.2]
\draw(-1, -1) node[circle, draw, fill=black, inner sep=0pt, minimum width=2pt]{};
\draw(1, -1) node[circle, draw, fill=black, inner sep=0pt, minimum width=2pt]{};
\draw(1, 1) node[circle, draw, fill=black, inner sep=0pt, minimum width=2pt]{};
\draw(-1, 1) node[circle, draw, fill=black, inner sep=0pt, minimum width=2pt]{};
\draw(-1, -1) -- (-1, 1);
\draw(1, -1) -- (1, 1);
\addvmargin{1mm};
\end{tikzpicture} & $1 + t^{d -1}$ & 2&3 \\ \hline

\begin{tikzpicture}[baseline=0, scale = 0.2]
\draw(-1, -1) node[circle, draw, fill=black, inner sep=0pt, minimum width=2pt]{};
\draw(1, -1) node[circle, draw, fill=black, inner sep=0pt, minimum width=2pt]{};
\draw(1, 1) node[circle, draw, fill=black, inner sep=0pt, minimum width=2pt]{};
\draw(-1, 1) node[circle, draw, fill=black, inner sep=0pt, minimum width=2pt]{};
\draw(-1, -1) -- (-1, 1);
\draw(-1, 1) -- (1, 1);
\addvmargin{1mm};
\end{tikzpicture} & $1 + 2t^{d -1} + t^{2d - 2}$ & 4&12 \\ \hline

\begin{tikzpicture}[baseline=0, scale = 0.2]
\draw(-1, -1) node[circle, draw, fill=black, inner sep=0pt, minimum width=2pt]{};
\draw(1, -1) node[circle, draw, fill=black, inner sep=0pt, minimum width=2pt]{};
\draw(1, 1) node[circle, draw, fill=black, inner sep=0pt, minimum width=2pt]{};
\draw(-1, 1) node[circle, draw, fill=black, inner sep=0pt, minimum width=2pt]{};
\draw(-1, -1) -- (-1, 1);
\draw(1, -1) -- (1, 1);
\draw(-1, 1) -- (1, 1);
\addvmargin{1mm};
\end{tikzpicture} & $1 + t^{d -1}$ &2& 12 \\ \hline

\begin{tikzpicture}[baseline=0, scale = 0.2]
\draw(-1, -1) node[circle, draw, fill=black, inner sep=0pt, minimum width=2pt]{};
\draw(1, -1) node[circle, draw, fill=black, inner sep=0pt, minimum width=2pt]{};
\draw(1, 1) node[circle, draw, fill=black, inner sep=0pt, minimum width=2pt]{};
\draw(-1, 1) node[circle, draw, fill=black, inner sep=0pt, minimum width=2pt]{};
\draw(-1, -1) -- (-1, 1);
\draw(-1, 1) -- (1, -1);
\draw(-1, 1) -- (1, 1);
\addvmargin{1mm};
\end{tikzpicture} & $1 + t^{d - 2} + t^{d - 1} + t^{2d - 3}$ & 4 &4 \\ \hline

\begin{tikzpicture}[baseline=0, scale = 0.2]
\draw(-1, -1) node[circle, draw, fill=black, inner sep=0pt, minimum width=2pt]{};
\draw(1, -1) node[circle, draw, fill=black, inner sep=0pt, minimum width=2pt]{};
\draw(1, 1) node[circle, draw, fill=black, inner sep=0pt, minimum width=2pt]{};
\draw(-1, 1) node[circle, draw, fill=black, inner sep=0pt, minimum width=2pt]{};
\draw(-1, -1) -- (-1, 1);
\draw(-1, 1) -- (1, 1);
\draw(1, 1) -- (-1, -1);
\addvmargin{1mm};
\end{tikzpicture} & $1 + t^{d -1}$ & 2&4  \\ \hline

\begin{tikzpicture}[baseline=0, scale = 0.2]
\draw(-1, -1) node[circle, draw, fill=black, inner sep=0pt, minimum width=2pt]{};
\draw(1, -1) node[circle, draw, fill=black, inner sep=0pt, minimum width=2pt]{};
\draw(1, 1) node[circle, draw, fill=black, inner sep=0pt, minimum width=2pt]{};
\draw(-1, 1) node[circle, draw, fill=black, inner sep=0pt, minimum width=2pt]{};
\draw(-1, -1) -- (-1, 1);
\draw(-1, 1) -- (1, 1);
\draw(1, 1) -- (-1, -1);
\draw(1, 1) -- (1, -1);
\addvmargin{1mm};
\end{tikzpicture} & $1 + t^{d -1}$ & 2&12 \\ \hline

\begin{tikzpicture}[baseline=0, scale = 0.2]
\draw(-1, -1) node[circle, draw, fill=black, inner sep=0pt, minimum width=2pt]{};
\draw(1, -1) node[circle, draw, fill=black, inner sep=0pt, minimum width=2pt]{};
\draw(1, 1) node[circle, draw, fill=black, inner sep=0pt, minimum width=2pt]{};
\draw(-1, 1) node[circle, draw, fill=black, inner sep=0pt, minimum width=2pt]{};
\draw(-1, -1) -- (-1, 1);
\draw(1, -1) -- (1, 1);
\draw(-1, 1) -- (1, 1);
\draw(-1, -1) -- (1, -1);
\addvmargin{1mm};
\end{tikzpicture} & $1 + t^{d - 2} + t^{d - 1} + t^{2d - 3}$ &4& 3 \\ \hline

\begin{tikzpicture}[baseline=0, scale = 0.2]
\draw(-1, -1) node[circle, draw, fill=black, inner sep=0pt, minimum width=2pt]{};
\draw(1, -1) node[circle, draw, fill=black, inner sep=0pt, minimum width=2pt]{};
\draw(1, 1) node[circle, draw, fill=black, inner sep=0pt, minimum width=2pt]{};
\draw(-1, 1) node[circle, draw, fill=black, inner sep=0pt, minimum width=2pt]{};
\draw(-1, -1) -- (-1, 1);
\draw(1, -1) -- (1, 1);
\draw(-1, 1) -- (1, 1);
\draw(-1, -1) -- (1, -1);
\draw(1, 1) -- (-1, -1);
\addvmargin{1mm};
\end{tikzpicture} & $1 + t^{d - 1}$ &2& 6 \\ \hline

\begin{tikzpicture}[baseline=0, scale = 0.2]
\draw(-1, -1) node[circle, draw, fill=black, inner sep=0pt, minimum width=2pt]{};
\draw(1, -1) node[circle, draw, fill=black, inner sep=0pt, minimum width=2pt]{};
\draw(1, 1) node[circle, draw, fill=black, inner sep=0pt, minimum width=2pt]{};
\draw(-1, 1) node[circle, draw, fill=black, inner sep=0pt, minimum width=2pt]{};
\draw(-1, -1) -- (-1, 1);
\draw(1, -1) -- (1, 1);
\draw(-1, 1) -- (1, 1);
\draw(-1, -1) -- (1, -1);
\draw(1, 1) -- (-1, -1);
\draw(-1, 1) -- (1, -1);
\addvmargin{1mm};
\end{tikzpicture} & 1 & 1&1 \\ \hline
\end{tabular}
\caption{Poincar\'e polynomials and Floer numbers for $W_{G, d}$ for $G$ on 4 vertices. For every graph ``labeled copies'' refers to the number of isomorphism classes of labeled graphs with the same unlabeled graph.}\label{tbl:PoincarePolynomials}
\end{table}

We can make some observations in a few cases that suggest conjectures about $\beta(G)$ for general graphs. In the case of graphs on four or fewer vertices, we see that the Poincar\'e polynomial of $W_{G, d}$ has a general form with exponents given in terms of $d$. Once $d$ is large enough that $G$ has a realization as a geometric graph in $\R^d$ we see that the Poincar\'{e} polynomial is determined by the general form. From this we conjecture that for every graph $G$ there is a general form of the Poincar\'{e} polynomial of $W_{G, d}$ with exponents given in terms of $d$ that is valid as long as $d$ is large enough that $W_{G, d}$ is nonempty. This would imply that as soon as $G$ can be realized as a geometric graph in $\R^d$, $b(W_{G, d}) = \beta(G)$. In the case of graphs realizable in $\R$, we would have that $\beta(G)$ counts the number of chambers of $W_{G, 1}$. To see this recall that in the $d = 1$ case  $\R^{1 \times n} \setminus \Delta_{1, n}$ is a disjoint union of polyhedra so every chamber is contractible; therefore homology can only exists in degree zero.
A first step toward the proof of this conjecture would be to prove that all the differentials of the spectral sequence that we use in the proof of Theorem \ref{theorem:floer} are zero for all $d\geq n$.


\begin{example}[Betti numbers of the configuration space of $n$ points in $\R^d$]\label{ex:confbettinumbers} The Poincar\'e polynomial of $\mathrm{Conf}_n(\R^d)$ for $d\geq 1$ is given by (see \cite[Chapter V, Corollary 1.4]{FadellHusseini})
\be P_{\mathrm{Conf}_n(\R^d)}(t)=\prod_{j=1}^{n-1}(1+j t^{d-1}).\ee
Consequently $b(\mathrm{Conf}_n(\R^d))=P_{\mathrm{Conf}_n(\R^d)}(t)(1)\equiv n!.$ In other words, for the graph $G$ consisting of $n$ disjoint points we have $\beta(G)=n!$.
\end{example}

\subsection{The case $n\to \infty$}\label{section:largen} 
 Concerning the other asymptotic regime, the first case of interest is when $d=1$: here $\Delta_{1,n}$ is a hyperplane arrangement, since each quadric $\{|p_i-p_j|^2=1\}\subset \R^{1\times n}$ is the union of the two hyperplanes $\{p_i-p_j=1\}$ and $\{p_i-p_j=-1\}.$ It turns out that the number of chambers of the complement of such a hyperplane arrangement, that is the number of rigid isotopy classes of $\R$--geometric graphs on $n$ vertices, equals the number of \emph{labeled semiorders} of $[n]$. Using techniques from analytic combinatorics, we will prove the following theorem.

\begin{theorem}The number of rigid isotopy classes of $\R$--geometric graphs on $n$ vertices equals:
\be b_0\left(\R^{1\times n}\backslash \Delta_{1,n}\right)=\frac{1}{n}\cdot\sqrt{6\, \textrm{\normalfont{log}}\frac{4}{3}}\cdot \left(\frac{n}{e\:\textrm{\normalfont{log}}\frac{4}{3} }\right)^{n}\left(1+O(n^{-\frac{1}{2}})\right).\ee
\end{theorem} 
It is in fact possible also to compute the asymptotics of $\#_{1,n}$ as $n\to \infty$, we do this in Theorem \ref{thm:number}.
It is remarkable that the two numbers $b_0(\R^{1\times n}\backslash \Delta_{1, n})$ and $\#_{1,n}$ have the same asymptotic, up to a multiplicative constant\footnote{The constant $\frac{8}{e^{\frac{1}{12}}}$ is approximatively 7.36.}:
\be\label{eq:equi} b_0(\R^{1\times n}\backslash \Delta_{1,n}) =\frac{8}{e^{\frac{1}{12}}}\cdot\#_{1, n} \left(1+O(n^{-\frac{1}{2}})\right).\ee

The case when $d\geq 2$ is more delicate to handle. This is in large part because the discriminant in higher dimensions is an arrangement of quadrics rather than an arrangment of hyperplanes. For this general case we will prove the following upper and lower bounds for the number of rigid isotopy classes and isomorphism classes. 

\begin{theorem}\label{thm:lower}
For $d \geq 2$ fixed and for $n\geq 4d+1$ one has the following bounds:

\be \left(\frac{1}{(d+1)e^2}\right)^{dn}n^{dn}\leq \#_{d, n}\leq b_0\left(\R^{d \times n}\backslash \Delta_{d,n}\right)\leq 2dn \left(\frac{3e}{2d} \right)^{dn} n^{dn}
\ee
\end{theorem}
Note that these bounds imply that $\#_{d, n}$ and $b_0\left(\R^{d \times n}\backslash \Delta_{d,n}\right)$ become equivalent at the logarithmic scale, giving the following analogue of \eqref{eq:equi}:
\be \log b_0(\R^{d \times n}\backslash \Delta_{d,n})=\left(\log\#_{d, n}\right)(1+o(1))\quad \quad \textrm{as $n\to \infty$}. \ee
For the proof of the upper bound we will use the fact that $\Delta_{d,n}$ is a real algebraic set: using Alexander--Pontryagin duality, we bound the topology of the complement of the discriminant $\R^{d \times n} \setminus \Delta_{d, n}$ by studying the topology the one-point compactification of $\Delta_{d, n}$ which we denote by $\hat{\Delta}_{d, n}$. This one-point compactification can also be described in an algebraic way and studied using \cite{milnor1964betti}. 

For the lower bound, our proof is an higher dimensional version of \cite{McDiarmidMuller}, where the authors prove that in the case $d=2$ there exists a constant $\alpha>0$ such that $\#_{2,n}\geq\alpha^nn^{2n}$.

Finally we observe that one of our central questions here has been to establish bounds on the number of connected components of $\R^{d \times n} \setminus \Delta_{d, n}$ and more generally on $b_k(\R^{d \times n} \setminus \Delta_{d, n})$ for values of $k$ that are fixed but not necessarily zero. Moving from these low-degree Betti numbers, it is also interesting to look at the Betti numbers of $\R^{d \times n} \setminus \Delta_{d, n}$ in degrees close to $nd$. Because of Alexander duality to determine the $(dn - k)$--th Betti number of $\R^{d \times n} \setminus \Delta_{d, n}$ it suffices to determine the $(k - 1)$st Betti number of $\hat{\Delta}_{d, n}$. The next result gives information on some of the top Betti numbers; this result holds in general with no restriction that $d$ and $n$ be sufficiently large.
\begin{theorem}\label{thm:highbettinumbers}
For every $d$, $n$, $\hat{\Delta}_{d, n}$ is $(n + d - 3)$-connected, but not $(n + d - 2)$-connected. By Alexander--Pontryagin duality this implies that $H^{nd - n - d + 1}(\R^{d \times n} \setminus \Delta_{d, n}) \neq 0$, but all higher cohomology groups of $\R^{d \times n} \setminus \Delta_{d, n}$ vanish.
\end{theorem}
The description of $\Delta_{d, n}$ as a union of quadrics allows us to use a Mayer-Vietoris type argument, and the theorem follows by an application of a generalized nerve lemma.

%% file: Preliminaries.tex
\section{Preliminaries}
\input{GeometricGraphs}
\input{Duality}
\input{Semialgebraic}

\input{Quadrics}
\input{AnalyticCombinatorics}

%% file: GeometricGraphs.tex
\subsection{Geometric graphs}
There are several different notions of geometric graphs in the literature. The type of geometric graph we consider here are also sometimes called intersection graphs or space graphs. Formally the definition for geometric graph we use here is the following.

\begin{definition}[Geometric graph]\label{def:gg} Given a point $P\in \R^{d\times n}$ we denote by $G(P)$ the labeled graph whose vertices and edges are, respectively:
\be V(G(P))=\{(1,p_1), \ldots, (n,p_n)\}\quad \textrm{and}\quad E(G(P))= \{\textrm{$((i,p_i), (j,p_j)) \, |\, i<j,\, \|p_i-p_j\|^2 < 1$}\}.\ee
We say that $G(P)$ is an $\R^{d}$--geometric graph. If a labeled graph $G$ is isomorphic to $G(P)$ for some $P\in\R^{d\times n}$, we say that $G$ is \textit{realizable} as an $\R^{d}$--geometric graph.
 
 We say that the geometric graph $G(P)$ is nondegenerate if $P\notin \Delta_{d, n}$, where
\be \Delta_{d, n}=\{P\in \R^{d\times n}\,|\, \textrm{there exist $1\leq i<j\leq n$ such that $\|p_i-p_j\|^2=1$}\}\subset \R^{d\times n}.\ee
\end{definition}
\begin{remark}The reason for considering in our definition the list of pairs $\{(1,p_1), \ldots, (n,p_n)\}$ as the set of vertices of $G(P)$, instead of the list $\{p_1, \ldots, p_n\}$, is just formal. In other settings it may be more natural to take $p_1, ..., p_n$ to be \emph{distinct} points in $\R^d$ and then to define a graph with vertex set $\{p_1, ..., p_n\}$ and edges $(p_i, p_j)$ provided that $\| p_i - p_j \|^2< 1$. This is the approach taken by Maehara \cite{Maehara} who studies the \emph{sphericity} of graphs, the minimum dimension $d$ in which a graph may be realized as a geometric graph in $\R^d$ with vertices given by distinct points. For us it makes sense to  associate graphs on $n$ vertices in $\R^d$ to points of $\R^{d \times n}$ therefore the actual points in $\R^d$ may not all be unique from one another. The two-coordinate approach to describe the vertices allows for such repetition and naturally associates each point in $\R^d \setminus \Delta_{d, n}$ to unique labeled graph.

\end{remark}
We will consider the following notions of equivalence of geometric graphs.
\begin{definition}Let $G(P_0), G(P_1)$ be two $\R^d$--geometric graphs on $n$ vertices, with $P_0, P_1\in \R^{d\times n}$. We will say that they are \emph{isomorphic} if they are isomorphic as labeled geometric graphs. Moreover, if they are both nondegenerate, we will say that they are \emph{rigidly isotopic} if there exists a continuous curve $P:[0,1]\to \R^{d\times n}\backslash \Delta_{d,n}$ such that $P(0)=P_0$ and $P(1)=P_1.$ 
\end{definition}
Notice that, since $\Delta_{d,n}$ is an algebraic set, its complement is a semialgebraic set and its path-components are the same as its connected components. Therefore two nondegenerate geometric graphs $G(P_0), G(P_1)$ are rigidly isotopic if and only if $P_0$ and $P_1$ belong to the same connected component of $\R^{d\times n}\backslash \Delta_{d,n}$.  

Let us introduce the following notation:
\be \label{eq:count}\#_{d,n}:=\#\{\textrm{isomorphism classes of geometric graphs on $n$ vertices in $\R^d$}\}.\ee
Notice that in the definition of $\#_{d,n}$ we did not assume the nondegeneracy of the graphs. However, the following lemma proves that isomorphism classes of \emph{nondegenerate} graphs are the same as all isomorphism classes as in \eqref{eq:count} (see also \cite[Lemma 2.2]{ALL} for an analogous statement in the more general context of geometric complexes).
\begin{lemma}\label{lem:broad}For every $P\in \Delta_{d,n}$ there exists $\widetilde{P}\in\mathbb{R}^{d\times n}\backslash \Delta_{d,n}$ such that $G(P)$ and $G(\widetilde{P})$ are isomorphic as labeled graphs.
\end{lemma}
\begin{proof}Take $P=(p_1,\dots,p_n)\in\Delta_{d,n}$ and for $\epsilon>0$ small enough consider:
\be \widetilde{P}:=(1+\epsilon)P.\ee
Then for $\epsilon$ small enough we have:
\be ((i,p_i), (j, p_j))\in E(G(P)) \iff ((i,(1+\epsilon)p_i), (j, (1+\epsilon)p_j))\in E(G(\widetilde{P})),\ee
which proves that $G(P)$ and $G(\widetilde{P})$ are isomorphic as labeled graphs. Moreover, again for $\epsilon$ small enough, we have that $\widetilde{P}\notin \Delta_{d,n}.$
\end{proof}
By definition, for nondegenerate graphs we have \be\textrm{rigidly isotopic} \Longrightarrow\textrm{isomorphic}, \ee
this means that isomorphism classes are union of rigid isotopy classes. The number of rigid isotopy classes of geometric graphs on $n$ vertices in $\R^d$ is given $b_0(\R^{d\times n}\setminus\Delta_{d,n})$ and Lemma \ref{lem:broad} implies we can compare the number of rigid isotopy classes with the number of isomorphism classes:
\be \#_{d,n}\leq b_0(\R^{d\times n}\setminus\Delta_{d,n}).\ee
Below we will prove, as Corollary \ref{cor:conn1}, that for $d\geq n+1$, two $\R^d$--geometric graphs on $n$ vertices are isomorphic if and only if they are rigid isotopic, i.e. that
\be \#_{d,n}= b_0(\R^{d\times n}\setminus\Delta_{d,n})\quad \quad \textrm{for $d\geq n+1$}.\ee
We will deal with the asymptotic of $\#_{d, n}$ and $b_0(\R^{d\times n}\backslash \Delta_{d, n})$ in the case $d$ fixed and $n\to \infty$ in Section \ref{sec:n}.

%% file: Duality.tex
\subsection{Alexander duality and the discriminant}\label{sec:aldu}
As we are interested in the topology of $\R^{d \times n} \setminus \Delta_{d, n}$, the topology of $\Delta_{d, n}$ should play an important role as well and in some case it will be easier to study. The key tool to connect the topology of the two is Alexander duality. Given a compact, locally contractible, nonempty and proper subspace $X$ of the $N$-dimensional sphere $S^N$, Alexander duality \cite[Corollary 3.45]{Hatcher} provides a way to study the topology of $X$ from the topology of $S^N \setminus X$. Namely for every $k$ we have the following isomorphisms between the homology of $X$ and the cohomology of $S^N \setminus X$
\[\tilde{H}_k(X) \cong \tilde{H}^{N - k - 1}(S^N \setminus X).\]

\begin{remark} If we are working with $\mathbb{Z}_2$ coefficients, as we will throughout, and with a space $X\subset S^N$ with finitely generated homology, we can relate the Betti numbers of $X$ with those of its complement in the sphere (i.e. we can freely identify homology and cohomology). When working with compact semialgebraic sets in the sphere, this last requirement will be satisfied thanks to \cite[Theorem 9.4.1]{BCR}.
\end{remark}
In order to use this duality in the present setting we work in the one-point compactification of $\Delta_{d, n} \subset \R^{d \times n}$ which we will denote by $\hat{\Delta}_{d, n} \subset S^{d \times n}$. Now $\hat{\Delta}_{d, n}$ contains the point at infinity so $S^{d \times n} \setminus \hat{\Delta}_{d, n}  = \R^{d \times n} \setminus \Delta_{d, n}$.

The discriminant itself is a union of quadratic hypersurfaces of the form
\be\label{eq:ij}\Delta_{d, n}^{i, j} = \{(x_1, ..., x_n) \in \R^{d \times n} \mid \|x_i - x_j\|^2 = 1\}.\ee Each of these quadrics is topologically $S^{d - 1} \times \R^{d \times (n - 1)}$ and establishing bounds on the top Betti number of $\hat{\Delta}_{d, n}$ establishes bounds on the number of rigid isotopy classes of $\R^d$-geometric graphs on $n$ vertices. We take such an approach in Section \ref{sec:upperboundb0}.

The discriminant itself is also interesting from the perspective of graph theory. One defines the unit-distance graph on $\R^d$ to be the graph whose vertex set is all of $\R^d$ and two vertices $x$ and $y$ are connected by an edge if and only if $\norm{x-y} = 1$. As a remark, the unit distance graph on $\R^d$, especially in the case $d = 2$, is studied in the case of the well known Hadwiger--Nelson problem to establish the chromatic number of the plane; that is the chromatic number of the unit distance graph on $\R^2$. 

In our situation $\Delta_{d, n}$ is connected to homomorphisms from graphs on $n$ vertices to the unit distance graph on $\R^d$. Indeed each quadric $\Delta_{d, n}^{i, j}$ is the space of images of homomorphisms from the graph on $[n]$ with an edge between vertex $i$ and vertex $j$. More generally, we denote an intersection of quadrics by $\Delta_{d, n}^G$ for a graph $G = ([n], E)$ by 
\be \label{eq:DG}\Delta_{d,n}^{G}:=\bigcap_{(i,j)\in E}\Delta_{d,n}^{(i,j)}.\ee
Then $\Delta_{d, n}^G$ is the space of images of homomorphisms from the graph $G$ to the unit distance graph on $\R^d$. Putting all of this together we have that $\Delta_{d, n}$ itself is the set of all points in $\R^{d \times n}$ that are the image of a graph homomorphism for a nonempty graph on $n$ vertices. 

%% file: Semialgebraic.tex
\subsection{Semialgebraic triviality} A most useful technical tool that we will use in the paper is Theorem \ref{thm:sat}, which relates the structure of semialgebraic families and their homotopy.

\begin{definition}
Let $S,$ $T$ and $T'$ be semialgebraic sets, $T'\subset T$, and let
$f:S\to T$ be a continuous semialgebraic mapping. A semialgebraic trivialization
of $f$ over $T'$, with fibre $F$, is a semi-algebraic homeomorphism
$\theta:T'\times F\to f^{-1}(T')$, such that $f\circ\theta$ is the projection mapping $\pi:T'\times F\to T'$.
We say that the semialgebraic trivialization $\theta$ is compatible with a subset
$S'$ of $S$ if there is a subset $F'$ of F such that $\theta(T'\times F') = S'\cap f^{-1}(T')$.
\end{definition}

\begin{theorem}[Semialgebraic triviality]\label{thm:sat} Let $S$ and $T$ be two semialgebraic sets, $f:S\to T$ a semialgebraic mapping, $(S_j )_{j=1,\dots,q}$ a finite family of semialgebraic subsets of $S$. There exist a finite partition of $T$ into semialgebraic sets $T=\bigcup_
{l=1}^r T_l$ and, for
each $l$, a semialgebraic trivialization $\theta_l: T_l\times F_l\to f^{-1}(T_l)$ of $f$ over $T_l$ compatible with
$S_j$ , for $j = 1,\dots, q$, i.e, there exists $F_l^j\subset F_l$ such that $\theta_l(T_l\times F^j_l)= S_j\cap f^{-1}(T_l)$.
\end{theorem}
\begin{proof}This is \cite[Theorem 9.3.2]{BCR}.
\end{proof}

\begin{corollary} \label{coro:semtriv}Let $S$ be a semialgebraic set and $f:S → \R$ be a continuous semialgebraic function. Then for $\epsilon>0$ small enough the inclusion
\[\{f \geq \epsilon\}\xhookrightarrow{} \{f > 0\}\]
is an homotopy equivalence.
\end{corollary}
\begin{proof}
Thanks to semialgebraic triviality, we know that for $\epsilon$ sufficiently small there exists $T_l$ such that $(0,\epsilon]\subset T_l$. Then, we define a map \[H:\{f>0\}\times[0,1]\to \{f>0\}\]
\[H(x,t)=\begin{cases} x & \textrm{if } f(x)\not\in (0,\epsilon) \\\theta_l((1-t)\cdot(\pi_1\circ\theta_l^{-1})+t\epsilon,\pi_2\circ\theta_l^{-1}) & \textrm{if } f(x)\in (0,\epsilon]\end{cases}\]
and this is a continuous function because the two expressions agree on $f^{-1}(\epsilon)\times[0,1]$ and both of them are continuous on closed subsets.
The map $H$ is therefore a deformation retraction of $\{f>0\}$ onto $f\geq\epsilon$.
\end{proof}

%% file: Quadrics.tex
\subsection{Systems of quadratic inequalities }\label{sec:spectral}

In this section we recall a general construction from \cite{Agrachev, AgrachevLerario} for computing the Betti numbers of the set of solutions of a system of quadratic inequalities.

To start with, let $h:\R^{N+1}\to \R^{k+1}$ be a \emph{quadratic map}, i.e. a map whose components $h=(h_0, \ldots, h_k)$ are homogeneous quadratic forms. Let also $K\subseteq \R^{k+1}$ be a closed convex polyhedral cone (centered at the origin). We are interested in the Betti numbers of
\be\label{eq:V} V=h^{-1}(K)\cap S^{N}\subset \R^{N+1}\ee
Such a set $V$ can be seen as the set of solutions of a system of homogeneous quadratic inequalities on the sphere $S^N$: in fact, since $K$ is polyhedral, we have
\be K=\{\eta_1\leq 0, \ldots, \eta_\ell\leq 0\}\ee
for some linear forms $\eta_1, \ldots, \eta_\ell\in (\R^{k+1})^*$ and 
\be V=\{\eta_1 h\leq 0, \ldots, \eta_\ell h\leq 0\}\cap S^N,\ee
which is a system of quadratic inequalities. (Here given a linear form $\eta \in (\R^{k+1})^*$ and a quadratic map $h:\R^{N+1}\to \R^{k+1}$ we simply denote by $\eta h$ the composition of the two). Every homogeneous system can be written in this way. 

We denote by $K^\circ$ the polar of $K$, i.e. $K^{\circ}=\{\eta\in (\R^{k+1})^*\,|\, \eta(y)\leq 0\quad \forall y\in K\}$, and we set:
\be \Omega=K^{\circ}\cap S^k,\ee
where $S^k$ denotes the unit sphere in $(\R^{k+1})^*$, with respect to a fixed scalar product. The scalar product on $\R^{k+1}$ plays no role, but we will also use a scalar product on $\R^{N+1}$, by choosing a positive definite quadratic form $g$ on $\R^{N+1}$. This scalar product will play a role, and we denote it by $\langle \cdot, \cdot\rangle_{g}$; it is defined by $\langle x, x\rangle_g=g(x)$ for all $x\in \R^{N+1}$; for practical purposes we will omit the ``$g$'' subscripts when not needed.

Once the scalar product on $\R^{N+1}$ has been fixed, we can associate to a quadratic form $q:\R^{N+1}\to \R$ a real symmetric matrix, via the equation:
\be\label{eq:polarization} q(x)=\langle x, Qx \rangle_g\quad \forall x\in \R^{N+1}.\ee

We will often use small letters for the quadratic form and capital letters for the associated matrices.

Accordingly we can define the eigenvalues of $q$ (with respect to $g$) as those of $Q$:
\be \lambda_1(q)\geq \cdots \geq\lambda_{N+1}(q).\ee

The eigenvalues (and the eigenvectors) of $q$ depend therefore on the chosen scalar product, but again we will omit this dependence in the notation if not needed.

\subsubsection{The index function}
Using the above notation, we will denote by $\textrm{ind}^{+}(q)=\textrm{ind}^{+}(Q)$ the positive inertia index, i.e. the number of positive eigenvalues of the symmetric matrix $Q$. Note that the index of a quadratic form \emph{does not} depend on the chosen scalar product.

When we are in the situation as above, i.e. when we are given a homogeneous quadratic map $h:\R^{N+1}\to \R^{k+1}$, for every covector $\eta\in (\R^{k+1})^*$ we can consider the composition $\eta h$, which is a quadratic form. For every natural number $ j\geq 0$ we define the sets:
\be \Omega^j=\{\omega\in \Omega\,|\, \mathrm{ind}^+(\omega h)\geq j\}.\ee
These sets are open and semialgebraic, as it is easily verified. Moreover, these sets \emph{do not} depend on the choice of the scalar product $g$.

Observe now that over each set $\Omega^{j}\backslash \Omega^{j+1}$ the function $\textrm{ind}^+\equiv j$ is constant, i.e. the number of positive eigenvalues of the corresponding matrices is $j$ and there exists a natural vector bundle $P^{j}\subseteq \Omega^{j}\backslash \Omega^{j+1}\times \R^{N+1}$
\be \label{eq:bundle}
\begin{tikzcd}
\mathbb{R}^j \arrow[r, hook] & P^j \arrow[d] \\
                             &   \Omega^{j}\backslash \Omega^{j+1}        
\end{tikzcd}
\ee
whose fiber over a point $\omega$ is the positive eigenspace of $\omega H=\omega_0H_0+\cdots +\omega_kH_k.$ In fact this bundle is the restriction of a more general bundle over the set
\be D_j=\{\omega\,|\, \lambda_j(\omega H)\neq \lambda_{j+1}(\omega H)\}\ee
i.e. the set where the $j$-th eigenvalue of $\omega H$ is distinct from the $(j+1)$-th. We still denote this bundle by $P_j\subset D_{j}\times \R^{N+1}$: 
\be \label{eq:bundle2}
\begin{tikzcd}
\mathbb{R}^j \arrow[r, hook] & P^j \arrow[d] \\
                             &   D_j                
\end{tikzcd}
\ee 
Here the fiber over a point $\omega\in D_{j}$ consists of the eigenspace of $\omega H$ associated to the first $j$ eigenvalues (this is well defined); however note that the bundle over $D_j$ \emph{depends} on the choice of the scalar product (since $D_j$ itself depends on this choice).

We denote by
\be \label{eq:sw}\nu_j\in H^{1}(D_j)\ee
the first Stiefel-Whitney class of this bundle. The following Lemma will be useful for us: \begin{lemma}\label{lem:makesense}The cup-product with the class $\nu_j$ defines a map:
\be\label{eq:smile} (\cdot)\smallsmile \nu_j:H^*(\Omega ^{j}, \Omega^{j+1})\to H^{*+1}(\Omega ^{j}, \Omega^{j+1}).\ee
\end{lemma}
\begin{proof}To see that the previous cup product is well defined observe that we can write:
\be \Omega^{j}=A\cup B, \quad A=\Omega^{j+1}, \quad B=\Omega^{j}\cap D_j.\ee
In fact, if a point $\omega$ belongs to $\Omega^{j}$ then, either $\omega \in \Omega^{j+1}$ or $\mathrm{ind}^{+}(\omega H)=j$ and consequently $\omega \in \Omega^{j}\cap D_{j}.$ Since both $A$ and $B$ are open, by excision we get:
\be H^*(\Omega^{j}, \Omega^{j+1})\simeq H^*(\Omega^{j}\cap D_{j}, \Omega^{j+1}\cap D_{j}).
\ee
In particular, in order to see that \eqref{eq:smile} is well defined, it is enough to see that 
\be (\cdot)\smallsmile \nu_j:H^*(\Omega ^{j}\cap D_{j}, \Omega^{j+1}\cap D_{j})\to H^{*+1}(\Omega ^{j}\cap D_{j}, \Omega^{j+1}\cap D_{j})\ee
is well defined. 
Let $\psi\in C^{k}(\Omega^{j})$ a singular cochain representing a cohomology class in $H^{k}(\Omega ^{j}\cap D_{j}, \Omega^{j+1}\cap D_{j})$ and $\phi_j\in C^1(D_{j})$ a cochain representing $\nu_j.$ The cup product of $\psi$ and $\phi_j$ is defined on a singular chain $\sigma:[v_0, \ldots , v_{k + 1}]\to \Omega^{j}\cap D_{j}$ in the usual way:
\be (\psi\smallsmile \phi_j)(\sigma)=\psi(\sigma|_{[v_0, \ldots, v_k]})\phi_j(\sigma|_{[v_k, v_{k+1}]}),\ee
from which we see that $\psi\smallsmile \phi_j$ vanishes on $C_{k+1}(\Omega^{j+1}\cap D_j)$ and defines an element of  $H^{k+1}(\Omega ^{j}\cap D_{j}, \Omega^{j+1}\cap D_{j})$.
\end{proof}
Since we have inclusions $\Omega^{j}\supseteq \Omega^{j+1}\supseteq \Omega^{j+2}$, we also consider the connecting homomophisms
\be \partial:H^{*}(\Omega^{j+1} , \Omega^{j+2} )\to H^{*+1}(\Omega^{j} , \Omega^{j+1} )\ee
of the long exact sequence for the triple $(\Omega^{j} ,\Omega^{j+1} , \Omega^{j+2} ).$

We summarize the directions of these homomorphism in the following \emph{non-commutative} diagram of maps:
\be 
\begin{tikzcd}
{H^i(\Omega^{j+1} , \Omega^{j+2} )} \arrow[rr, "\partial"] \arrow[dd, "(\cdot)\smallsmile \nu_{j+1} "'] &  & {H^{i+1}(\Omega^{j} , \Omega^{j+1} )} \arrow[dd, "(\cdot)\smallsmile \nu_{j} "] \\
                                                                                                                                   &  &                                                                                                            \\
{H^{i+1}(\Omega^{j+1} , \Omega^{j+2} )} \arrow[rr, "\partial"]                                                   &  & {H^{i+2}(\Omega^{j} , \Omega^{j+1} )}                                                   
\end{tikzcd}
\ee

\begin{remark}
\label{rem:delta} Let $\partial:H^{i}(X,Y)\to H^{i+1}(Z,X)$ be the boundary operator in the exact sequence of the triple $(Z,X,Y)$, where all spaces are open. Following \cite[pag. 201]{Hatcher}, thanks to the fact that we are working with $\mathbb{Z}_2$ coefficients, we have that $\partial([\phi])=[\phi\circ\pi\circ\delta]$ where $\delta:C_{i+1}(Z,X)\to C_{i}(X)$ is the boundary operator and $\pi: C_{i}(X)\to  C_{i}(X,Y)$ is the projection operator. Let us also consider $\tilde{X}\subset A$ both open and such that $(X\cap\tilde{X})\cup Y$ is open. If we take the relative cup product
\[H^i(X,Y)\times H^1(\tilde{X})\to H^{1}(X,Y)\]
as defined in Lemma \ref{lem:makesense}, then this coincides with the cup product 
\[H^i(X,Y)\times H^1(A)\to H^{1}(X,Y)\]
as defined in Lemma \ref{lem:makesense}, meaning that given $a\in H^i(X,Y)$ and $b\in H^1(A)$ then $a\smallsmile b= a\smallsmile r^*(b)$ where $r^*$ is just the restriction. In the same way, if we suppose that $\tilde{Z}\subset A$, we can repeat a similar reasoning for the cup product $H^i(Z,X)\times H^1(\tilde{Z})\to H^{i+1}(Z,X)$.
Now, given $[\gamma]\in H^1(A)$, we claim that 
\be\label{eq:cc}\partial([a]\smallsmile [\gamma])=\partial [a]\smallsmile [\gamma].\ee
At the level of cochains if we take $c_{i+1}\in C_{i+1}(Z,X)$ a singular chain we have $\partial(a\smallsmile \gamma)(c_{i+1})=(a\smallsmile \gamma)(\pi\circ\delta c_{i+1})$. Reasoning as in the proof of Lemma 3.6 in \cite[pag. 206]{Hatcher}, we see that $(a\smallsmile \gamma)(\pi\circ\delta c_{i+1})=\partial a\smallsmile \gamma+(-1)^i a\smallsmile \delta^*(\gamma)$ and \eqref{eq:cc} follows from $\delta^*(\gamma)=0$.
\end{remark}

\subsubsection{The spectral sequence}
For the computation of the Betti numbers of $V$, defined in \eqref{eq:V} we will need the following result, which is an adaptation from \cite{Agrachev, AgrachevLerario}. Clearly the computation of the cohomology of $V$ is equivalent to that of $S^{N}\backslash V$, by Alexander duality, and in \cite{Agrachev, AgrachevLerario} it is introduced a spectral sequence for computing the latter.

The delicate part here is that in \cite{Agrachev} the spectral sequence is defined for \emph{nondegenerate} systems of quadrics, i.e. for systems such that the map $h$ is transversal to the cone $K$, in the sense of \cite{AgrachevLerario}; in \cite{lerario2011homology} the spectral sequence is defined also for degenerate systems, but the second differential is not computed explicitly, and in \cite{AgrachevLerario} it is defined also for degenerate systems, and the second differential is explicitly computed, but the solutions are studied in the projective space rather than the sphere. Since in our case the system of quadratic inequalities might be degenerate, we will need to prove the existence of such a spectral sequence and to compute its second differential. 

\begin{theorem}\label{thm:spectralbasis}Let $V=h^{-1}(K)\cap S^{N}$ be defined by a system of quadratic inequalities, as above. There exists a cohomology spectral sequence $(E_r, d_r)_{r\geq 1}$ converging to $H^{*}(S^N\backslash V;\mathbb{Z}_2)$ such that:
\begin{enumerate}
\item 
the second page of the spectral sequence is given, for $j>0$, by
\be E_2^{i,j}=
 H^i(\Omega^{j+1}, \Omega^{j+2};\mathbb{Z}_2).\ee
 For $j=0$, the elements of the second page of the spectral sequence fit into a long exact sequence:
\be\label{eq:exact} \cdots \rightarrow H^i(\Omega^{1};\mathbb{Z}_2)\rightarrow E_{2}^{i, 0}\rightarrow H^i(\Omega^{1} , \Omega^{2} ;\mathbb{Z}_2)\stackrel{(\cdot) \smile \nu_1 }{\longrightarrow} H^{i+1}(\Omega^{1} ;\mathbb{Z}_2)\rightarrow\cdots.\ee
\item for $j\geq{1}$ the second differential $d_2^{i,j} :H^i(\Omega^{j+1} , \Omega^{j+2} )\to H^{i+2}(\Omega^{j} , \Omega^{j+1} )$ is given by 
\be d_2^{i,j} \xi=\partial(\xi\smallsmile \nu_{j+1} )+\partial \xi\smallsmile \nu_{j} .\ee
\end{enumerate} 
\end{theorem}
\begin{proof}The proof proceeds similar to \cite[Theorem 25 and Theorem 28]{AgrachevLerario}, using a regularization process and taking the limit over the regularizing parameter. More precisely, let $q_0$ be a positive definite quadratic form, chosen as in \cite[Lemma 13]{AgrachevLerario}, and for $t>0$ consider the set:
\be B(t)=\{(\omega, x)\in \Omega\times S^N\,|\, \omega h(x)-tq_0(x)\geq 0\}.\ee
The choice of $q_0$ as in \cite[Lemma 13]{AgrachevLerario} makes the map $\omega\mapsto \omega h-tq_0$ nondegenerate with respect to $K$ and will allow to compute the second differential of our spectral sequence.
By semialgebraic trivilality, for $t>0$ small enough the set $B(t)$ is homotopy equivalent to
\be B=\{(\omega, x)\in \Omega\times S^N\,|\, \omega h(x)>0\}.\ee Moreover the projection on the second factor (i.e. $p_2:\Omega\times S^n\to S^N$) restricts to a homotopy equivalence $B\sim p_2(B)=S^{N}\backslash V$, see \cite[Section 3.2]{lerario2011homology}. Therefore, for $t>0$ small enough 
\be H^*(S^{N}\backslash V)\simeq H^*(B(t)).
\ee

We consider now the Leray spectral sequence $(E_r[t], d_r[t])_{r\geq 0}$ of the map 
\be p_{t}:=p_1|_{B(t)}:B(t)\to \Omega.\ee
This spectral sequence converges to the cohomology of $B(t),$ 

  For the first part of the statement, the structure of  $E_2^{i,j} $ in the case $j>0$ is proved in \cite[Section 3.2]{lerario2011homology}, as follows.
If $t_1<t_2$ then $B(t_2)\xhookrightarrow{} B(t_1)$ is an homotopy equivalence and $p_{t_1}\vert_{B(t_2)}=p_{t_2}$. For $0<t_1<t_2<\delta$ the inclusion defines a morphism of filtered differential graded modules \[\phi_0(t_1,t_2):(E_0 [t_1],d_0 [t_1])\to(E_0 [t_2],d_0 [t_2])\] turning $\{E_0 [t]\}_{t}$ into an inverse system and thus $\{(E_r [t],d_r [t])\}_t$ into an inverse system of spectral sequences. Then, we can define a new spectral sequence \[(E_r ,d_r ):=\varprojlim_t\{(E_r [t],d_r [t])\}.\] 

The proof shows that for $j>0$ we have $E_2^{i,j} [t]=H^i (\Omega_{n−j} [t], \Omega_{n−j−1} [t];\mathbb{Z}_2)$ where the sets $\Omega_{k}[t]$ are defined by:
\[\Omega_k [t]:=\{w\in\Omega\mid\;i^{-}(w\cdot h-tg)\leq k\}.\] 
Moreover we also have that for $j>0$ the isomorphism $\phi_2(t_1,t_2)$ is just the homomorphism induced in cohomology  by the inclusion $\Omega_j [t_2]\subseteq\Omega_j [t_1]$ and that
\[E_{2}^{i,j} =\varprojlim_{t} E_{2}^{i,j} [t]=H^{i}(\Omega^{j} ,\Omega^{j+1} ;\mathbb{Z}_2).\]
Thanks to this, by semialgebraic triviality $\phi_2(t_1,t_2)$ is an isomorphism for $0<t_1<t_2<\delta$ with $\delta$ sufficiently small and therefore also $\phi_{\infty}(t_1,t_2)$ is an isomorphism, assuring the convergence of $(E_{r} ,d_r )$ to $B(t)$ (see also the proof of \cite[Theorem 25]{AgrachevLerario} for more details on this point).
Let us call \[e_{t}^*:H^{*}(\Omega^{j},\Omega^{j+1};\mathbb{Z}_2)\to H^*(\Omega_{n−j}(t), \Omega_{n−j−1}(t);\mathbb{Z}_2)\] the isomorphism induced by the inclusion.
From now on we choose our scalar product on $\R^{N+1}$ to be $g=q_0$, in such a way that the  matrix associated to $q_0$ through the polarization identity \eqref{eq:polarization} is the identity matrix.

For the case $j=0$ we know thanks to \cite{agrachev1988homology} that there exist a long exact sequence
\begin{align} \cdots \rightarrow H^i(\Omega_{n-1} [t];\mathbb{Z}_2)\rightarrow E_{2}^{i, 0} [t]\rightarrow\\\rightarrow H^i(\Omega_{n-1} [t], \Omega_{n-2} [t]&;\mathbb{Z}_2)\stackrel{(\cdot) \smile \nu_1 }{\longrightarrow} H^{i+1}(\Omega_{n-1} [t];\mathbb{Z}_2)\rightarrow\cdots.\end{align}
We can pass to the inverse limit of these long exact sequences respect to $t$ in the obvious way obtaining a long exact sequence\footnote{In this long exact sequence we are still using $\nu_1$ because we chose our scalar product to be $g_0$. Same for the definition of $d_2(t)$, where we used $\nu_j$.}
\be\cdots \rightarrow H^i(\Omega^{1} ;\mathbb{Z}_2)\rightarrow E_{2}^{i,0} \rightarrow H^i(\Omega^{1} , \Omega^{2} ;\mathbb{Z}_2)\stackrel{(\cdot) \smile \nu_1 }{\longrightarrow} H^{i+1}(\Omega^{1} ;\mathbb{Z}_2)\rightarrow\cdots,\ee
where we have used the fact that $(e^{*})^{-1}\circ((\cdot)\smallsmile \nu_j )\circ e^*=(\cdot)\smallsmile \nu_j $ (we will get back to this point later).

This proves point (1) of the statement. For what concerns the differential, thanks to \cite[Theorem 3]{agrachev1988homology} we know that the second differential $d_2 [t]$ of the spectral sequence $(E_r [t],d_r [t])$ with \[d_2^{i,j} [t]: H^i (\Omega_{n−j-1} [t], \Omega_{n−j−2} [t];\mathbb{Z}_2)\to H^{i+2} (\Omega_{n−j} [t], \Omega_{n−j-1} [t];\mathbb{Z}_2)\] has the form 
\[d_2^{i,j} [t]\xi=\partial_{t}(i_{t}^*)^{-1}( i_{t}^*\xi\smallsmile \nu_{j+1} )+(i_{t}^*)^{-1}(i_{t}^*\partial_{t}\xi\smallsmile \nu_{j} ),\]
where \[\partial_{t}:H^i (\Omega_{n−j-1} [t], \Omega_{n−j−2} [t];\mathbb{Z}_2)\to H^{i+1} (\Omega_{n−j} [t], \Omega_{n−j-1} [t];\mathbb{Z}_2)\] is the connecting homomorphism in the exact sequence of the triple $(\Omega_{n−j} [t], \Omega_{n−j-1} [t],$ $\Omega_{n−j-2} [t])$ and the map \[i_t^*:H^i(\Omega_{n−j} [t], \Omega_{n−j−1} [t];\mathbb{Z}_2)\to H^i(\Omega_{n−j} [t]\cap\mathcal{D}_j , \Omega_{n−j−1}(t)\cap\mathcal{D}_j ;\mathbb{Z}_2)\] is the map induced by the inclusion; this map is an isomorphism by excision.

The second differential for $j>1$ of our new spectral sequence $(E^{i,j}_r ,d_r )$ is given by $d_2^{i,j}:=(e^*_{t})^{-1}\circ d_2^{i,j}(t)\circ e^*_{t}. $ More explicitly,
\[ d_2^{i,j}=\partial(i_{t}^*\circ e^*_{t})^{-1}((i_{t}^*\circ e^*_{t})\xi\smallsmile \nu_{j+1} )+(i_{t}^*\circ e_{t}^*)^{-1}( (i_{t}^*\circ e_{t}^*)\partial\xi\smallsmile \nu_{j} ) \]
thanks to the naturality of the connecting homomorphism.

Let us now consider the following diagram where all the maps are inclusions:
\[\begin{tikzcd}
                                                              & {(\Omega^j ,\Omega^{j+1} )} \\
                                                              &                                                                                             &                                                                            \\
(\Omega_{n−j-1} [t]\cap\mathcal{D}_j , \Omega_{n−j−2} [t]\cap\mathcal{D}_j ) \arrow[rr, "j_t"] \arrow[ruu, "{e_t\circ i_t}"] & 
                              & {(\Omega^j \cap \mathcal{D}_j ,\Omega^{j+1} \cap\mathcal{D}_{j} )} \arrow[luu, "i"']          
\end{tikzcd}\]
and all the induced homomorphisms in cohomology are isomorphisms. We can write
\begin{align}d_2^{i,j}=&\partial(j_t^*\circ i^*)^{-1}((j_t^*\circ i^*)\xi\smallsmile \nu_{j+1} )+(j_t^*\circ i^*)^{-1}( (j_t^*\circ i^*)\partial\xi\smallsmile \nu_{j} )=\\  =&\partial(i^*)^{-1}(j^*)^{-1}(j_t^*(i^*\xi)\smallsmile j_t^*\circ \nu_{j+1} )+(i^*)^{-1}(j_t^*)^{-1}(j_t^*(i^*\partial\xi)\smallsmile j_t^*\circ \nu_{j} )=\\=&\partial(i^*)^{-1}(i^*\xi\smallsmile \nu_{j+1} )+(i^*)^{-1}(i^*\partial\xi\smallsmile \nu_{j} )\end{align}
where the pull-back property of the pullback in the third equality holds true because in that case it is just the standard cup-product. Because of how we defined the cup product in Lemma \ref{lem:makesense}, we have the claim.
\end{proof}

%% file: AnalyticCombinatorics.tex
\subsection{Analytic Combinatorics}
In order to study the asymptotic of the number of isotopy classes of geometric graphs on the real line we will need some tools from analytic combinatorics. For a full introduction to the topic see \cite{flajolet2009analytic}.
Given a generating function $G(x)=\sum_{n=0}^{\infty}a_n x^n$ of a sequence $a_n$ we want to study the asymptotic of such a sequence. There are various techniques to do this.
\begin{definition}
We say that a sequence $\{a_n\}$ is of exponential order $K^n$ which we abbreviate as $a_n\bowtie K^n$ if and only if $\limsup |a_n|^{\frac{1}{n}}=K$.
\end{definition}

If we have $a_n\bowtie K^n$ then $a_n=K^n\theta(n)$ with $\limsup |\theta(n)|^{\frac{1}{n}}=1$. The term $\theta(n)$ is called subexponential factor.
In order to study the subexponential factor $\theta(n)$ we should look at the singularities of the generating function.


\begin{definition}
Given two numbers $\phi$, $R$ with $R>1$ and $0<\phi<\frac{\pi}{2}$, the open domain $D(\phi,R)$ is defined as 
\[D(\phi,R):=\{z\mid |z|<R,\:z\neq 1,\: |arg(z-1)|>\phi\}.\]
A domain of this type is called $D-$domain.
\end{definition}

Denoting by $S$ the set of all meromorphic functions of the form
\[S:=\{(1-z)^{-\alpha}\mid\alpha\in\R\},\]
we recall the next result \cite[Theorem VI.4]{flajolet2009analytic}, which we will need in the sequel.
\begin{theorem}\label{theorem:anacomb}
Let $G(z)$ be an analytic function at $0$ with a singularity at $\zeta$, such that $G(z)$ can be continued to a domain of the form $\zeta\cdot D_0$, for a $D-$domain $D_0$, where $\zeta\cdot D_0$ is the image of $D_0$ by the mapping $z\to\zeta z$. Assume there exists two functions $\sigma,\tau,$ where $\sigma$ is a finite linear combination of elements in $S$ and $\tau\in S$ such that
\[G(z)=\sigma  \left( \frac{z}{\zeta} \right) +O(\tau\left( \frac{z}{\zeta}\right))\:\:\:\textrm{\normalfont{as}}\:\:z\to \zeta \:\:\textrm{\normalfont{in}}\:\:\zeta\cdot D_0.\]
Then the coefficients of $G(z)$ satisfy the asymptotic estimate
\[a_n=\zeta^{-n}\sigma_n+O(\zeta^{-n}\tau_n^{*})\]
where $\sigma(z)=\sum_{n=0}^{\infty}\sigma_n z^n$ and $\tau_n^{*}=n^{\alpha-1}$, if $\tau(z)=(1-z)^{-\alpha}$.
\end{theorem}

\begin{remark}\label{remark:anadelt}
For a later use, we record the following. The Newton Binomial Series  is defined by
\[(1-z)^{-\alpha}=\sum_{n=0}^{\infty}b_n z^n\textnormal{  where  }b_n=\binom{n+\alpha-1}{n}.\]
Using 
\[\binom{n+\alpha-1}{n}=\frac{\Gamma(n+\alpha)}{\Gamma(\alpha)\Gamma(n+1)},\]
we get the following asymptotic for its coefficients:
\[b_n=\frac{n^{\alpha-1}}{\Gamma(\alpha)}\left( 1+O\left(\frac{1}{n}\right)\right).\]
\end{remark}

%% file: Increasing_d.tex
\section{Homology of the chambers and the Floer number}
\subsection{Graphs and sign conditions}Recall that given a graph $G$ on $n$ vertices we have defined
\be W_{G, d}=\{P\in \R^{d\times n}\backslash \Delta_{d,n}\,|\, G(P)\cong G\}\subset \R^{d\times n}.\ee
In other words, $W_{G,d}$ consists of all the points $P\in \R^{d\times n}$ not on the discriminant whose corresponding graph is isomorphic to $G$. For small $d$ this set could be a union of several chambers, but for large $d$ it is an actual chamber (a connected open set).

Now we introduce an alternative notation for labelling the sets $W_{G, d}$. For every $1\leq i<j \leq n$ let us denote by $q_{ij}:\R^{d\times n}\to \R$ the quadratic polynomial:
\be\label{eq:qp} q_{ij}(x_1, \ldots, x_n)=\|x_i-x_j\|^2-1,\quad(x_1, \ldots x_n)\in \R^{d\times n}.\ee
Notice that the discriminant $\Delta_{d, n}$ is given by:
\begin{align}\label{eq:deltap} \Delta_{d,n}&=\bigg\{(x_1, \ldots, x_n)\in \R^{d\times n}\,\bigg|\, \prod_{i<j}q_{ij}(x_1,\ldots, x_n)=0\bigg\}\\
&=\bigcup_{i<j}\Delta_{d,n}^{(i,j)},\end{align}
where the sets $\Delta_{d,n}^{(i,j)}$ are defined in \eqref{eq:ij}.
We denote by $\binom{[n]}{2}$ the set of all possible pairs $1\leq i<j\leq n$  and by $2^{\binom{[n]}{2}}$ the set of all possible choices of signs $\sigma_{ij}\in\{\pm\}$ for elements in $\binom{[n]}{2}$. 
\begin{definition}[Sign condition]For every $\sigma\in 2^{\genfrac{\{ }{\}}{0pt}{}{n}{2}}$ we denote by $W_{\sigma, d}\in \R^{d\times n}$ the open set:
\be W_{ \sigma,d}=\left\{x=(x_1, \ldots, x_n)\in \R^{d\times n}\,\big|\, \mathrm{sign}(q_{ij}(x))=\sigma_{ij}\, \right\}.\ee
\end{definition}
At this point what is clear from \eqref{eq:deltap} is that $\R^{d\times n}\backslash \Delta_{d,n}$ can be written as the union of all the possible sign conditions. 
The following lemma will be useful. It tells that we can label the chambers of $\R^{d\times n}\backslash \Delta_{d,n}$ either with a graph or with a sign condition -- however at this point we only prove that the sets $\{W_{G, d}\}$ and $\{W_{\sigma, d}\}$ coincide; the fact that the chambers of $\R^{d\times n}\backslash \Delta_{d,n}$ are exactly the sign conditions will follow from Corollary \ref{cor:conn1}.

\begin{lemma}\label{lemma:identity}For every $G$ graph on $n$ vertices there exists a sign condition $\sigma=\sigma(G)$ such that $W_{G,d}=W_{\sigma, d}.$
Viceversa, for every $\sigma$ there exists $G(\sigma)$ such that $W_{\sigma, d}=W_{G, \sigma}.$
In other words, the signs of the family of quadrics $\{q_{ij}\}_{1\leq i<j\leq n}$ on a point $P$ determine the isomorphism class of $G(P)$ uniquely as a labeled graph.
\end{lemma}
\begin{proof}
Given $P\in\R^{d\times n}\setminus\Delta_{d,n}$ is clear from the definition of geometric graph that $G(P)\cong G$ iff $q_{i,j}<0$ when $(i,j)\in G$ and $q_{i,j}>0$ when $(i,j)\notin G$. From this it follows that $W_{G,d}=W_{\sigma,d}$ where $\sigma_{i,j}=1$ if $(i,j)\in G$ and $\sigma_{i,j}=-1$ if $(i,j)\notin G$.
\end{proof}

\input{betti}
\section{Homotopy groups of the chambers}
We turn our attention now to proving Theorem \ref{thm:zerohomotopy}, and we start by introducing some notation. For every $0\leq r\leq \max\{n, d\}$ let us denote by $(\R^{d\times n})_r$ the set of matrices of rank $r$:
\be (\R^{d\times n})_r=\left\{P\in \R^{d\times n}\,\big|\, \textrm{rk}(P)=r\right\}\subseteq \R^{d\times n}.\ee
When $r=n$ we have that $(\R^{d\times n})_n$ deformation retracts the Stiefel manifold of orthonormal $n$--frames in $\R^n$ (the retraction is given by the Gram-Schmidt procedure; in the case $d=n$ this is simply the deformation retraction of $\mathrm{GL}(n, \R)$ onto $O(n)$). We recall that for $k\leq d-n-1$ this Stiefel manifold is $k$--connected, see \cite[Example 4.53]{Hatcher}:
\be\label{eq:V} \pi_k((\R^{d\times n})_n)=0\quad \quad\textrm{if $k\leq d-n-1$}. \ee
We will need the next elementary lemma.
\begin{lemma}\label{lemma:codim}The complement of $(\R^{d\times n})_n$ can be written as a finite union of smooth submanifolds of codimension at least $d-n+1$.
\end{lemma}
\begin{proof}
Recall that for every $0\leq r\leq \max\{n, d\}$ the codimension of $(\R^{d\times n})_r$ in $\R^{d\times n}$ equals $(n-r)(d-r)$ (see \cite[Chapter 3, Section 2, Exercise 4]{Hirsch}) and, in particular if $r\leq n-1$:
\be \label{eq:codim}\mathrm{codim}_{\R^{d\times n}}(\R^{d\times n})_r\geq d-n+1.\ee
Now, the complement of $(\R^{d\times n})_n$ in $\R^{d\times n}$ is a semialgebraic set that can be written as:
\be  \R^{d\times n}\backslash(\R^{d\times n})_n=\coprod_{r=0}^{n-1}(\R^{d\times n})_r\ee
and it is therefore a semialgebraic set of codimension at least $d-n+1$. 
\end{proof}

We will now prove a sequence of results on the homotopy groups of the chambers. These results will imply Theorem \ref{thm:zerohomotopy}. Since $W_{G, d}$ might not be connected if $d\leq{n\choose 2}{+1}$, part of these results are formulated using the set $[S^k, W_{G, d}]$ of homotopy classes of continuous maps from $S^k$ to $W_{G, d}$, instead of the homotopy group $\pi_k(W_{G, d})$.  As soon as $W_{G, d}$ becomes connected and simply connected, we can endow $[S^k, W_{G, d}]$ with a group structure. To stress this subtlety we will keep both notations.
\begin{proposition}\label{propo:isok}
If $d\geq k+n+1$ the inclusion
\be i:W_{G, d}\cap(\R^{d\times n})_n\longhookrightarrow W_{G, d}\ee
induces a bijection between $[S^k, W_{G, d}\cap(\R^{d\times n})_n]$ and $[S^k, W_{G, d}]$.
\end{proposition}
\begin{proof}We need to prove that the map $i_*:[S^k, W_{G, d}\cap(\R^{d\times n})_n]\to [S^k, W_{G, d}]$ induced by the inclusion is a bijection if $k\leq d-n-1.$

We first prove the surjectivity of $i_*$. Let $f_0:S^k\to W_{G, d}$ be a map representing an element of $[S^k, W_{G, d}].$ Since $W_{G, d}$ is open, up to homotopies, we can assume that the map $f_0$ is smooth.
Moreover, by \cite[Chapter 3, Theorem 2.5]{Hirsch}, the map $f_0$ is homotopic to a map $f_1:S^k\to W_{G, d}$ which is transversal to all the strata of the complement of $(\R^{d\times n})_n$. If now $k<d-n+1$, the transversality condition and Lemma \ref{lemma:codim} imply that the image of $f_1$ does not intersect these strata; therefore $f_1:S^k\to W_{G, d}\cap(\R^{d\times n})_n$ and $i_*$ is surjective. (Notice that for the surjectivity we only need $d\geq k+n$.)

For the injectivity we argue similarly. Let $f_0, f_1:S^k\to W_{G, d}\cap(\R^{d\times n})_n$ be two maps such that $i\circ f_0:S^k\to W_{G,d}$ is homotopic to $i\circ f_1:S^k\to W_{G,d}$. This means that there exists a map $F:S^k\times I\to W_{G, d}$ such that $F(\cdot, 0)=i\circ f_0 $ and $F(\cdot, 1)=i\circ f_1$. Now, we can approximate $F$ with a new map $\widetilde{F}:S^k\times I$ which is smooth, homotopic to $F$ and $C^0$ arbitrarily close to it, and transversal to all the strata of the complement of $(\R^{d\times n})_n$.
By Lemma \ref{lemma:codim}, if $k+1<d-n+1$ this implies that the image of $\widetilde{F}$ does not intersect the complement of $(\R^{d\times n})_n$. In particular we have a homotopy between $\widetilde{F}(\cdot, 0)$ and $\widetilde{F}(\cdot, 1)$ all contained in $W_{G, d}\cap(\R^{d\times n})_n$. On the other hand, since both $F(\cdot, 0)$ and $F(\cdot, 1)$ miss the complement of $(\R^{d\times n})_n$, which is closed, by compactness of $S^k$, any two maps $C^0$ sufficiently close to these maps will be homotopic to them and will also miss this complement. In particular, if $\widetilde{F}$ is sufficiently close to $F$, $F(\cdot, 0)$ is homotopic to $\widetilde{F}(\cdot, 0)$, $F(\cdot, 1)$ is homotopic to $\widetilde{F}(\cdot, 1)$ and these homotopies miss the complement of $(\R^{d\times n})_n$. In this way we have build a homotopy between $f_0$ and $f_1$ already in $W_{G, d}\cap(\R^{d\times n})_n$, i.e. $i_*$ is injective.
\end{proof}
\subsubsection{Some useful maps}\label{sec:useful}We introduce now some useful maps. First recall the ``Gram-Schmidt'' map $\sigma:\R^{d\times n}\to \R^{d\times n}$ which orthonormalizes the columns of a matrix $P\in \R^{d\times n}$ and is defined by:
\be \sigma(P)=P(P^TP)^{-1/2}.\ee
Since the columns of $\sigma(P)$ are orthonormal, it follows that
\be \label{eq:s1}\sigma(P)^T\sigma(P)=\mathbf{1}_{n}.
\ee
Moreover $\sigma(P)\sigma(P)^T$ is the orthogonal projection on the span of the columns of $P$.

Let now $G$ be a geometric graph on $n$ vertices and $d\geq n$. Our first useful map is:
\be\label{eq:alpha}\begin{tikzcd}
{W_{G,n}\times (\R^{d\times n})_{n}} \arrow[rr, "\alpha"] &  & {W_{G, d}} \\
{(Q, P)} \arrow[rr, maps to]                            &  & (\sigma(P)Q)                   
\end{tikzcd}.
\ee
This map essentially takes a $n$--tuple of points in $\R^{n}$ and a linear space of dimension $n$ in $\R^d$ and puts these points isometrically on this linear space. We need to verify that the isomorphism class of the labeled graph is not changed, i.e. that $G(\sigma(P)Q)\simeq G(Q)\simeq G.$ This is true because all the relative distances of the points in $\sigma(P)Q$ are the same as the distances of the points in $Q$. More precisely, denote by $Q=(q_1, \ldots, q_n)$ and by $\sigma(P)Q=(p_1', \ldots, p_n')$, where $p_i'=\sigma(P)q_i$. Then, using \eqref{eq:s1}, we have:
\begin{align} \|p_i'-p_j'\|^2&=\|\sigma(P)(q_i-q_j)\|^2\\
&=(q_i-q_j)^T\sigma(P)^T\sigma(P)(q_i-q_j)\\&=(q_i-q_j)^T(q_i-q_j)\\
&=\|q_i-q_j\|^2.\end{align}
Since the relative distances between the points are the same, and $d\geq n$ it follows by Lemma \ref{lemma:identity} that $G(\sigma(P)Q)\simeq G(Q)$.

A second useful map is:
\be\begin{tikzcd}
{W_{G, d}\cap (\R^{d\times n})_n} \arrow[rr, "\beta"] &  & {W_{G,n}\times (\R^{d\times n})_{n}} \\
P \arrow[rr, maps to]                               &  & {(\sigma(P)^TP, P)}               
\end{tikzcd}.
\ee
This map ``decouples'' the geometric graph into its part in the span of its vertices and the span of the vertices. Also for this map we need to check that its first component has target in $W_{G, n}$. Again this follows from the fact that the mutual distances of the corresponding points are preserved. Denoting by $P=(p_1, \ldots, p_n)$, we have:
\begin{align} \|\sigma(P)^Tp_i-\sigma(P)^Tp_j\|^2&=(p_i-p_j)^T\sigma(P)\sigma(P)^T(p_i-p_j)\\&=(p_i-p_j)^T(p_i-p_j)\\&=\|p_i-p_j\|^2,\end{align}
where we have used the fact that $\sigma(P)\sigma(P)^T$ is the orthogonal projection onto the span of the columns of $P$. The claim follows again from Lemma \ref{lemma:identity}.

Observe that it follows immediately from the definition of the maps $\alpha $ and $\beta$ that:
\be
\label{lemma:ab}\alpha\circ\beta=i:W_{G, d}\cap (\R^{d\times n})_n\longhookrightarrow W_{G, d}.
\ee 
\subsubsection{Stabilization}
\begin{proposition}\label{propo:inj}If $d\geq k+n+1$ the map $j_*:[S^k,W_{G, d}]\longrightarrow[S^k, W_{G, d+1}]$ induced by the inclusion is injective. \end{proposition}
\begin{proof}Let $g_0, g_1:S^k\to W_{G, d}$ be two continuous maps such that $j\circ g_0:S^k\to W_{G, d+1}$  and $j\circ g_1:S^k\to W_{G, d+1}$ are homotopic. Thanks to Proposition \ref{propo:isok} we can assume $g_0$ and $g_1$ to be elements of $[S^k, W_{G,d}\cap (\R^{d\times n })_n]$. We want to prove that $g_0$ and $g_1$ are homotopic.
For a map $f:S^k\to W_{G, d}$, we consider the following commutative diagram of maps:
\be
\begin{tikzcd}
                                                              & {(W_{G, n}\cap(\R^{n^2})_n)\times (\R^{d n})_n} \arrow[dd, "\alpha_1"] \arrow[r, "u", hook] & {(W_{G, n}\cap (\R^{n^2})_n)\times (\R^{(d+1)n})_n} \arrow[dd, "\alpha_2"] \\
                                                              &                                                                                             &                                                                            \\
S^k \arrow[r, "f"] \arrow[ruu, "{\beta_1\circ f=(f_1, f_2)}"] & {W_{G, d}\cap(\R^{dn})_n} \arrow[r, "j", hook] \arrow[uu, "\beta_1"', bend right=49]        & {W_{G,d+1}\cap(\R^{(d + 1)n})_n} \arrow[uu, "\beta_2"', bend right=49]          
\end{tikzcd}\ee
Here the maps $\alpha_1$ and $\alpha_2$ are the restriction of the maps defined in \eqref{eq:alpha} to the set of pairs $(Q,P)$ with $\textrm{rk}(Q)=n$; the values of these maps are in the set of matrices of rank $n$.

Since $\alpha_1\circ\beta_1=\mathrm{id}$, we can write the map $f=$ as:
\be f=\alpha_1\circ(\beta_1\circ f) =\alpha_1\circ (f_1, f_2),\ee
where $(f_1, f_2)$ are the components of $\beta_1\circ f$. 
We apply now the diagram to the map $f=g_0$ and $f=g_1$, writing them as:
\be g_i=\alpha_1\circ(\beta_1\circ g_i) =\alpha_1\circ (g_{i,1}, g_{i, 2}),\quad i=0, 1.\ee
We will prove that both components are homotopic $g_{0,1}\sim g_{1,1}$ and $g_{0, 2}\sim g_{1,2}$, which implies that $g_0$ is homotopic to $g_1$.


Since the map $j\circ g_0$ is homotopic to $j\circ g_1$, then also the first component of $\beta_2\circ j\circ g_0$ is homotopic to the first component of $\beta_2\circ j\circ g_1.$ But the first component of $\beta_2\circ j\circ g_0$ equals $g_{0, 1}$, the first component of $\beta_1\circ g_{0}$, and similarly for the first component of $\beta_2\circ j\circ g_1$. Therefore $g_{0,1}\sim g_{1, 1}.$ 

On the other hand the second components of $\beta_1\circ g_0, \beta_1\circ g_1:S^k\to (\R^{d\times n})_n$ are homotopic simply because $\pi_k((\R^{d\times n})_n)=0$ for $k\leq d-n-1.$ This concludes the proof.
\end{proof}
\begin{proposition}\label{propo:constant}For every $d\geq n$ the inclusion $W_{G, d}\hookrightarrow W_{G, d+n}$ is homotopic to a constant map.
\end{proposition}
Before we give the proof, let us observe that, since we do not know if $W_{G, d}$ is path-connected, then there are \emph{several} constant maps up to homotopy, one for each component; this Proposition tells us that all the maps $[S^k,W_{G,d}]$ are mapped to the \textit{same} constant map in $W_{G,d+n}$. This also tells us that $W_{G,d}$ is contained in just one connected component of $W_{G,d+n}$.
\begin{proof}Since for $d\geq n$ every graph is realizable as a geometric graph, pick $R=(r_1, \ldots, r_n)\in W_{G, n}$ by the previously-cited result of Maehara \cite{Maehara}.

Consider the homotopy $f_t:W_{G, d}\to \R^{(d+n)\times n},$ defined for $t\in [0,1]$ by
\be f_t(P)=\left(\begin{matrix}\sqrt{1-t}P\\ \sqrt{t} R\end{matrix}\right).\ee
With this choice:
\be f_0=i:W_{G, d}\longhookrightarrow W_{G, d+n}\subset \R^{(d+n)\times n}\quad \textrm{and}\quad f_1\equiv \left(\begin{matrix}0\\ R\end{matrix}\right)\in W_{G, d+n}.\ee
We only need to prove that $f_t(W_{G, d})\subseteq W_{G, d+n}$ for all $t\in [0,1].$ To this end let us write
\be f_t(P)=\left(p_1(t), \ldots, p_n(t)\right)=\left(\begin{matrix}\sqrt{1-t}p_1& \cdots&\sqrt{1-t}p_n\\\sqrt{t}r_1&\cdots& \sqrt{t}r_n\end{matrix}\right).\ee
Because of Lemma \ref{lemma:identity}, in order to show that $G(f_t(P))\equiv G$ it is enough to show that the signs of the family quadrics $\{\|p_i-p_j\|^2-1:\R^{(d+n)\times n}\to \R\}$ evaluated on $f_t(P)$ are constants. We have
\be \|p_i(t)-p_j(t)\|^2=(1-t)\|p_i-p_j\|^2+t\|r_i-r_j\|^2\ee
and therefore as 
\be \mathrm{sign}\left(\|p_i-p_j\|^2-1\right)=\mathrm{sign}\left(\|r_i-r_j\|^2-1\right), \ee
it must be the case that 
\be \mathrm{sign}\left(\|p_i(t)-p_j(t)\|^2-1\right)=\mathrm{sign}\left(\|p_i-p_j\|^2-1\right)=\mathrm{sign}\left(\|r_i-r_j\|^2-1\right).\ee
This concludes the proof.
\end{proof}
We are now ready to prove Theorem \ref{thm:zerohomotopy}.
\begin{proof}[Proof of Theorem \ref{thm:zerohomotopy}] We first prove that for $d\geq n+1$ the set $W_{G, d}$ is path connected. 
By Proposition \ref{propo:constant} the map $i_*:[S^0,W_{G, d}]\to [S^0,W_{G, d+n}]$ is the map that sends everything to the class of a constant map. On the other hand this map factors through the sequence of maps induced by the inclusions $W_{G, d}\longhookrightarrow W_{G, d+1}$
\be [S^0,W_{G, d}]\to [S^0,W_{G, d+1}]\to \cdots\to [S^0, W_{G, d+n}].\ee 
Each map in the previous sequence is an injection for $d\geq n+1$ by Proposition \ref{propo:inj}, therefore also $i_*:[S^0,W_{G, d}]\to [S^0,W_{G, d+n}]$ is an injection, and $[S^0, W_{G, d}]$ consists of only one element. Therefore $W_{G, d}$ is path connected.

We prove now that, for $d\geq n+2$, $W_{G, d}$ is also simply connected. By Proposition \ref{propo:constant} the map $i_*:[S^1,W_{G, d}]\to [S^1,W_{G, d+n}]$ is the map that sends everything to the class of a constant map. On the other hand this map factors through the sequence of maps induced by the inclusions $W_{G, d}\longhookrightarrow W_{G, d+1}$
\be [S^1,W_{G, d}]\to [S^1,W_{G, d+1}]\to \cdots\to [S^1, W_{G, d+n-1}]\to [S^1, W_{G, d+n}].\ee 
Each map in the previous sequence is an injection for $d\geq n+2$ by Proposition \ref{propo:inj}, therefore also $i_*:[S^1,W_{G, d}]\to [S^1,W_{G, d+n}]$ is an injection, and $[S^1, W_{G, d}]$ consists of only one element. Recall now that $[S^1, W_{G, d}]$ consists of the set of conjugacy classes in $\pi_1(W_{G, d})$ (we can omit the base point because $W_{G, d}$ is path-connected): the fact that $[S^1, W_{G, d}]$ consists of one element implies that there is only one conjugacy class in $\pi_1(W_{G, d})$, which means that $W_{G, d}$ is simply connected.

Let now $k\geq 2$ and $d\geq k+n+1 $. Since $\pi_1(W_{G, d})=0$ for $d\geq n+2 $, it follows that
\be [S^k, W_{G, d}]=\pi_k(W_{G, d})/\pi_1(W_{G, d})=\pi_k(W_{G, d}),\ee by \cite[Proposition 4A.2]{Hatcher}. By Proposition \ref{propo:constant} the map $i_*:\pi_k(W_{G, d})\to \pi_k(W_{G, d+n})$ is the zero map. On the other hand this map factors through the sequence of maps induced by the inclusions $W_{G, d}\longhookrightarrow W_{G, d+1}$
\be \pi_k(W_{G, d})\to \pi_{k}(W_{G, d+1})\to \cdots\to \pi_k(W_{G, d+n-1})\to \pi_k(W_{G, d+n}).\ee 
Each map in the previous sequence is an injection for $d\geq k+n+1$ by Proposition \ref{propo:inj}, therefore also $i_*:\pi_k(W_{G, d})\to \pi_k(W_{G, d+n})$ is an injection, and $\pi_{k}(W_{G, d})=0.$
\end{proof}
Notice that thanks to this theorem we have the following corollary.
\begin{corollary}\label{cor:conn1}
For $d\geq n+1$, for each labeled graph $G$ on $[n]$ the isomorphism class $W_{G,d}$ is connected.
\end{corollary}

\subsection{The infinite--dimensional case}\label{sec:infinite}
The space $\R^{\infty\times n}$ is a pre-Hilbert space (since it is not complete) with respect to the natural scalar product\footnote{The completion of $\R^{\infty\times n}$ is $(\ell^2(\mathbb{N}))^n=\left\{p=(x_1, x_2, \ldots)\,\bigg|\, \sum_{k=1}^{\infty} x_k^2<\infty\right\}$. }. The notion of geometric graph and discriminant also makes sense in this infinite dimensional space. More precisely, given an element $P=(p_1, \ldots, p_n)\in \R^{\infty\times n}$, we build the graph $G(P)$ whose vertices and edges are defined as in Definition \ref{def:gg}.  The discriminant $\Delta_{\infty, n}$ consists of points $P=(p_1, \ldots, p_n)\in \R^{\infty\times n}$ such that there exists a pair $1\leq i<j\leq n$ with $\|p_i-p_j\|^2=1.$ The chambers are now defined as follows: for a given graph $G$ on $n$ vertices, we set
\be W_{G, \infty}=\left\{P=(p_1, \ldots, p_n)\in \R^{\infty\times n}\backslash \Delta_{\infty, n}\,\big|\, G(P)\simeq G\right\}.\ee
It is easy to see that $W_{G, \infty}$ is the direct limit of the sequence of inclusions in \eqref{eq:inclusionsWg}. In particular, from Theorem \ref{thm:zerohomotopy} we deduce the following.
\begin{theorem}\label{thm:contractible} For every graph $G$, the set $W_{G, \infty}=\varinjlim W_{G, d}$ is contractible.
\end{theorem}
\begin{proof}We first observe that, by Lemma \ref{lemma:identity}, for $d\geq n$ each $W_{G, d}$ is described by a list of quadratic inequalities and therefore it is semialgebraic and it has the homotopy type of a CW--complex. Since $W_{G, \infty}=\varinjlim W_{G, d}$, it follows by \cite[Corollary on pag. 253]{Morsetheory} that also $W_{G, \infty}$ has the homotopy type of a CW--complex. By Whitehead's Theorem (\cite[Theorem 4.5]{Hatcher}), in order to prove that $W_{G, \infty}$ is contractible it is enough to prove that all its homotopy groups are zero.

To this end, let $f:S^k\to W_{G, \infty}$. We need to prove that $f$ is homotopic to a constant map, which implies $\pi_k(W_{G, \infty})=0.$
We will prove that $f$ is homotopic to a map $g:S^k\to W_{G, d}$ for some large enough $d$. Then we can use Theorem \ref{thm:zerohomotopy} to conclude that $g$ is homotopic to a constant map, and so is $f$.

We first observe that, since $S^k$ is compact and $W_{G, \infty}$ is open, there exists $\epsilon>0$ such that
\be\label{eq:cont} \bigcup_{\theta \in S^k}B(f(\theta), \epsilon)\subset W_{G, \infty}.\ee
In fact, for every $\theta\in S^k$ there exists $\epsilon(\theta)>0$ such that $B(f(\theta), \epsilon(\theta))\subset W_{G,\infty}$ (because $W_{G, \infty}$ is open), and the existence of a uniform $\epsilon>0$ follows by compactness of $f(S^k).$

The inclusion \eqref{eq:cont} implies that if $g:S^k\to \R^{\infty\times n}$ is a continuous map with 
\be \label{eq:sup}\sup_{\theta\in S^k}\|f(\theta)-g(\theta)\|\leq \epsilon,\ee then $f_t=(1-t)f+tg$
is a homotopy between $f$ and $g$ with $f_t(S^k)\subset W_{G, \infty}.$ Our scope is to build a map $g:S^k\to W_{G,\infty}$ satisfying \eqref{eq:sup} and such that $g(S^k)\subset W_{G, d}$ with $d\geq k+n+1$. 

Given a point $P=(p_1, \ldots, p_n)\in \R^{\infty\times n}$ let us denote by $P^{\leq d}$ the point given by the same $n$--tuple of sequences as in $P$, but where we have set to zero all the terms of the sequences past the $d$--th one. In other words $P^{\leq d}$ is the inclusion in $\R^{\infty\times n}$ of the projection of $P$ to $\R^{d\times n}$. The definition of $\R^{\infty \times n}$ implies that for every $P$ there exists $d\geq 1$ such that $P=P^{\leq d}$. This means that for every $\theta\in S^k$ there exists $d(\theta)$ such that $f(\theta)=f(\theta)^{\leq d(\theta)}.$ 

We claim now that there exists $d_f\geq 1$ such that for all $\theta\in S^k$ we have 
\be\label{eq:sup2} \|f(\theta)-f(\theta)^{\leq d_f}\|\leq \epsilon.\ee If this was false, then there would be a sequence $\{\theta_n\}_{n\geq 1}\subset S^k$ such that for all $n\geq 1$ the following inequality would be satisfied:
\be \|f(\theta_n)-f(\theta_n)^{\leq n}\|>\epsilon.\ee
Up to subsequences, we may assume that $\theta_n\to \theta$ and $f(\theta_n)\to f(\theta)$. Let $d(\theta)>0$ be such that $f(\theta)=f(\theta)^{\leq d(\theta)}$ and pick $\theta_n$ with $n\geq d(\theta)$ and such that $\|f(\theta)-f(\theta_n)\|\leq  \epsilon/2.$ Then
\begin{align} \epsilon&<\|f(\theta_n)-f(\theta_n)^{\leq n}\|\\
&\leq\left(\|f(\theta_n)-f(\theta_n)^{\leq n}\|^2+\|f(\theta)^{\leq n}-f(\theta_n)^{\leq n}\|^2\right)^{1/2}\\
&\label{eq:bra1}=\|f(\theta)-f(\theta_n)\|\leq \epsilon/2.
\end{align}
The last line follows from the fact that both $f(\theta)$ and $f(\theta_n)$ can be written as $f(\theta)=(f(\theta)^{\leq n})+ (f(\theta)-f(\theta)^{\leq n})$ and $f(\theta_n)=(f(\theta_n)^{\leq n})+ (f(\theta_n)-f(\theta_n)^{\leq n})$ with the two summands belonging in both cases to the same orthogonal subspaces; this tells that
\be\label{eq:bra2} f(\theta)-f(\theta_n)=\left[f(\theta)^{\leq n}-f(\theta_n)^{\leq n}\right]+\left[(f(\theta)-f(\theta)^{\leq n})-(f(\theta_n)-f(\theta_n)^{\leq n}))\right],\ee
with the two summands in the square brackets belonging to orthogonal subspaces. Moreover, since $f(\theta)=f(\theta)^{\leq n}$ by construction, the summand in the second square brackets of \eqref{eq:bra2} simply equals  $f(\theta_n)-f(\theta_n)^{\leq n}$ and this gives \eqref{eq:bra1}.
Equation \eqref{eq:bra1} gives a contradiction, and therefore such a $d_f\geq 1$ exists. Define now $d=\max\{d_f, k+n+1\}$ and set:
\be g(\theta)=f(\theta)^{\leq d}.\ee
The condition \eqref{eq:sup2} gives now $\eqref{eq:sup}$, which is what we wanted.
\end{proof}

%% file: betti.tex
\subsection{Betti numbers of the chambers}In this section we study the asymptotic distribution of the Betti numbers of the chambers. Before giving the main result, we will need some intermediate steps.

\subsubsection{Some preliminary reductions}Using Lemma \ref{lemma:identity} we can immediately switch from the graph labelling to the sign condition one and given $G$ there exists $\sigma$ such that $W_{G, d}=W_{\sigma, d}$. In this way we describe the chamber we are interested in with a system of quadratic inequalities, and we will take advantage of this description.

Our first step is to replace $W_{\sigma, d}$ with another set which has the same homology and which is compact. To start with, we have $W_{\sigma, d}=\{\textrm{sign}(q_{ij})=\sigma_{i,j}, \forall 1\leq i<j\leq n\},$ with the quadrics $q_{ij}:\R^{d\times n}\to \R$ defined above. Since $\sigma$ is fixed, it will be convenient for us to define the new quadrics:
\be s_{ij}=\sigma_{ij}q_{ij}\quad \textrm{and} \quad h_{ij}(x, z)=\sigma_{ij}(\|x_i-x_j\|^2-z^2).\ee
We set $N=nd$ and $k={n\choose 2}$ and for every $\epsilon>0$ consider the set:
\be W_{\sigma, d}(\epsilon)=\{[x]\in \mathbb{R}\mathrm{P}^N\, |\,h_{ij}(x, z)\geq \epsilon z^2\quad \forall 1\leq i<j\leq n, \quad \|x\|^2\leq \epsilon^{-1}z^2\}\subseteq \mathbb{R}\mathrm{P}^N.\ee
Notice that $W_{\sigma, d}(\epsilon)\cap \{z=0\}=\emptyset$, because if $z=0$ then the last inequality defining $W_{\sigma, d}(\epsilon)$ forces $x=0$. Therefore, in the affine chart $\{z=1\}$ the set $W_{\sigma, d}(\epsilon)$ can be described as
\be W_{\sigma, d}(\epsilon)=\{s_{ij}(x)\geq \epsilon\quad \forall 1\leq i<j\leq n, \quad \|x\|^2\leq \epsilon^{-1}\}\ee
and can be identified with a subset of $W_{\sigma, d}.$

\begin{proposition}\label{proposition:ed}For every $d>0$ there exists $\epsilon(d)$ such that  for all $\epsilon<\epsilon(d)$ the inclusion
\be W_{\sigma, d}(\epsilon)\longhookrightarrow W_{\sigma, d}\ee
induces an isomorphism on the homology groups.
\end{proposition}
\begin{proof}Observe that $\{W_{\sigma, d}(\epsilon)\}_{\epsilon>0}$ is an increasing family of compact sets such that:
\be \bigcup_{\epsilon>0}W_{\sigma, d}(\epsilon)=W_{\sigma, d}.\ee
Therefore:
\be\label{eq:dirlim} H_*(W_{\sigma, d})=\varinjlim \{H_*(W_{\sigma, d}(\epsilon))\}.\ee
On the other hand, consider the semialgebraic set
\be S=\{(x, \epsilon)\,\in \R^{d\times n}\times \R\,|\,s_{ij}(x)\geq \epsilon\quad \forall 1\leq i<j\leq n, \quad \epsilon\|x\|^2\leq 1\}\ee
together with the semialgebraic map $f:S\to \R$ given by $(x, \epsilon)\mapsto \epsilon$. For $\epsilon>0$ we have that $W_{\sigma, d}(\epsilon)=f^{-1}(\epsilon)$ and,  by Semialgebraic Triviality, there exists $\epsilon(d)>0$ such that for all $\epsilon_1<\epsilon_2<\epsilon(d)$ the inclusion $W_{\sigma, d}(\epsilon_1)\longhookrightarrow W_{\sigma, d}(\epsilon_2)$ is a homotopy equivalence and the direct limit \eqref{eq:dirlim} stabilizes. This proves the statement.
\end{proof}

The set $W_{\sigma, d}(\epsilon)$ is now compact and for $\epsilon<\epsilon(d)$ has the same homology of $W_{\sigma, d}$. For technical reasons, this is not yet the set we will work with. Instead we will work with its double cover $V_{\sigma, d}(\epsilon)\subset S^{N}$:
\be V_{\sigma, d}(\epsilon)=\{x\in S^N\, |\,h_{ij}(x, z)\geq \epsilon z^2\quad \forall 1\leq i<j\leq n, \quad \|x\|^2\leq \epsilon^{-1}z^2\}\subset S^{N}. \ee
This will not be an obstacle for computing the Betti numbers of $W_{\sigma, d}(\epsilon)$, because of next lemma.
\begin{lemma}\label{lemma:sphere}For every $\epsilon>0$ the set $V_{d, \epsilon}(\epsilon)\subset S^N$ consists of two disjoint copies of $W_{\sigma, d}(\epsilon)$. In particular for all $k\geq 0$\be b_k(W_{\sigma, d}(\epsilon))=\frac{1}{2}b_k(V_{\sigma, d}(\epsilon)).\ee
\end{lemma}
\begin{proof}Let $\{z=0\}\simeq S^{N-1}$ be the equator in $S^N$ and observe that $\{z=0\}\cap V_{\sigma, d}(\epsilon)=\emptyset.$
This implies that 
\be V_{\sigma, d}(\epsilon)=\left(V_{\sigma, d}(\epsilon)\cap \{z>0\}\right) \sqcup\left(V_{\sigma, d}(\epsilon)\cap \{z<0\}\right).\ee
The involution $(x,z)\mapsto (-x-z)$ on the sphere $S^N$ restricts to a homeomorphism between $V_{\sigma, d}(\epsilon)\cap \{z>0\} $ and $V_{\sigma, d}(\epsilon)\cap \{z<0\}$. Each of these sets is homeomoprhic to its projection to the projective space $\mathbb{R}\mathrm{P}^N$, which is the set $W_{\sigma,d}(\epsilon).$ This concludes the proof.
\end{proof}

\subsubsection{Systems of quadratic inequalities}
The set $V_{\sigma, d}(\epsilon)$ defined above is the set of solutions of a system of quadratic inequalities,  and we will now use the spectral sequence from Section \ref{sec:spectral} for computing its Betti numbers with $\mathbb{Z}_2$ coefficients. 

In order to reduce to the  framework of Section \ref{sec:spectral}, let us introduce the homogeneous quadrics $h_{ij,\epsilon}, h_{0,\epsilon}:\R^{N+1}\to \R^{k+1}$ defined for all $1\leq i<j\leq n$ by:
\be h_{ij,\epsilon}(x, z)=\sigma_{ij}\|x_i-x_j\|^2-\sigma_{ij}z^2-\epsilon z^2 \quad \textrm{and} \quad h_{0,\epsilon}(x, z)=\|x\|^2-\epsilon^{-1} z^2.\ee
These quadrics can be put as the components of a quadratic \emph{map} defined by
\be h_\epsilon=(h_{0,\epsilon},h_{1,\epsilon}, h_{2, \epsilon}, \ldots, h_{k, \epsilon}):\R^{N+1}\to \R^{k+1},\ee 
where we are using the identification of sets of indices $\{1, 2, \ldots, k\}=\{(1,2),(1,3), \ldots, (n-1,n)\}.$
Inside the space $\R^{k+1}$ we can consider the closed convex cone
\be K=\{y_0\leq 0, y_{1}\geq 0 , \ldots, y_k\geq 0\},\ee
so that our original set can be written as
\be V_{\sigma, d}(\epsilon)=h_\epsilon^{-1}(K).\ee
In this case the set $\Omega\subset S^{k}$ is the set 
\be \Omega=\{(\omega_0, \ldots, \omega_k)\in S^k\,|\, \omega_0\geq 0, \omega_1\leq 0, \ldots, \omega_k\leq 0\}.\ee
For every point $\omega=(\omega_0, \ldots, \omega_k)$ we can consider the quadratic form $\omega h_\epsilon$ defined by:
\be \omega h_\epsilon=\omega_0h_{0, \epsilon}+\cdots +\omega_kh_{k, \epsilon}.\ee
Using this notation, for every $ j\geq 0$ we define the sets:
\be \Omega^j(\epsilon)=\{(\omega_0, \ldots, \omega_k)\in \Omega\,|\, \mathrm{ind}^+(\omega H_\epsilon)\geq j\}.\ee
These are just the sets $\Omega^j$ defined in Section \ref{sec:spectral}, in the case of the quadratic map $h_\epsilon$.

For every $1\leq i<j\leq n$ let us also denote by $U_{ij}\in \mathrm{Sym}(n, \R)$ the symmetric matrix representing the quadratic form $u_{ij}:\R^n\to \R$ defined by:
\be u_{ij}(t_1, \ldots, t_n)=\sigma_{ij}(t_i-t_j)^2.\ee
Then, if $H_{ij}\in \mathrm{Sym}(dn, \R)$ is the matrix representing the quadratic form $x\mapsto \sigma_{ij}\|x_i-x_j\|^2$, we have:
\be H_{ij}=U_{ij}\otimes \mathbf{1}_d\ee

\begin{lemma}\label{lemma:indexfunction}The index function $\mathrm{ind}^+:\Omega\to \mathbb{N}$ for our family of quadrics can be written as:
\be \mathrm{ind}^+(\omega H_\epsilon)=d\cdot \mathrm{ind}_1^+(\omega)+\mathrm{ind}_{0, \epsilon}^+(\omega),\ee
where 
\be \mathrm{ind}^+_1(\omega)=\mathrm{ind}^+\left(\omega_0\mathbf{1}_n+\sum_{i<j}\omega_{ij}U_{ij}\right)\ee
and 
\be \mathrm{ind}_{0, \epsilon}^+(\omega)=\mathrm{ind}^{+}\left(-\frac{\omega_0}{\epsilon}-\sum_{i<j}\omega_{ij}(\sigma_{ij}+\epsilon)\right).\ee
\end{lemma}
Before giving the proof, observe that none of the functions $\mathrm{ind}^+_1,\mathrm{ind}^{+}_{0, \epsilon}:\Omega\to \mathbb{N}$ depends on $d$ and that $\mathrm{ind}^+_1$ does not even depend on $\epsilon.$ 
\begin{proof}Observe that, for $\omega=(\omega_0, \omega_{ij})\in \Omega$, the matrix $\omega H_\epsilon$ is a block matrix:
\be \omega H_\epsilon=\left(\begin{array}{c|ccc} -\frac{\omega_0}{\epsilon}-\sum_{i<j}\omega_{ij}(\sigma_{ij}+\epsilon)& 0 & \cdots & 0 \\\hline 0 &  &  &  \\\vdots &  &\omega_0\mathbf{1}_{dn}+\sum_{i<j}\omega_{ij}H_{ij}  &  \\0 &  &  & \end{array}\right)
\ee
and in particular:
\be \mathrm{ind}^+(\omega H_\epsilon)=\mathrm{ind}^+\left(-\frac{\omega_0}{\epsilon}-\sum_{i<j}\omega_{ij}(\sigma_{ij}+\epsilon)\right)+\mathrm{ind}^+\left(\omega_0\mathbf{1}_{dn}+\sum_{i<j}\omega_{ij}H_{ij}\right).\ee
The matrix $\omega_0\mathbf{1}_{dn}+\sum_{i<j}\omega_{ij}H_{ij}$ is a tensor product of matrices:
\be \omega_0\mathbf{1}_{dn}+\sum_{i<j}\omega_{ij}H_{ij}=\left(\omega_0\mathbf{1}_{d}+\sum_{i<j}\omega_{ij}U_{ij}\right)\otimes \mathbf{1}_n.\ee
If a matrix $Q\in \mathrm{Sym}(n, \R)$ has eigenvalues $\lambda_1(Q)\geq \cdots\geq \lambda_n(Q)$ (possibly with repetitions), the matrix $Q\otimes \mathbf{1}_d$ has eigenvalues:
\be \lambda_{i,j}(Q\otimes \mathbf{1}_d)=\lambda_i(Q)\quad i=1, \ldots, n,\, j=1, \ldots, d.\ee
In particular
\be \mathrm{ind}^{+}(Q\otimes \mathbf{1}_d)=d\cdot \mathrm{ind}^+(Q),\ee
and the result now follows.\end{proof}
\begin{corollary}\label{corollary:contractible}For $d\geq  n+1$ the set $\Omega^{nd}(\epsilon)$ is contractible and $\Omega^{nd+1}(\epsilon)$ is empty.
\end{corollary}
\begin{proof}
Let us first show that $\Omega^{nd+1}(\epsilon)=\emptyset$. To this end consider the set:
\be B(\epsilon)=\{(\omega, [x])\in \Omega\times \mathbb{R}\mathrm{P}^{N}\,|\, \omega h_\epsilon(x)\geq 0\}.\ee
By \cite[Lemma 24]{AgrachevLerario} the projection $\pi=p_2|_{B(\epsilon)}$ on the second factor gives a homotopy equivalence between $B(\epsilon)$ and its image 
\be \pi(B(\epsilon))=\mathbb{R}\mathrm{P}^N\backslash W_{\sigma, d}(\epsilon).\ee
Since $W_{\sigma, d}(\epsilon)$ is nonempty, we know that
\be\label{eq:cont} \pi(B(\epsilon))\neq \mathbb{R}\mathrm{P}^N.\ee If now there was $\omega\in \Omega$ such that $\mathrm{ind}^+(\omega)=N+1$, then $\omega h_\epsilon>0$ and $\{\omega\}\times \mathbb{R}\mathrm{P}^N\subset
B(\epsilon)$. This would imply $\pi(B(\epsilon))=\mathbb{R}\mathrm{P}^N$, which contradicts \eqref{eq:cont}.

Let us now prove that $\Omega^{nd}(\epsilon)$ is contractible. For $d\geq n+1$, since $\Omega^{nd+1}(\epsilon)=\emptyset$, then the set $\Omega^{nd}(\epsilon)$ can be described as:
\begin{align}
\Omega^{nd}(\epsilon)&=\{\mathrm{ind}^{+}=nd\}\\
&=\{d\cdot \mathrm{ind}_{1}^++\mathrm{ind}_{0, \epsilon}\geq nd\}\\
&=\{\mathrm{ind}_{1}^+=n\}\cap \{\mathrm{ind}_{0, \epsilon}=0\}.
\end{align}
Observe that the point $\omega=(1, 0, \ldots, 0)\in \Omega$ belongs to both the sets $\{\mathrm{ind}_{1}^+=n\}$ and $\{\mathrm{ind}_{0, \epsilon}=0\}$, and their intersection is nonempty. 

Now, $\{\mathrm{ind}_{1}^+=n\}$ and $\{\mathrm{ind}_{0, \epsilon}=0\}$ are obtained by intersecting a convex set in $\R^{k+1}$ with $\Omega\cap S^{k}$, as they coincide with the set of the points where the linear families of symmetric matrices $\omega_0\mathbf{1}_{n}+\sum_{i<j}\omega_{ij}U_{ij}$ and $-\frac{\omega_0}{\epsilon}-\sum_{i<j}\omega_{ij}(\sigma_{ij}+\epsilon)$ are positive definite (respectively negative semidefinite). In other words, $\{\mathrm{ind}_{1}^+=n\}$ is the preimage of the positive definite cone under the linear map $\omega\mapsto \omega_0\mathbf{1}_{d}+\sum_{i<j}\omega_{ij}U_{ij}$ and $\{\mathrm{ind}_{0, \epsilon}=0\}$ is the preimage of the negative semidefinite cone under the linear map $\omega\mapsto-\frac{\omega_0}{\epsilon}-\sum_{i<j}\omega_{ij}(\sigma_{ij}+\epsilon)$. 

Therefore $\Omega^{nd}(\epsilon)$ is the intersection in $S^{k}\cap \Omega$ of convex sets, and being $\Omega$ itself also convex this intersection is contractible.
\end{proof}

Recalling the notation of Section \ref{sec:spectral}, but making it dependent on $\epsilon$, we have the vector bundle $P^{j}(\epsilon)\subseteq \Omega^{j}(\epsilon)\backslash \Omega^{j+1}(\epsilon)\times \R^{N+1}$
\be \label{eq:bundle}
\begin{tikzcd}
\mathbb{R}^j \arrow[r, hook] & P^j(\epsilon) \arrow[d] \\
                             &   \Omega^{j}(\epsilon)\backslash \Omega^{j+1}(\epsilon)                  
\end{tikzcd}
\ee
whose fiber over a point $\omega$ is the positive eigenspace of $\omega H_\epsilon.$ As above, this bundle is the restriction of a bundle over the set
\be D_j(\epsilon)=\{\omega\,|\, \lambda_j(\omega H_\epsilon)\neq \lambda_{j+1}(\omega H_\epsilon)\}\ee
(i.e. the set where the $j$th eigenvalue of $\omega H_\epsilon$ is distinct from the $(j+1)$st). We still denote this bundle by $P_j(\epsilon)\subset D_{j}(\epsilon)\times \R^{N+1}$: 
\be \label{eq:bundle2}
\begin{tikzcd}
\mathbb{R}^j \arrow[r, hook] & P^j(\epsilon) \arrow[d] \\
                             &   D_j(\epsilon)                 
\end{tikzcd}
\ee 
Here the fiber over a point $\omega\in D_{j}(\epsilon)$ consists of the eigenspace of $\omega H_\epsilon$ associated to the first $j$ eigenvalues. We denote by
\be \label{eq:sw}\nu_j(\epsilon)\in H^{1}(D_j(\epsilon))\ee
the first Stiefel-Whitney class of this bundle.

Restating Theorem \ref{thm:spectralbasis} in this setting, we get the following.

\begin{theorem}\label{thm:spectral}There exists a cohomology spectral sequence $(E_r(\epsilon), d_r(\epsilon)_{r\geq 1})$ converging to $H^{*}(S^N\backslash V_{\sigma, d}(\epsilon);\mathbb{Z}_2)$ such that:
\begin{enumerate}
\item the second page of the spectral sequence is given, for $j>0$, by
\be E_2^{i,j}(\epsilon)=
 H^i(\Omega^{j+1}(\epsilon), \Omega^{j+2}(\epsilon);\mathbb{Z}_2).\ee
 For $j=0$, the elements of the second page of the spectral sequence fit into a long exact sequence:
\be\label{eq:exact} \cdots \rightarrow H^i(\Omega^{1}(\epsilon);\mathbb{Z}_2)\rightarrow E_{2}^{i, 0}(\epsilon)\rightarrow H^i(\Omega^{1}(\epsilon), \Omega^{2}(\epsilon);\mathbb{Z}_2)\stackrel{(\cdot) \smile \nu_1(\epsilon)}{\longrightarrow} H^{i+1}(\Omega^{1}(\epsilon);\mathbb{Z}_2)\rightarrow\cdots.\ee
\item for $j\geq 1$ the second differential $d_2^{i,j}(\epsilon):H^i(\Omega^{j+1}(\epsilon), \Omega^{j+2}(\epsilon))\to H^{i+2}(\Omega^{j}(\epsilon), \Omega^{j+1}(\epsilon))$ is given by 
\be d_2^{i,j}(\epsilon)\xi=\partial(\xi\smallsmile \nu_{j+1}(\epsilon))+\partial \xi\smallsmile \nu_{j}(\epsilon).\ee
\end{enumerate} 
\end{theorem}

\begin{remark}As explained in \cite[Introduction]{AgrachevLerario}, the second differential only depends on the restriction of $\nu_j(\epsilon)$ to the set $\Omega^{j}(\epsilon)\backslash\Omega^{j+1}(\epsilon)$.
\end{remark}

\begin{remark}In the previous spectral sequence the coefficient group for the various cohomologies is the field $\mathbb{Z}_2$. There is an analogous spectral sequence for coefficients in $\mathbb{Z}$, but the description of its differentials is less clear. 
\end{remark}

\subsubsection{The analysis of the spectral sequence and its asymptotic structure}
We will start by proving the following proposition, which deals with the stabilization of entries of the second page of the spectral sequence of Theorem \ref{thm:spectral}.
\begin{proposition}\label{propo:structure}There exist semialgebraic topological spaces
\be \Omega=A_0\supseteq B_0\supseteq A_1\supseteq B_1\supseteq \cdots \supseteq A_{n}\supseteq B_n=\emptyset,\ee vector spaces $N^{0, 0}, \ldots, N^{k, 0}$
and $\epsilon_1>0$ such that for all $\epsilon\leq\epsilon_1$ the second page of the spectral sequence of Theorem \ref{thm:spectral} has the following structure
\be\label{eq:structure} E_{2}^{i,j}(\epsilon)\simeq \left\{\begin{matrix} 
H^i(B_\ell, A_{\ell+1})&\textrm{if $j=\ell d$}\\[3pt]
H^i(A_\ell, B_{\ell})&\textrm{if $j=\ell d-1$}\\[3pt]
N^{i,0}&\textrm{if $j=0$}\\[3pt]
0&\textrm{otherwise}\end{matrix}\right.
\ee
\end{proposition}
\begin{proof}Observe first that the second page of the spectral sequence is zero in the region $\{(i,j)\,|\, i\geq k+1, j>0\}$, because all the sets $\Omega^{j}(\epsilon)$ are semialgebraic and of dimension at most $k$ (since they are contained in $\Omega\subset S^{k}$). The $j = 0$ row of the spectral sequence is also zero for $i\geq k+2$, since for the same reason all the groups in the exact sequence in \eqref{eq:exact}  are zero.

Observe now that Lemma \ref{lemma:indexfunction} implies that the only possible values of the function $\mathrm{ind}^+:\Omega \to \mathbb{N}$ are $0, 1, d, d+1, \ldots, nd, nd+1$ and in particular:
\begin{align} \Omega=\Omega^0(\epsilon)\supseteq \Omega^1(\epsilon)\supseteq \Omega^2(\epsilon)=\Omega^3(\epsilon)=\cdots=&\Omega^d(\epsilon)\supseteq \Omega^{d+1}(\epsilon)\supseteq\Omega^{d+2}(\epsilon)=\Omega^{d+3}(\epsilon)=\cdots\\
\label{eq:inclusions}\cdots =\Omega^{nd-1}(\epsilon)=&\Omega^{nd}(\epsilon)\supseteq\Omega^{nd+1}(\epsilon)\supseteq\emptyset.
\end{align}
In particular, for every $\ell=0, \ldots, n$, we deduce the vanishing of the homology of all the relative pairs:
\be H^*(\Omega^{d\ell+2}(\epsilon), \Omega^{d\ell+3}(\epsilon))=\cdots=H^*(\Omega^{(\ell+1) d-1}(\epsilon), \Omega^{(\ell+1) d}(\epsilon))=0.\ee
This proves the ``otherwise'' part of the claim in \eqref{eq:structure}.

We now defined the sets $A_\ell(\epsilon)=\{\textrm{ind}^+\geq d\ell\}$ and $B_\ell(\epsilon)=\{\textrm{ind}^+\geq d\ell+1\}$ and observe that:
\begin{align} A_\ell(\epsilon)&=\{\textrm{ind}_1^+\geq \ell\}\\
B_{\ell}(\epsilon)&=\left(\{\textrm{ind}_1^+\geq \ell\}\cap \{\textrm{ind}_{0, \epsilon}^+=1\}\right)\cup\{\mathrm{ind}_1^{+}\geq \ell+1\},
\end{align}
where the index functions $\textrm{ind}_{0, \epsilon}^+, \textrm{ind}_1^+:\Omega\to \mathbb{N}$ are defined in Lemma \ref{lemma:indexfunction}. Since $\textrm{ind}^+_{1}$ does not depend on $d$ nor on $\epsilon$ and $\textrm{ind}_{0, \epsilon}$ does not depend on $d$, by semialgebraic triviality it follows that there exists $\epsilon_1>0$ such that the homotopy of the sequence of inclusions
\be \Omega=A_0(\epsilon)\supseteq B_0(\epsilon)\supseteq A_1(\epsilon)\supseteq B_1(\epsilon)\supseteq \cdots \supseteq A_{n}(\epsilon)\supseteq B_n(\epsilon)=\emptyset\ee
stabilizes for $\epsilon\leq \epsilon_1.$ 

We define $A_\ell=A_\ell(\epsilon_1)$ and $B_\ell=B_\ell(\epsilon_1).$ With this notation we have that the sequence of inclusions \eqref{eq:inclusions} for $\epsilon\leq \epsilon_0$ becomes (up to natural homotopy equivalences):
\begin{align} \Omega=A_0\supseteq B_0\supseteq A_1=A_1=\cdots=&A_1\supseteq B_1\supseteq A_2=A_2=\cdots\\
\cdots =A_n=&A_n\supseteq B_n\supseteq\emptyset.
\end{align}
This proves the statement for the term $E_2^{i,j}(\epsilon)$ of the spectral sequence with $j=\ell d-1,\ell d.$ 

In the case $j=0$, we observe that the dimension of $E_2^{i, 0}(\epsilon)$ is determined by the exact sequence:
\begin{align} 0\rightarrow \mathrm{ker}\rightarrow H^{i-1}(\Omega^1(\epsilon),\Omega^{2}(\epsilon))\rightarrow &H^{i}(\Omega^1(\epsilon))\rightarrow E_2^{i, 0}(\epsilon)\rightarrow\\&\rightarrow  H^i(\Omega^{1}(\epsilon), \Omega^2(\epsilon))\rightarrow H^{i+1}(\Omega^1(\epsilon))\rightarrow \mathrm{coker}\rightarrow 0,\end{align}
where $\mathrm{ker}$ and $\mathrm{coker}$ refer to the map $x\mapsto x\smile \nu_1(\epsilon).$ The homotopy of the first, the third and the fourth element of the above sequence stabilizes for $\epsilon\leq \epsilon_1$; moreover
(possibly choosing a smaller $\epsilon_1$) also the homotopy of the bundle $P^1(\epsilon)\to D_1(\epsilon)$ from \eqref{eq:bundle2} stabilizes for $\epsilon\leq \epsilon_1$ and therefore the map $x\mapsto x\smile \nu_1(\epsilon))$ stabilizes as well, and consequently the ranks of $\mathrm{ker}$ and $\mathrm{coker}$ stabilize. This gives the stabilization of $\mathrm{dim}_{\mathbb{Z}_2}(E_2^{i,0}(\epsilon))$ to a finite number for $\epsilon\leq \epsilon_1$. We set:
\be N^{i,0}:=\mathbb{Z}_2^{\mathrm{dim}_{\mathbb{Z}_2}(E_2^{i,0}(\epsilon))}\quad \forall\epsilon\leq\epsilon_1. \ee
This concludes the proof.
\end{proof}

Next we deal with the stabilization of the second differential. 

\begin{proposition}\label{proposition:differential}The second differential of the spectral sequence \eqref{eq:structure} is zero.
\end{proposition}

\begin{proof}
Observe that the only possible nonzero differential of the spectral sequence is (for $d\geq 2$)
\be d_2^{*, \ell d}(\epsilon):E_2^{*, \ell d}(\epsilon)\to E_2^{*+2, \ell d-1}(\epsilon).\ee
Let us recall that we have defined $\omega H_{\epsilon}=\omega q_1+\omega q_2$ where $\omega q_1=(\omega_0\mathbf{1}_{dn}+\sum_{i<j}\omega_{ij}H_{ij}) $ and $\omega q_2=\left(-\frac{\omega_0}{\epsilon}-\sum_{i<j}\omega_{ij}(\sigma_{ij}+\epsilon)\right)z^2$. We introduce the vector bundles
\be \begin{tikzcd}
\R^{d\times l} \arrow[r, hook] & N_{\ell d} \arrow[d]        \\
                           & \mathcal{D}_{\ell d}^{1}
\end{tikzcd}\quad \quad \begin{tikzcd}
\R \arrow[r, hook] & E(\epsilon) \arrow[d]                \\
                   & {\Omega}
\end{tikzcd}\ee
where $\mathcal{D}_{\ell d}^{1}:=\{\omega\in\Omega\mid\lambda_{\ell d}(\omega q_1)\neq\lambda_{ld+1}(\omega q_1)\}$ is given by  $N_{\ell d}\subset \mathcal{D}_{\ell d}^{1}\times \R^{d n}$ is the bundle of the eigenspace of the first $\ell d$ eigenvalues of the upper-left block of $\omega H_\epsilon$ and the bundle $E(\epsilon)$ associates to every point of $\Omega$ the unique eigenvector of $\omega q_2$. 

Observe that $\mathcal{D}_{\ell d}(\epsilon)\subset\mathcal{D}_{ \ell d}^{1}$ and also $\mathcal{D}_{\ell d+1}(\epsilon)\subset\mathcal{D}_{\ell d}^{1}$. 

The vector bundle $P_{\ell d}(\epsilon)$ from \eqref{eq:bundle2} for $j=\ell d$ has the property that
\be P_{\ell d}(\epsilon)=N_{\ell d}|_{\mathcal{D}_{\ell d}(\epsilon)}\ee
when $N_{\ell d}$ is thought as a sub-bundle of $\mathcal{D}_{\ell d}^{1}\times\R^{nd+1}$.
When $j=\ell d+1$ we have 
\[P_{\ell d+1}(\epsilon)=N_{\ell d}|_{\mathcal{D}_{\ell d+1}(\epsilon)}\oplus E(\epsilon)|_{\mathcal{D}_{\ell d+1}(\epsilon)} \]
because the quadratic form $\omega q$ has two diagonal blocks. In particular, denoting by $\gamma_{\ell d}$ and by $\eta(\epsilon)$  the first Stiefel--Whitney classe of $N_{\ell d}$ and $E(\epsilon)$ respectively, by naturality of characteristic classes we have the following identities:
\be \nu_{\ell d}=\gamma_{\ell d}|_{\mathcal{D}_{\ell d}(\epsilon)}\quad \textrm{and}\quad \nu_{\ell d+1}=\gamma_{\ell d}|_{\mathcal{D}_{\ell d+1}(\epsilon)}+\eta(\epsilon)|_{\mathcal{D}_{\ell d+1}(\epsilon)}.\ee
Notice that both $\nu_{\ell d}$ and $\nu_{\ell d +1}$ contain the restriction of the same class $\gamma_{\ell d}$ as a summand.
Thanks to Theorem \ref{thm:spectral}, the second differential  $d_2^{*, \ell d}(\epsilon)$ can be written as
\begin{align}
d_2^{*, \ell d}(\epsilon)\xi&=\partial(\xi\smallsmile \nu_{\ell d+1})+\partial \xi\smallsmile \nu_{\ell d}\\
&=\partial\left(\xi\smallsmile\left(\gamma_{\ell d}|_{\mathcal{D}_{\ell d+1}(\epsilon)}+\eta(\epsilon)|_{\mathcal{D}_{\ell d+1}(\epsilon)}\right)\right)+\partial \xi\smallsmile \gamma_{\ell d}|_{\mathcal{D}_{\ell d}(\epsilon)}\\
&=\partial\left(\xi \smallsmile \eta(\epsilon)|_{\mathcal{D}_{\ell d+1}(\epsilon)}\right)=\partial\left(\xi \smallsmile \eta(\epsilon)\right),
\end{align}
where we have used Remark \ref{rem:delta} (taking $(Z,X,Y)=(\Omega^{\ell d}(\epsilon),\Omega^{\ell d+1}(\epsilon),\Omega^{\ell d+2}(\epsilon))$ and $(\tilde{X},\tilde{Z},A)=(\mathcal{D}_{\ell d+1}(\epsilon),\mathcal{D}_{\ell d}(\epsilon),\mathcal{D}_{\ell d}^1)$) and the fact that we are working with $\mathbb{Z}_2$--coefficients.

On the other hand the bundle $E(\epsilon)$ is trivial, 
because the space $\Omega$ is contractible and the class $\eta(\epsilon)$ is zero. Therefore the differential is zero and this concludes the proof.
\end{proof}
\begin{remark}It is actually possible to prove the stabilization of the second differential, up to subsequences, in a simpler way. In fact $\{d_2(\epsilon)^{*, \ell d}: H^*(A_\ell, B_\ell)\to H^{*+2}(B_\ell, A_{\ell+1})\}_{d\geq 0}$ is a sequence of maps between finite dimensional $\mathbb{Z}_2$-vector spaces, i.e. 
\be d_2(\epsilon)^{*, \ell d}\in \mathrm{Hom}(\mathbb{Z}_2^{a}, \mathbb{Z}_2^b),\ee
where $a=\dim_{\mathbb{Z}_2}(H^*(A_\ell, B_\ell))$ and $b=\dim_{\mathbb{Z}_2}(H^{*+2}(B_\ell, A_{\ell+1})).$
Since $\mathrm{Hom}(\mathbb{Z}_2^{a}, \mathbb{Z}_2^b)\simeq \mathbb{Z}_2^{a\times b}$ is a finite set, then up to subsequences $d_2(\epsilon)^{*, \ell d}$ is eventually constant.
\end{remark}
\subsubsection{The asymptotic for the Betti numbers of the chamber}

\begin{figure}
\centering
\begin{tikzpicture}[thick,scale=0.7, every node/.style={scale=0.7}][line cap=round,line join=round,>=triangle 45,x=1.0cm,y=1.0cm]
\clip(0.,-1.) rectangle (21.,16.5);
\fill[line width=0.5pt,color=mycyan,fill=mycyan,fill opacity=0.10000000149011612] (3.,15.802222222222214) -- (5.,15.796296296296289) -- (5.,15.597777777777797) -- (3.,15.600740740740736) -- cycle;
\fill[line width=0.5pt,color=mycyan,fill=mycyan,fill opacity=0.10000000149011612] (3.,9.304444444444455) -- (5.,9.295555555555568) -- (5.,9.1) -- (3.,9.1) -- cycle;
\fill[line width=0.5pt,color=mycyan,fill=mycyan,fill opacity=0.10000000149011612] (3.,8.904444444444456) -- (3.,9.1) -- (5.,9.1) -- (5.,8.89555555555557) -- cycle;
\fill[line width=0.5pt,color=mycyan,fill=mycyan,fill opacity=0.10000000149011612] (3.,6.602222222222234) -- (5.,6.602222222222233) -- (5.,6.397777777777787) -- (3.,6.3977777777777876) -- cycle;
\fill[line width=0.5pt,color=mycyan,fill=mycyan,fill opacity=0.10000000149011612] (3.,6.202222222222232) -- (3.,6.3977777777777876) -- (5.,6.397777777777787) -- (5.,6.2022222222222325) -- cycle;
\fill[line width=0.5pt,color=mycyan,fill=mycyan,fill opacity=0.10000000149011612] (3.,3.9) -- (5.,3.9) -- (5.,3.7044444444444475) -- (3.,3.6955555555555595) -- cycle;
\fill[line width=0.5pt,color=mycyan,fill=mycyan,fill opacity=0.10000000149011612] (3.,3.5) -- (3.,3.6955555555555595) -- (5.,3.7044444444444475) -- (5.,3.5) -- cycle;
\fill[line width=0.5pt,color=mycyan,fill=mycyan,fill opacity=0.10000000149011612] (3.,1.1977777777777803) -- (5.,1.2092592592592668) -- (5.,1.) -- (3.,1.) -- cycle;
\fill[line width=0.5pt,color=mycyan,fill=mycyan,fill opacity=0.10000000149011612] (8.,13.) -- (20.,13.) -- (20.,11.) -- (8.,11.) -- cycle;
\draw [line width=0.5pt] (3.,1.)-- (5.,1.);
\draw [line width=0.5pt] (3.,1.)-- (3.,16.);
\draw [line width=0.3pt,dash pattern=on 5pt off 5pt] (5.,16.)-- (5.,1.);
\draw [line width=0.5pt] (3.,1.1977777777777805)-- (5.,1.209259259259267);
\draw [line width=0.5pt] (3.,3.5)-- (5.,3.5);
\draw [line width=0.5pt] (3.,3.69555555555556)-- (5.,3.704444444444448);
\draw [line width=0.5pt] (3.,3.9)-- (5.,3.9);
\draw [line width=0.5pt] (5.,6.2022222222222325)-- (3.,6.202222222222232);
\draw [line width=0.5pt] (5.,6.397777777777787)-- (3.,6.3977777777777876);
\draw [line width=0.5pt] (3.,6.602222222222232)-- (5.,6.602222222222231);
\draw [line width=0.5pt] (5.,8.895555555555568)-- (3.,8.904444444444456);
\draw [line width=0.5pt] (3.,9.1)-- (5.,9.1);
\draw [line width=0.5pt] (5.,9.295555555555566)-- (3.,9.304444444444457);
\draw [line width=0.3pt,dash pattern=on 5pt off 5pt] (1.1,10.0)-- (1.1,15.0);
\draw [line width=0.3pt,dash pattern=on 5pt off 5pt] (2.1,10.0)-- (2.1,15.0);
\draw (0.9872222222222141,15.962592592592616) node[anchor=north west] {$n\cdot d -1$};
\draw (0.996111111111103,6.513703703703718) node[anchor=north west] {$2\cdot d-1$};
\draw (0.8538888888888808,9.269259259259275) node[anchor=north west] {$3\cdot d -1$};
\draw [line width=0.3pt,dash pattern=on 5pt off 5pt] (3.,16.)-- (5.,16.);
\draw [line width=0.5pt] (3.,15.600740740740736)-- (5.,15.5977777777778);
\draw [line width=0.5pt] (3.,15.802222222222216)-- (5.,15.79629629629629);
\draw (1.3961111111111026,3.829259259259271) node[anchor=north west] {$d-1$};
\draw (2.0361111111111025,1.3848148148148238) node[anchor=north west] {$0$};
\draw [rotate around={-2.6978056333162965:(4.665555555555554,9.191481481481492)},line width=0.5pt,color=mycyan] (4.665555555555554,9.191481481481492) ellipse (2.261128078549663cm and 1.2068554957432387cm);
\draw [->,line width=0.5pt, color=mycyan] (6.489615740370821,9.842763696128216) -- (8.,11.);
\draw (8.231666666666653,12.824814814814834) node[anchor=north west] {$H^{0}(B_3, A_{4};\mathbb{Z}_2)$};
\draw (14.276111111111094,11.909259259259278) node[anchor=north west] {$H^{2}(A_3,B_3;\mathbb{Z}_2)$};
\draw [line width=0.5pt,color=blue] (11.135555555555554,12.561481481481493)-- (15.365,12.569259259259283);
\draw [->,line width=0.5pt,color=blue] (15.365,12.569259259259283) -- (15.365,11.769259259259282);
\draw [color=blue](12.516111111111096,12.664814814814834) node[anchor=north west] {$d_2$};
\draw [->,line width=0.5pt,color=mygray] (9.480555555555547,12.130740740740771) -- (18.440555555555555,10.4);
\draw (18.51611111111109,10.771481481481498) node[anchor=north west] {$0$};
\draw [->,line width=0.5pt,color=blue] (3.1155555555555545,9.181481481481491) -- (3.915555555555555,8.981481481481492);
\draw [->,line width=0.5pt,color=mygray] (3.115555555555563,9.18148148148149) -- (4.215555555555555,8.661481481481491);
\draw [->,line width=0.5pt] (3.115555555555568,9.181481481481487) -- (4.495555555555555,8.201481481481492);
\draw [line width=0.3pt,dash pattern=on 5pt off 5pt] (1.6,10.0)-- (1.6,15.0);
\draw [line width=0.3pt,dash pattern=on 5pt off 5pt] (4.495555555555555,8.201481481481492)-- (6.655555555555556,6.501481481481493);
\draw [line width=0.5pt] (5.,15.79629629629629)-- (5.,15.5977777777778);
\draw [line width=0.5pt] (4.997138447231981,9.295568273567868)-- (4.997170501965439,8.895568131102388);
\draw [line width=0.5pt] (4.997169143665958,6.602222222222231)-- (4.997169143665958,6.2022222222222325);
\draw [line width=0.5pt] (4.997169143665958,3.9)-- (5.,3.5);
\draw [line width=0.5pt] (5.,1.209259259259267)-- (5.,1.);
\draw [line width=0.3pt,dash pattern=on 5pt off 5pt] (3.1,0.45)-- (5.,0.45);
\draw [line width=0.3pt] (3.,0.6)-- (3.,0.3);
\draw [line width=0.3pt] (5.0,0.6)-- (5.0,0.3);
\draw (3.4,0.3) node[anchor=north west] {${n \choose 2}+1$};
\draw [line width=0.5pt] (8.,13.)-- (8.,10.);
\draw [line width=0.5pt] (8.,10.)-- (20.,10.);
\draw [line width=0.5pt] (20.,10.)-- (20.,13.);
\draw [line width=0.5pt] (20.,13.)-- (8.,13.);
\draw [line width=0.5pt] (8.,13.)-- (8.,14.);
\draw [line width=0.5pt] (8.,14.)-- (20.,14.);
\draw [line width=0.5pt] (20.,14.)-- (20.,13.);
\draw [line width=0.5pt] (11.,14.)-- (11.,10.);
\draw [line width=0.5pt] (14.,14.)-- (14.,10.);
\draw [line width=0.5pt] (17.,14.)-- (17.,10.);
\draw [line width=0.5pt] (8.,12.)-- (20.,12.);
\draw [line width=0.5pt] (8.,11.)-- (20.,11.);
\draw [line width=0.5pt] (20.,13.)-- (8.,13.);
\draw [color=mygray](11.947222222222207,11.615925925925945) node[anchor=north west] {$d_3$};
\begin{scriptsize}
\draw [fill=black] (3.,1.) circle (2.5pt);
\draw [fill=black] (3.,3.5) circle (2.5pt);
\draw [fill=black] (3.,6.202222222222232) circle (2.5pt);
\draw [fill=black] (3.,8.904444444444456) circle (2.5pt);
\draw [fill=black] (3.,15.600740740740736) circle (2.5pt);
\end{scriptsize}
\end{tikzpicture}
       \caption{This is a schematic image of the $E_2(\epsilon)$ term of the spectral sequence we are describing. The coloured parts correspond to the elements $E_2^{i,j}(\epsilon)$ of the spectral sequence which are possibly non-zero.}\label{fig:spectral}
\end{figure}
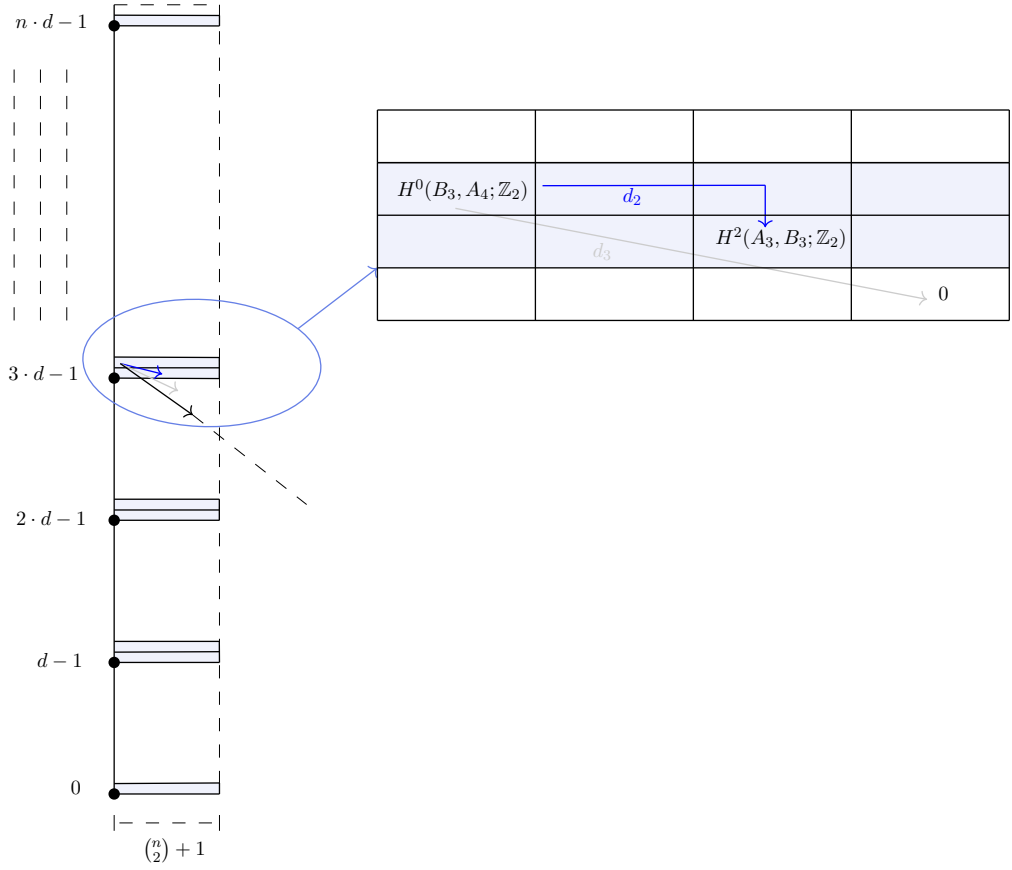
We are now in the position of proving the main theorem of this section, namely Theorem \ref{theorem:floer}.
\begin{proof}[Proof of Theorem \ref{theorem:floer}]The proof of this theorem is based on the analysis of the structure of the spectral sequence and its last page. First observe that by Proposition \ref{proposition:ed}, for all $k\geq 0$ and for all $\epsilon<\epsilon(d)$ we have:
\be b_k(W_{\sigma, d})=b_k(W_{\sigma, d}(\epsilon)).\ee
For the rest of the proof we will take $\epsilon\leq \min\{\epsilon(d), \epsilon_2\}$, where $\epsilon_2\leq \epsilon_1$ is given by Proposition \ref{proposition:differential}.
Lemma \ref{lemma:sphere} implies now that:
\be b_k(W_{\sigma, d})=\frac{1}{2}b_k(V_{\sigma, d}(\epsilon)).\ee
On the other hand, since the involved spaces are semialgebraic sets (hence triangulable), the Betti numbers of $V_{\sigma, d}(\epsilon)$ are related to those of $S^N\backslash V_{\sigma, d}(\epsilon)$ through Alexander duality:
\be \tilde b_k(V_{\sigma, d}(\epsilon))=\tilde b_{N-k-1}(S^N\backslash V_{\sigma, d}(\epsilon)).\ee
Finally, denoting by $e_{\infty}^{i,j}(\epsilon)$ the dimension of $E_\infty^{i,j}(\epsilon)$, where $E_\infty(\epsilon)$ is the last page of the spectral sequence from Theorem \ref{thm:spectral}, we have:
\be \tilde b_{N-k-1}(S^N\backslash V_{\sigma, d}(\epsilon))=\sum_{i+j=N-k-1}e_{\infty}^{i,j}(\epsilon).\ee
Collecting all this together,  for $\epsilon\leq \min\{\epsilon(d), \epsilon_2\}$, we have:
\be\label{eq:betti}b_k(W_{\sigma, d})= \frac{1}{2}\left\{\begin{matrix}
1+\displaystyle\sum\limits_{i+j=N-1}e_{\infty}^{i,j}(\epsilon)&\textrm{if $k=0$}\\[7pt]
\displaystyle\sum\limits_{i+j=N-k-1}e_{\infty}^{i,j}(\epsilon)&\textrm{if $0<k<N-1$}\\[7pt]
-1+e_{\infty}^{0,0}(\epsilon)&\textrm{if $k=N-1$}\end{matrix}\right.
\ee 

Observe now that Proposition \ref{propo:structure} implies that in the second page of the spectral sequence only the first ${n\choose 2}+1$ columns are nonzero (i.e. those with $0\leq i\leq {n\choose 2}$); moreover in the second page only the rows with $j=\ell d$ and $j=\ell d-1$ are potentially nonzero. Therefore, for $d\geq {n\choose 2}+2$ all the higher differentials are zero and:
\be E_{\infty}(\epsilon)=E_3(\epsilon).\ee
On the other hand Proposition \ref{proposition:differential} implies that $E_2(\epsilon)=E_3(\epsilon)=E_\infty(\epsilon)$, with the last equality for $d\geq {n\choose 2}+2$. 

Looking now at the top two rows of $E_\infty$, by Corollary \ref{corollary:contractible} we know that:
\be E_2^{i, {dn}}(\epsilon)=0\quad \forall i\geq 0, \quad E_2^{0, {dn-1}}(\epsilon)\simeq \mathbb{Z}_2\quad \textrm{and}\quad E_2^{i, {dn-1}}(\epsilon)=0\quad \forall i\geq 1.\ee
Thus, for $d\geq {n\choose 2}+2$,
\be e_\infty^{i, {dn}}(\epsilon)=0\quad \forall i\geq 0, \quad e_\infty^{0, {dn-1}}(\epsilon)=1\quad \textrm{and}\quad e_\infty^{i, {dn-1}}(\epsilon)=0\quad \forall i\geq 1.\ee
From this we immediately see that, for $d\geq {n\choose 2}+2$ we have:
\be b_0(V_{\sigma,d}(\epsilon))=1\quad \textrm{and} \quad b_k(V_{\sigma, d}(\epsilon))=0 \quad \forall 1\leq k\leq {n\choose 2}.\ee
This already proves:
\be \label{eq:betti0}b_0(W_{\sigma, d})=1\quad \textrm{and}\quad b_k(W_{\sigma, d})=0\quad \forall 1\leq k\leq {n\choose 2}.\ee
Observe now that the fact that the rows with $j=\ell d$ and $j=\ell d-1$ in $E_2=E_{\infty}$ are the only possibly non-zero rows for $\ell=0, \ldots, n$, influences the Betti numbers $b_k(W_{\sigma, d})$ with 
\be \label{eq:nonzero}k=md, \ldots, md-{n\choose 2}-1,\ee
where $m=n-\ell.$
We define now, for $m=1, \ldots, n-1$:
\be Q_{G, m}(t)=\frac{1}{2}\cdot\left(e_\infty^{0, (n-m)d-1}t^{\binom{n}{2}+1}+\sum_{i=1}^{{n\choose 2}}\left(e_\infty^{\binom{n}{2}-i, (n-m)d}+e_\infty^{\binom{n}{2}-i+1, (n-m)d-1}\right) t^{i}+e_\infty^{\binom{n}{2}, (n-m)d}\right).\ee
The $i$--th coefficient of the polynomial $Q_{G,m}$ is $b_{m d-{n\choose 2}+i-1}(W_{G, d})$. 
In principle we would have to consider also the case $m=n$, but Theorem \ref{thm:highbettinumbers} (which we prove below) guarantees that there is no homology in dimension greater than $(n-1)d-n+1.$
By \eqref{eq:nonzero}, the conclusion of the theorem follows.


\end{proof}
\begin{remark}As we noticed in the introduction, since the polynomials $Q_{G, 1}, \ldots, Q_{G, n-1}$ do not depend on $d$, but only on the graph, and since these polynomials are the same for isomorphic graphs, they define a graph invariant. Similarly the same is true for the Floer number $\beta(G)$, which is just the sum of their coefficients. Of course the polynomials are finer invariants, however we do not have a clear interpretation of these quantities. 
\end{remark}

\begin{remark}From \eqref{eq:betti0} it immediately follows that for $d\geq {n\choose 2}+2$ each sign condition is connected. In particular, if $d\geq {n\choose 2}+2$, two $\R^d$--geometric graphs on $n$ vertices are isomorphic if and only if they are rigid isotopic.
We will actually sharpen this in Corollary \ref{cor:conn1} below.
\end{remark}

%% file: Increasing_n.tex
\section{Increasing the number of points}\label{sec:n}
\input{1n}

\input{bounds}

%% file: 1n.tex
\subsection{Geometric graphs on the real line}
We now want to study the number of possible isotopy classes of geometric graphs on the real line when the number of points is large: this is precisely the case $d=1$, $n$ large. If we look at the discriminant $\Delta_{1,n}$, this is an arrangement of hyperplanes, namely
\[ \Delta_{1,n}=\{(x_1,\dots,x_n) \in\R^{1\times n}\mid\: \exists \, i,j \, |x_i-x_j|=1\}.\]

\begin{remark} There is a way to compute explicitly the number $b_0(\R^{1\times n}\setminus \Delta_{1,n})$ using a generalized version of the Mayer-Vietoris spectral sequence for semialgebraic sets. This gives $19$ for $n=3$, $183$ for $n=4$ and $2371$ for $n=5$. The computations becomes tricky for larger $n$, however this numbers are the beginning of a known integer sequence, which is the sequence of \textit{labeled semiorders on} $[n]$.
\end{remark}

The remark leads us to an obvious observation.  
An \textit{interval order for intervals $\{I_i\}_{i=1}^n$ of unit length (=semiorder for $[n]$)} is the partial order corresponding to their left-to-right precedence relation, i.e. one interval $I_i$ being considered less than another $I_j$ iff $I_i$ is completely to the left of $I_j$. In the case $d=1$ the number of components of the complement of $\Delta_{1,n}$ is exactly the number of possible semiorders for $[n]$. This because, once we defined the intervals $[p_i-1,p_i]$  for all $i$, each component of $\R^{1 \times n} \setminus \Delta_{1, n}$ is uniquely determined by whether $p_i<p_j-1$ or $p_i\nless p_j-1$ (see \cite[pag. 73]{stanley2004introduction}).

\begin{remark}
The type of semiorders introduced are usually addressed as semiorders on $n$ labeled items. The number of distinct semiorders on $n$ unlabeled items is given by the Catalan numbers $\{C_n\}_n.$ 
\end{remark}

Let us define $f(n):=$ \textit{number of labeled semiorders of} $[n]$. There is an explicit generating function for this sequence (see \cite[pag.78, Corollary 5.12]{stanley2004introduction}). We have
\[G(x):=\sum_{n\geq 0} f(n)\frac{x^n}{n!}=C(1-e^{-x})\]
where $C$ is the generating function of the known sequence of Catalan numbers $\{C_n\}$. More explicitly we have:
\[\sum_{n\geq 0} C_nx^n=C(x)=\frac{1-\sqrt{1-4x}}{2x}.\]

\begin{theorem}The number of rigid isotopy classes of $\R$--geometric graphs on $n$ vertices equals:
\be b_0\left(\R^{1\times n}\backslash \Delta_{1,n}\right)=\frac{1}{n}\cdot\sqrt{6\, \textrm{\normalfont{log}}\frac{4}{3}}\cdot \left(\frac{n}{e\:\textrm{\normalfont{log}}\frac{4}{3} }\right)^{n}\left(1+O(n^{-\frac{1}{2}})\right).\ee
\end{theorem} 

\begin{proof}
First of all let us notice that the Theorem \ref{theorem:anacomb} can be applied to $G(x)$ since we have only one singularity in $\log\frac{4}{3}$ and we can extend the function to a $\log\frac{4}{3}D$-domain, and actually to the whole $\mathbb{C}\setminus[\log\frac{4}{3},+\infty)$.
The function $C(x)$ has a unique singularity at $x=\frac{1}{4}$. It is easy to see
\[C(x)=2-2\sqrt{1-4x}+O(1-4x).\]
By composition we get
\[G(x)=2-2\sqrt{4e^{-x}-3}+O(4e^{-x}-3)\]
and from this 
\[G(x)=2-2\sqrt{3\textrm{\normalfont{log}}\frac{4}{3}}\cdot\sqrt{1-\frac{x}{\textrm{\normalfont{log}}\frac{4}{3}}}+O\left(1-\frac{x}{\textrm{\normalfont{log}}\frac{4}{3}}\right)=F\left(\frac{x}{\textrm{\normalfont{log}}\frac{4}{3}}\right)+O\left(1-\frac{x}{\textrm{\normalfont{log}}\frac{4}{3}}\right).\]
We can now apply Theorem \ref{theorem:anacomb}. We get
\[\frac{f(n)}{n!}=\left(\textrm{\normalfont{log}}\frac{4}{3}\right)^{-n}\cdot\sigma_n+O\left(\left(\textrm{\normalfont{log}}\frac{4}{3}\right)^{-n}\frac{1}{n^{2}}\right)\]
where $F(x)=\sum_{n=0}^{\infty}\sigma_n z^n$.
Using Remark \ref{remark:anadelt} and $\Gamma\left(-\frac{1}{2}\right)=-2\sqrt{\pi}$ we get
\[\frac{f(n)}{n!}=\left(\textrm{\normalfont{log}}\frac{4}{3}\right)^{-n}\cdot\frac{1}{\sqrt{\pi n^3}}\cdot\sqrt{3\textrm{\normalfont{log}}\frac{4}{3}}+O\left(\left(\textrm{\normalfont{log}}\frac{4}{3}\right)^{-n}\frac{1}{n^{2}}\right)\]
and by Stirling approximation
\[f(n)=\left(\frac{n}{e\:\textrm{\normalfont{log}}\frac{4}{3} }\right)^{n}\cdot\frac{1}{n}\cdot\sqrt{6\textrm{\normalfont{log}}\frac{4}{3}}+O\left(\left(\frac{n}{e\:\textrm{\normalfont{log}}\frac{4}{3} }\right)^{n} n^{-\frac{3}{2}}\right)\]
\end{proof}

With these computations we know asymptotically the number of isotopy classes of geometric graphs on the real line. However, as we discussed before, different isotopy classes can correspond to the same isomorphism class. It is therefore natural to ask for the number $\#_{1,n}$ of isomorphism classes of $\R$--geometric graphs, for $n$ large. In \cite{hanlon1982counting}, Hanlon computes the exponential generating function for this sequence (the author calls the corresponding graphs \emph{labeled unit interval graphs}).

The exponential generating function for $\{\#_{1,n}\}_n$ is
\be \Lambda(x)=\textrm{exp}(\Gamma(x))-1 \ee
where $\Gamma(x)$ is the generating function for the sequence $\{b_n\}_n$ of isomorphism classes of \emph{connected} $\R$--geometric graphs on $n$ vertices. More explicitly, we have \be\Gamma(x)=\frac{1}{4}(1-2z)-\frac{1}{4}\sqrt{\frac{1-3z}{1+z}}\textrm{  where  }z=e^x-1.\ee
Reasoning as before, we prove the following theorem.

\begin{theorem}\label{thm:number}
The number of isomorphism classes of $\R$--geometric graphs on $n$ vertices equals:
\be \#_{1,n}=\frac{e^{\frac{1}{12}}}{8}\cdot\frac{1}{n}\cdot\sqrt{6\, \textrm{\normalfont{log}}\frac{4}{3}}\cdot \left(\frac{n}{e\:\textrm{\normalfont{log}}\frac{4}{3} }\right)^{n}\left(1+O(n^{-\frac{1}{2}})\right).\ee
\end{theorem}
\begin{proof}
First of all let us notice that the Theorem \ref{theorem:anacomb} can be applied to $\Lambda(x)$ since we have only one singularity at $\log\frac{4}{3}$ and we can extend the function to a $\log\frac{4}{3}D$-domain, actually to $\mathbb{C}\setminus [\log\frac{4}{3},+\infty)$.
We start with
\be 1+\Lambda(x)=\textrm{exp}\left(\frac{1} {4}\cdot(3-2e^x)\right)\cdot\textrm{exp}\left(-\frac{1}{4}\sqrt{4e^{-x}-3}\right)\ee
Then, 
\[1+\Lambda(x)=\left(e^{\frac{1}{12}}+O\left(1-\frac{x}{\log\frac{4}{3}}\right)\right)\cdot\left(-\frac{1}{4}\sqrt{4e^{-x}-3}+1+O(4e^{-x}-3)\right)\]
\[1+\Lambda(x)=e^{\frac{1}{12}}-\frac{e^{\frac{1}{12}}}{4}\sqrt{3\log\frac{4}{3}}\sqrt{1-\frac{x}{\log\frac{4}{3}}}+O\left(1-\frac{x}{\log\frac{4}{3}}\right).\]
Finally,
\[\frac{\#_{1,n}}{n!}=\frac{e^{\frac{1}{12}}}{8}\cdot\left(\textrm{\normalfont{log}}\frac{4}{3}\right)^{-n}\cdot\frac{1}{\sqrt{\pi n^3}}\cdot\sqrt{3\textrm{\normalfont{log}}\frac{4}{3}}+O\left(\left(\textrm{\normalfont{log}}\frac{4}{3}\right)^{-n}\frac{1}{n^{2}}\right)\]
\[\#_{1,n}=\frac{e^{\frac{1}{12}}}{8}\cdot\left(\frac{n}{e\:\textrm{\normalfont{log}}\frac{4}{3} }\right)^{n}\cdot\frac{1}{n}\cdot\sqrt{6\textrm{\normalfont{log}}\frac{4}{3}}+O\left(\left(\frac{n}{e\:\textrm{\normalfont{log}}\frac{4}{3} }\right)^{n} n^{-\frac{3}{2}}\right).\]
\end{proof}
Even though in the general case we still do not have a clear understanding of the relation between $b_0\left(\R^{1\times n}\setminus \Delta_{1,n}\right)$ and ${\#_{1,n}}$, in the case $d=1$ we have the following corollary.
\begin{corollary}\label{coro:ratio}
We have 
\be b_0(\R^{1\times n}\backslash \Delta_{1,n})=\frac{8}{e^{\frac{1}{12}}}\cdot\#_{1, n}\left(1+O(n^{-\frac{1}{2}})\right)\ee
where $8/\sqrt[12]{e} = 7.3603...$.
\end{corollary}
The number $8/\sqrt[12]{e}$ can be roughly interpreted as the average number of rigid isotopy classes realizing a particular $\R$-geometric graph isomorphism type.

%% file: bounds.tex
\subsection{Asymptotic enumeration in higher dimensions}
While the situation for isotopy classes of geometric graphs on the real line is given by the number of semiorders on $[n]$, such a closed form description apparently does not exist for larger values of $d$. Nonetheless we are able to obtain reasonable bounds on the asymptotics following methods of McDiarmid and M\"{u}ller from \cite{McDiarmidMuller} who study asymptotic enumeration of labeled disk graphs in $\mathbb{R}^2$. A disk graph in $\mathbb{R}^2$ is a graph given by an arrangement of open disks in $\mathbb{R}^2$ where the vertices are the disks and there is an edge between a pair of them if and only if the corresponding disks intersect one another. In the case that all the disks have the same radius this is exactly the setting of our geometric graphs in the case that $d = 2$. McDiarmid and M\"{u}ller show that the number of labeled graphs on $n$ vertices which are unit disk graphs in $\mathbb{R}^2$, in our notation $\#_{2, n}$, is order $\exp(2n \log(n) + \Theta(n))$, and adapting their method we prove the following theorem. In particular we prove that the right asymptotic rate of growth both for $\#_{d, n}$ and for $b_0\left(\R^{d \times n}\backslash \Delta_{d,n}\right)$ is $\exp(dn \log(n) + \Theta(n))$.
We will prove the following theorem, which is a consequence of Theorem \ref{lemma:upperbound} and Theorem \ref{theo:lowbound}.
\begin{theorem}\label{thm:lower}
For $d \geq 2$ fixed and $n\geq 4d+1$ one has the following bounds:
\be \left(\frac{1}{(d+1)e^2}\right)^{dn}n^{dn}\leq \#_{d, n}\leq b_0\left(\R^{d \times n}\backslash \Delta_{d,n}\right)\leq 2dn \left(\frac{3e}{2d} \right)^{dn} n^{dn}
\ee
\end{theorem}

\subsection{General case: the upper bound} \label{sec:upperboundb0}
While McDiarmid and M\"{u}ller are primarily interested in enumerating labeled geometric graphs in the plane, in our notation $\#_{2, n}$, their upper bound holds for general $d$, as they point out in \cite{McDiarmidMuller}. The key lemma in their proof of their upper bound is the following result of Warren \cite{Warren}. 
\begin{theorem}[\cite{Warren}]
If $P_1, ..., P_m$ are polynomials of degree at most $t$ in real variables $z_1, ..., z_k$ then the number of distinct sign patterns 
$(\mathrm{sign}(P_1(\overline{z}), ..., \mathrm{sign}(P_m(\overline{z})) \in \{-1, 1\}^m$
that occur in $\R^k \setminus \cup_{i = 1}^m \{\overline{z} : P_i(\overline{z}) = 0\}$ is at most
\[\left(\frac{4etm}{k}\right)^k.\]
\end{theorem}
Now given $n, d$ we take the $\binom{n}{2}$ polynomials in variables $(x_1, ..., x_n) \in \R^{dn}$ given by $q_{i, j}(x) = \|x_i-x_j\|^2 - 1$, defined in \eqref{eq:qp}. Then each sign pattern of these $\binom{n}{2}$ degree 2 polynomials in $dn$ variable corresponds to a unique isomorphism class of labeled geometric graphs on $n$ vertices in $\R^d$. Therefore we have the following bound:
\be
\#_{d,n}\leq\left(\frac{4e}{d}\right)^{nd} n^{nd}.\ee
However, a single sign pattern could be a disjoint union of several rigid isotopy classes, so we need a different argument to bound $b_0(\R^{d \times n} \setminus \Delta_{d,n})$. We prove the following theorem.

\begin{theorem}[Upper bound]\label{lemma:upperbound}For fixed $d$ and for $n\geq 4d+1$, we have the bound:
\[b_0(\R^{d \times n} \setminus \Delta_{d, n}) \leq 2dn \left(\frac{3e}{2d}\right)^{dn} n^{dn}\]
\end{theorem}
\begin{remark}\label{rem:comeq} Let us denote with  $\hat{\Delta}_{d,n}^G$ the one point compactification of the set $\Delta_{d,n}^G$ defined in \eqref{eq:DG} where $G$ is any graph on $[n]$. This is an algebraic set $X$ of $\R^{nd+1}=(x_1,\dots,x_n,z)$ defined by $k+1$ equations which are
\[\norm{x_i-x_j}^2=(1-z)^2\]
with $(i,j)$ edge of $G$ and the equation of the sphere $\norm{x_1}^2+\dots+\norm{x_n}^2+z^2=1$. In fact, if we look at the explicit expression of the stereographic projection we get an homeomorphism  between $\Delta_{d,n}^G$ and $X\setminus(\{(0,1)\})$ and from this the claim.
\end{remark}
\begin{proof}
By Alexander duality (see Section \ref{sec:aldu}) $b_0(\R^{d \times n} \setminus \Delta_{d, n}) = b_{dn - 1}(\hat{\Delta}_{d, n}) + 1$, where $\hat{\Delta}_{d,n}$ is the one-point compactification of the discriminant. Therefore, it is sufficient to bound $b_{dn - 1}(\hat{\Delta}_{d, n})$.
Let us consider the Mayer--Vietoris spectral sequence for simplicial complexes (see \cite[Section 3.2] {basu2003different} for a complete construction). Thanks to the previous remark $\hat{\Delta}_{d, n}$ is an algebraic set and we can use the mentioned spectral sequence with respect to the algebraic covering $\{\hat{\Delta}^G_{d, n}\}_{G}$ where $G$ varies over nonempty labeled graphs on $[n]$.
The $E_1$ page of the spectral sequence has 
\be\label{eqn:E1}
E_{1}^{i,j} = \bigoplus_{G \text{ a graph on $[n]$ with exactly $i + 1$ edges}} H^j(\hat{\Delta}^{G}_{d, n}),
\ee and 
\be\label{eqn:DN1} b_{dn - 1}(\hat{\Delta}_{d, n}) \leq \sum_{i = 0}^{dn - 1} \textrm{dim}_{\mathbb{Z}_2} (E_1^{i,(dn - 1 - i)}).\ee

Using Theorem 2 of \cite{milnor1964betti} and the fact that for any labeled graph $G$ on $[n]$ the topological space $\hat{\Delta}_{d,n}^G$ is an algebraic set defined by equations of degree $2$ in $\R^{dn}$, we get that its total Betti number is at most $2(3)^{dn - 1}$. Using this, we have:
\begin{eqnarray*}
b_{dn - 1}(\hat{\Delta}_{d,n}) &\leq& \sum_{k = 1}^{dn} \binom{\binom{n}{2}}{k} 2(3)^{dn} \\
&\leq& 2(3)^{dn} \sum_{k = 1}^{dn} \binom{\binom{n}{2}}{k} \\
&\leq& 2(3)^{dn} dn \binom{\binom{n}{2}}{dn} \\
&\leq& 2(3)^{dn} dn \left(\frac{n^2 e}{2dn}\right)^{dn},
\end{eqnarray*}
where in the third inequality we used $n\geq 4d+1$.
\end{proof}

\subsection{General case: the lower bound}
For the lower bound on the number of labeled disk graphs, McDiarmid and M\"{u}ller give a procedure for inductively generating many distinct labeled disk graphs. Here we generalize this procedure to higher dimensions.

For each $k \geq d + 1$ we construct a family of non-isomorphic labeled geometric graphs on $k$ vertices in $\R^d$, $U_{k, d}$. If we let $u_{k, d}$ denote the number of graphs in $U_{k, d}$, we show that for $k \geq d + 1$, 
\[u_{k+1, d} \geq \left(\left\lfloor \frac{k}{d + 1} \right\rfloor\right)^d u_{k, d}. \]
This recursion implies the following which we prove in Section \ref{sec:proof}.
\begin{theorem}[Lower bound]\label{theo:lowbound} We have for $n> d+1$ that \[\left(\frac{n}{(d+1)e^2}\right)^{dn}\leq\#_{d, n}\]
\end{theorem}


For the base of the recursion, we start with the regular $d$-simplex in $\R^d$ with edges of length 1 and vertices given by $P_1$, $P_2$, ..., $P_{d+1}$. The $1$-skeleton of the $d$-simplex is a geometric graph in $\R^d$ this will be the singleton element of $U_{d + 1, d}$. Though this graph is degenerate it will still contribute to $\#_{d, n}$ which is always a lower bound for $b_0(\R^{d \times n} \setminus \Delta_{d, n})$ by the discussion following Lemma \ref{lem:broad}.

To construct the families $U_{k,d}$ for $d + 1 < k\leq n$ we need the following technical lemma which generalizes Lemma 4.1 of \cite{McDiarmidMuller}.

\begin{lemma}\label{lem:ball}
There exist constants $\epsilon_0>0$ and $C>0$ such that for all $0<\epsilon<\epsilon_0$ and all $p_i\in B(P_i,\epsilon)$ for all $i \in [d]$ there exists a unique point \[q(p_1,\dots,p_d)\in B(P_{d+1},C\epsilon)\] with $\norm{q-p_i}=1$ for all $i \in \{1, ..., d\}$.
\end{lemma}
In other words for $\epsilon$ small enough, this lemma tells us that there is a well-defined Lipschitz continuous function, with Lipschitz constant $C$, $q$ on $B(P_1, \epsilon) \times B(P_2, \epsilon) \times \cdots \times B(P_d, \epsilon)$ mapping $(x_1, ..., x_d)$ to the unique point of the intersection of sphere $S(x_1, 1) \cap S(x_2, 1) \cap \cdots \cap S(x_d, 1)$ closest to $P_{d + 1}$. The $d=2$ case is Lemma 4.1 of \cite{McDiarmidMuller} and is essentially proved directly via a closed form for $q$ in terms of $p_1, p_2 \in B(P_1, \epsilon) \times B(P_2, \epsilon)$ that is well-defined and Lipschitz continuous for $\epsilon$ small enough. For larger values of $d$ writing down the closed form of $q$ would be much more complicated. Therefore, we instead describe algorihmically how one would compute $q$ given $p_1, ..., p_d$ sufficiently close respectively to $P_1, P_2, ..., P_d$ and show that $q$ will ultimately be a combination of Lipschitz continuous functions.
\begin{proof}[Proof of Lemma \ref{lem:ball}]
We show that for $\epsilon$ small enough, the intersection $S(p_1, 1) \cap \cdots \cap S(p_d, 1)$ with $p_i \in B(P_i, \epsilon)$ for all $i$ is two points $q^+(p_1, ..., p_d)$ and $q^-(p_1, ..., p_d)$ with $q^+(p_1,..., p_d)$ the closer of the two to $P_{d + 1}$, and that $q^+ : B(P_1, \epsilon) \times \cdots \times B(P_d, \epsilon) \rightarrow \R^d$ is Lipschitz continuous.

We will prove that $q^+$ is well-defined and Lipschitz continuous close to $P_1, P_2, ..., P_d$ by describing the algorithm one would use to compute $q^+$ and show that each step of the algorithm is given by composition or addition of Lipschitz continuous functions. Given a tuple of points $(p_1, p_2, ..., p_d) \in B(P_1, \epsilon) \times B(P_2, \epsilon) \times \cdots \times B(P_d, \epsilon)$, with $\epsilon$ sufficiently small one could compute $q^+$ via the following recursive procedure.

First find the $(d - 2)$-dimensional sphere given by the intersection of  $S(p_1, 1)$ and $S(p_2, 1)$. Now for any $0 \leq k \leq (d - 1)$ a $k$-dimensional sphere in $\R^d$ may be described completely by its center, its radius, and the affine subspace of dimension $k + 1$ in which it is contained. In other words a $k$-dimensional sphere in $\R^d$ is described by a point in $\R^d$, a positive real number, and an element of the Grassmannian $\text{Gr}(k + 1, d)$. Given $p_1$ and $p_2$ in $\R^d$ with the distance from $p_1$ to $p_2$ smaller than 2 the intersection $S(p_1, 1) \cap S(p_2, 1)$ is a $(d - 2)$-dimensional sphere. It follows that taking $\epsilon$ small enough so that for any $p_1, p_2 \in B(P_1, \epsilon) \times B(P_2, \epsilon)$, $||p_1 - p_2||^2 < 4$ we have a continuous function $(C, R, G) : B(P_1, \epsilon) \times B(P_2, \epsilon) \rightarrow \R^d \times \R^+ \times \text{Gr}(d - 1, d)$. This map sends $(p_1, p_2)$ to the $(d - 2)$-dimensional sphere $S(p_1, 1) \cap S(p_2, 1)$ with center $C(p_1, p_2)$ radius $R(p_1, p_2)$ living in the affine hyperplane $C(p_1, p_2) + G(p_1, p_2).$

Now given $(p_1, ..., p_d) \in B(P_1, \epsilon) \times \cdots \times B(P_d, \epsilon)$ with $\epsilon$ small enough we have that the intersection of $C(p_1, p_2) + G(p_1, p_2)$ with $S(p_1, 1) \cup S(p_2, 1) \cup \cdots \cup S(p_d, 1) \subseteq \R^d$ gives an arrangement of $(d - 1)$ many $(d - 2)$-dimensional spheres in the affine hyperplane $C(p_1, p_2) + G(p_1, p_2)$. The center and radii of these spheres will be determined by how $C(p_1, p_2) + G(p_1, p_2)$ intersects each $S(p_i, 1)$. By induction we find the two points of intersection of these $(d - 2)$-dimensional spheres in the $(d-1)$-dimensional Euclidean space given by the affine hyperplane $C(p_1, p_2) + G(p_1, p_2)$. Once these two points of intersection have been found we pick the one that is closest to $P_{d + 1}$ to be $q^+(p_1, ..., p_d)$.

It can be verified routinely that $(C, R, G)$ as defined above is Lipschitz continuous in each coordinate. Moreover the arrangement given by intersecting $C(p_1, p_2) + G(p_1, p_2)$ with $S(p_1, 1) \cup S(p_2, 1) \cup \cdots \cup S(p_d, 1) \subseteq \R^d$ when $p_i$ is sufficiently close to $P_i$ for all $i$, can be described by a $2(d - 1)$ tuple of points $(c_2, r_2, ..., c_{d}, r_d)$ where each $c_i$ belongs to $C(p_1, p_2) + G(p_1, p_2)$ and $r_i \in (0, 1]$. Here $c_i$ and $r_i$ are respectively the center and the radius of the $(d-2)$-dimensional sphere given by $S(p_i, 1) \cap (C(p_1, p_2) + G(p_1, p_2))$ for $i \geq 3$ with $c_2$ and $r_2$ respectively the center and radius of the $(d-2)$-dimensional sphere given by the intersection $S(p_1, 1) \cap S(p_2, 1)$ i.e. $c_2$ and $r_2$ are $C(p_1, p_2)$ and $R(p_1, p_2)$ .

By continuity $C(p_1, p_2)$ can be made arbitarily close to $C(P_1, P_2)$, $G(p_1, p_2)$ can be made arbitrarily close to $G(P_1, P_2)$, and $R(p_1, p_2)$ can be made arbitrarily close to $R(P_1, P_2)$. From here it may be verified that for $\epsilon$ small enough there is a Lipschitz continuous function from $\phi: B_{\text{Gr}(d - 1, d)}(G(P_1, P_2), \epsilon) \times B_{\R^d}(C(P_1, P_2), \epsilon) \times B_{\R}(R(P_1, P_2), \epsilon) \times B_{\R^d}(P_3, \epsilon) \times \cdots \times B_{\R^d}(P_d, \epsilon) \rightarrow (\R^{d - 1} \times \R)^{d - 1}$ mapping an element of the domain to the arrangement of $(d - 1)$ many $(d - 2)$-dimensional spheres in the affine hyperplane as described above. By induction we have that the $S^0$ at the intersection of the arrangement is Lipschitz continuous on the image of $\phi$, and finally picking the closest of the two points to the fixed point $P_{d + 1}$ is Lipschitz continuous too. Note that the base case for the induction can simply be the $d = 1$ case; given two points in $\R$ picking the one closest to a fixed point is always Lipschitz continuous when the center of the two points lives in some small enough interval around a second fixed point.\end{proof}

Let us take the sequence $0<\epsilon_1<\dots<\epsilon_n$ defined by $\epsilon_{i}=\epsilon_0/C^{n-i}$, where $\epsilon_0$ and $C$ as in the previous Lemma and we are assuming $C>1$. To construct elements of $U_{k,d}$ recursively from $U_{k-1,d}$ we will also use as an inductive hypothesis that all the elements $P=(p_1,\dots,p_k)\in U_{k,d}$ satisfy the following two properties:
\begin{itemize}
    \item[\textcolor{mycyan}{P$1$}]$\norm{p_i-P_j}<\epsilon_i$ with $i\equiv j$ mod $d+1$
    \item[\textcolor{mycyan}{P$2$}] $S(p_{i_1},1)\cap\dots\cap S(p_{i_{d+1}},1)=\emptyset$ for all distinct $\{i_1,\dots,i_{d+1}\}$
\end{itemize}
These two conditions hold for $U_{d + 1, d}$ whose vertices are $P_1, ..., P_{d + 1}$.

Condition \textcolor{mycyan}{P$1$} above naturally partitions the vertices of $G \in U_{k, d}$ into $d + 1$ distinct classes given by the clustering of the vertices of $G$ around the points $P_1$, ..., $P_d$, $P_{d + 1}$. The proof of the claimed recursive lower bound on $u_{k, d}$ will be that if $G \in U_{k-1, d}$ and, without loss of generality, $k  \equiv 0 \mod (d + 1)$ then picking a transversal $\sigma = \{p_{i_1}, p_{i_2}, \cdots, p_{i_d} \}$ of vertices of $G$ where $i_l \equiv l \mod (d + 1)$ for every $i$, we give a procedure to choose a position for a new vertex $p_k$ to be added to $G$ so that the neighborhood of $p_k$ is unique for each choice of $\sigma$ and so that \textcolor{mycyan}{P$1$} and \textcolor{mycyan}{P$2$} are satisfied still satisfied after adding $p_{k}$. As $p_k$ will depend on $\sigma$ and each choice of $\sigma$ gives a distinct neighborhood for $p_k$ we have that there are at least as many combinatorially distinct ways to extend $G$ as there are choices for $\sigma$. 



If we denote by $\mathcal{P}(G)$ for $G \in U_{k-1, d}$ the set of all such transversals, that is the number of $d$-tuples $(i_1,\dots,i_{d})$ with $1\leq i_j< k$ such that $i_j\equiv j$ mod $d+1$, we have
\[\label{eq:num}|\mathcal{P}|\geq\lfloor (k-1)/(d+1)\rfloor^d.\]

Now, let $\mathcal{M}_{\pi}$ be defined as the intersection of the open balls,
\[\mathcal{M}_{\pi}:=B(p_{i_1},1)\cap\dots\cap B(p_{i_d},1).\]
We have the following lemma, which in the 2-dimensional case is Claim 4.3 of \cite{McDiarmidMuller}. The proof, which we omit, is exactly analogous to the 2-dimensional case and relies on the fact that for each $\pi = (i_1, ..., i_d) \in \mathcal{P}$, $S(p_{i_1}, 1) \cap \cdots \cap S(p_{i_d}, 1) \cap B(P_{d + 1}, \epsilon_k)$ is a single point due to Lemma \ref{lem:ball} and that single point is unique for each choice of $\pi$ by condition \textcolor{mycyan}{P$2$}.

\begin{lemma} \label{lemma:opensubsets}There exists nonempty open sets $O_{\pi}\subset\mathcal{M}_{\pi}$, such that for all $\pi\neq \sigma\in\mathcal{P}$ we have either $O_{\pi}\cap\mathcal{M}_{\sigma}=\emptyset$ or $O_{\sigma}\cap\mathcal{M}_{\pi}=\emptyset$.
\end{lemma}
Now, for each $\pi\in\mathcal{P}$ let us pick an arbitrary
\[q_{\pi}=O_{\pi}\setminus \bigcup_{i=1}^{k-1} S(p_i,1),\]
and we have that the $\R^d$-geometric graph obtained by adding the vertex $q_{\pi}$ to $G$, which we will denote $(G, q_{\pi})$, satisfies conditions \textcolor{mycyan}{P$1$} and \textcolor{mycyan}{P$2$}. Moreover for every choice of $\pi$ we obtain a unique way to extend $G$ by the following lemma.

\begin{lemma}\label{lem:notisog}
If $\pi\neq\sigma \in \mathcal{P}(G)$ for $G \in U_{k - 1, d}$ the geometric graphs $(G,q_{\pi})$ and $(G,q_{\sigma})$ are not isomorphic.
\end{lemma}
This holds because we have that $q_{\pi}$ and $q_{\sigma}$ will have different sets of neighbors, generalizing Claim 4.4 of \cite{McDiarmidMuller}.

\begin{lemma}\label{lemma:differentneighborhood}
If $\pi\neq\sigma \in \mathcal{P}(G)$ for $G \in U_{k - 1, d}$, $N(q_{\pi}) \neq N(q_{\sigma})$ where $N(v)$ denotes the neighbors of a point $v$ in $\R^d$ in the geometric graph $(G, v)$, that is the vertices of $G$ at distance less than 1 from $v$.
\end{lemma}
\begin{proof}
For $\pi \neq \sigma$ we have that $\sigma \subseteq N(q_\pi)$ if and only if $q_{\pi} \in \mathcal{M}_{\sigma}$. Clearly $\sigma \subseteq N(q_{\sigma})$ and $\pi \subseteq N(q_{\pi})$ but by Lemma \ref{lemma:opensubsets} it cannot be the case that both $\sigma \subseteq N(q_{\pi})$ and $\pi \subseteq N(q_{\sigma})$.
\end{proof}

Now, for each $G \in U_{k - 1, d}$ and $\pi\in\mathcal{P}(G)$ we construct a geometric graph $(P,q_{\pi})$ which satisfies conditions \textcolor{mycyan}{P$1$} and \textcolor{mycyan}{P$2$} and which satisfies Lemma \ref{lem:notisog}. Then,
\[u_{k,d}\geq|\mathcal{P}|\cdot u_{k-1,d}\geq \lfloor (k-1)/(d+1)\rfloor^d\cdot u_{k-1,d}.\]

\subsubsection{Proof of Theorem \ref{theo:lowbound}}\label{sec:proof}We have \[\#_{n,d}\geq u_{n,d}\geq\left(\prod_{i=d+1}^{n-1} \left \lfloor \frac{i}{d+1} \right \rfloor\right)^d\geq\left(\prod_{i=d+1}^{n-1} \frac{i-d}{d+1}\right)^d\geq\left(\frac{(n-d-1)!}{(d+1)^{n-d-1}}\right)^d.\]
Using the estimate $k!\geq\left(\frac{k}{e}\right)^k$ we get
\[\#_{n,d}\geq \left(\frac{n-d-1}{(d+1)e} \right)^{d(n-d-1)}\geq\left(\frac{n}{(d+1)e}\right)^{dn}\left(\frac{n}{d+1}\right)^{-d(d+1)},\]
where for the last inequality we used $\left(1-\frac{d+1}{n}\right)^{d(n-d-1)}\geq e^{-d(d+1)}$, which derives from $(1+\frac{d+1}{n-d-1})^{d(n-d-1)}\leq e^{d(d+1)}$. We then use $\left(\frac{n}{d+1}\right)^{d(d+1)}\leq\left(e^{d}\right)^n$ to obtain the lower bound.

\subsection{The top Betti numbers}\label{sec:topbetti}
The goal of this section will be to prove Theorem \ref{thm:highbettinumbers} regarding the low degree Betti numbers of the one point compactification of the discriminant.

We will prove this using the generalized Nerve Lemma of Bj\"{o}rner.
\begin{theorem}[Special case of Theorem 6 of \cite{Bjorner}]\label{lemma:BjornerLemma}
Let $X$ be a regular $CW$ complex and $(X_i)_{i \in I}$ a family of subcomplexes such that $X = \bigcup_{i \in I}X_i$. Suppose that every non-empty finite intersection $X_{i_1} \cap \cdots \cap X_{i_t}$ is $(k - t + 1)$-connected then $X$ is $k$-connected if and only if the nerve $\mathcal{N}(X_i)$ is $k$-connected.
\end{theorem}
Recall that given a CW-complex $X$ and a covering by subcomplexes $(X_i)_{i \in I}$ the nerve of the covering $\mathcal{N}(X_i)$ is the simplicial complex on the vertex set $I$ where $\sigma = [i_1, ..., i_t]$ is a face of the nerve if and only if $X_{i_1} \cap \cdots \cap X_{i_t}$ is nonempty.

The covering that we will use for $\hat{\Delta}_{d, n}$ will be given by $(\hat{\Delta}_{d, n}^{i, j})_{1 \leq i < j \leq n}$ where $\hat{\Delta}_{d, n}^{i, j}$ denotes the compactification in $\R^{d \times n}$ of the space $\Delta_{d, n}^{i, j} = \{(x_1, ..., x_n) \in \R^{d \times n} \mid \|x_i - x_j\|^2 = 1\}$. 

Now each intersection of a set of $\Delta_{d, n}^{i, j}$'s is naturally associated to a graph $G$ and a space $\Delta_{d, n}^G$ described in (\ref{eq:DG}). 


To be able to apply the generalized Nerve Lemma, a key step will be to establish the following about the topology of $\hat{\Delta}_{d,n}^G$.
\begin{lemma}\label{lemma:graphtopology}
For any graph $G$ with $\beta_0(G)$ connected components ${\Delta}_{d, n}^G$ is a direct product of a compact set $K := K(G)$ of dimension at most $(d - 1)(n - \beta_0)$ and $\R^{d \times \beta_0}$. Therefore $\hat{\Delta}_{d, n}^G$ is $d\beta_0 - 1$ connected.
\end{lemma}
Toward proving this result it will be helpful to observe that $\Delta_{d, n}^G$ is the space of graph homomorphisms from $G$ into the \emph{unit distance graph on $\R^d$}. The unit distance graph on $\R^d$ is the graph whose vertices are the points in $\R^d$ with an edge between two points if and only if the two points are at distance 1 from each other. 
\begin{proof}[Proof of Lemma \ref{lemma:graphtopology}]
Let $H$ be a component of $G$ and let $T$ be a spanning tree of $H$. Clearly $\Delta_{d, |V(H)|}^H \subseteq \Delta_{d, |V(H)|}^T$ as any homomorphism from $H$ to the unit distance graph on $\R^d$ induces a homomorphism from $T$ to the unit distance graph on $\R^d$. Moreover, we have that $\Delta_{d, |V(H)|}^T \sim \R^d \times (S^{d - 1})^{|E(T)|}$. Indeed we may regard $T$ as being a rooted tree and we can map the root of $T$ to any point of $\R^d$ and from there every vertex may live anywhere on the sphere of radius 1 centered at the image of its parent vertex. Now taking $\Delta^{H}_{d, |V(H)|}$ and  modding out by the $\R^d$ factor coming from the choice of image for the root in $T$ we have a closed subset of the compact set $(S^{d - 1})^{|E(T)|} = (S^{d - 1})^{|V(H)| - 1}$. Thus, without fixing the image of the root, we have that $\Delta_{d, |V(H)|}^H$ is the direct product of $\R^d$ and the compact space $K(H)$ given by the closed subset of $(S^{d - 1})^{|V(H)| - 1}$. It is clear that we may describe any homomorphism from $G$ to the unit distance graph on $\R^d$ as a product of graph homomorphisms on the connected components. We have that 
\[\Delta_{d, n}^G \cong K(G) \times \R^{d\beta_0}\]
where $K(G)$ is the direct product of $K(H)$ over all connected components $H$. Thus $K(G)$ is contained in some $n - \beta_0$ fold product of $(d - 1)$-dimensional sphere so it is at most $(d - 1)(n - \beta_0)$ dimensional. 

We now turn our attention to the compactification of $\Delta_{d, n}^G$. By the description of $\Delta_{d,n}^G$, we have that $\Delta_{d, n}^G$ is a $d\beta_0$-ranked vector bundle over a compact CW complex (since we are working with semialgebraic sets), so its compactification is the Thom space of this vector bundle which is $d \beta_0 - 1$ connected by Lemma 18.1 of \cite{MilnorStasheff}.
\end{proof}

\begin{proof}[Proof of Theorem \ref{thm:highbettinumbers}]
We apply Lemma \ref{lemma:BjornerLemma} to prove that $\hat{\Delta}_{d, n}$ is $n + d - 3$ connected. We consider the cover of $\hat{\Delta}_{d, n}$ by $(\hat{\Delta}_{d, n}^{i, j})_{1 \leq i < j \leq n}$ as discussed above so that for any $t$, the $t$-fold intersection of complexes in the cover is $\hat{\Delta}_{d, n}^G$ for some graph $G$ on $t$ edges. Moreover we observe that the nerve of this covering is just the simplex on $\binom{n}{2}$ vertices, as any intersection of the spaces in the cover at least contains the point at infinity so in particular the nerve is $(n + d - 3)$-connected.  Thus it suffices to check that $\hat{\Delta}_{d, n}^G$ is $n + d - 3 -t + 1 = n - t + d - 2$ connected for any graph $G$ on $n$ vertices with $t$ edges. By Lemma \ref{lemma:graphtopology}, it suffices to verify that for such a graph $n - t + d - 2 \leq d\beta_0(G) - 1$. This always holds for $d \geq 1$ because $\beta_0 \geq 1$, $\beta_1 \geq 0$ and $n - t = \beta_0 - \beta_1$.  

To show that $\hat{\Delta}_{n, d}$ is not $(n + d - 2)$-connected we use duality and verify that $\R^{d \times n} \setminus \Delta_{d, n}$ has at least one $dn - 1 - (n + d - 2) = (n - 1)(d - 1)$ reduced homology class. The path components of $\R^{d \times n} \setminus \Delta_{d, n}$ are the rigid isotopy classes of graph on $n$ vertices in $\R^d$. We consider the rigid isotopy classes of the empty graph on $n$ vertices. The rigid isotopy classes of the empty graph on $n$ vertices gives all the configurations of $n$ points in $\R^d$ so that the distance between any pair of them is larger than 1. This is  homotopy equivalent to  $\mathrm{Conf}_n(\R^d)$, see Examples \ref{ex:confhomology} and \ref{ex:confbettinumbers}, which has its top positive reduced Betti number in dimension $(n - 1)(d-1)$. 
\end{proof}

%% file: Examples.tex
\section{Examples}
Here we work out a few examples for computing the Betti numbers of $\R^{d \times n} \setminus \Delta_{d, n}$ for small values of $d$ and $n$. We start with the case that $d = 1$. In the case that $d = 1$, computing the topology of $\hat{\Delta}_{1, n}^G$ across all graphs $G$ on $n$ vertices is sufficient to compute $b_0(\R^{1 \times n} \setminus \Delta_{1, n})$. While we have shown the the number of rigid isotopy classes of $\R$-geometric graphs on $n$ vertices is given by the number of labeled semiorders on $n$ elements, we work out a computation for $n = 3$ here primarily to show how spectral sequences and Alexander--Pontryagin duality can be used to compute the exact number of rigid isotopy classes of $\R$-geometric graphs. 

\begin{example}[$n = 3$, $d = 1$]
By Alexander--Pontryagin duality, it suffices to compute the Betti numbers of $\hat{\Delta}_{1, 3}$. This is a union of three compactified quadrics given by the solutions in $\R$ to $|x_i- x_j|^2 = 1$ for $1 \leq i < j \leq 3$. Each of these is simply the disjoint union of two hyperplanes in $\R^{1 \times 3}$. Thus $\hat{\Delta}_{1, 3}$ is a 2-dimensional cell complex and we know by Theorem \ref{thm:highbettinumbers} that this complex is 1-connected so only the 2nd Betti number is interesting. This is not surprising as the dual in $S^3$ of $\hat{\Delta}_{1,3}$ is $\R^3 \setminus \Delta_{1, 3}$ which is a disjoint union of finite intersections of halfspaces, so only its zeroth Betti number is interesting. 

Now by computing the $E_1$ page of the Mayer--Vietoris spectral sequence given by the covering of $\hat{\Delta}_{1, 3}$ by $\left(\hat{\Delta}_{1, 3}^{i, j}\right)_{1 \leq i < j \leq n}$ we have that the Euler characteristic of $\hat{\Delta}_{1, 3}$ will be given by 
\[\chi(\hat{\Delta}_{1, 3}) = \sum_{0 \leq i \leq 2, 0 \leq j \leq 2} (-1)^{i + j}\dim(E_1^{i,j}).\]
On the other hand $\chi(\hat{\Delta}_{1, 3}) = 1 + b_2(\hat{\Delta}_{1, 3})$, so computing the first page is enough. 

Now
\[E_{1}^{i,j} = \bigoplus_{G \text{ a graph on $\{1, 2, 3\}$ with exactly $i + 1$ edges}} H^j(\hat{\Delta}^{G}_{1, 3}),\]
So we can compute $\dim(E_1^{i,j})$ for every value of $i$ and $j$. If $i= 0$ we are looking at $\hat{\Delta}_{1, 3}^G$ for $G$ a graph with three vertices and one edge. Suppose the edge is between vertex $1$ and vertex $2$, then we may place vertex 1 anywhere on $\R$ then vertex $2$ at either point of the $S^0$ centered at the location of vertex 1. Next vertex 3 may be mapped into $\R$ arbitrarily. So $\Delta_{1, 3}^G$ is given by $S^0 \times \R \times \R$ which compactifies in $\R^3$ to $S^2 \vee S^2$, and there are three graphs with exactly 1 edge. Next we look at graphs with two edges. If $G$ is such a graph then it is easy to see that $\Delta_{1, 3}^G$ is $\R \times S^0 \times S^0$ which compactifies to $S^1 \vee S^1 \vee S^1 \vee S^1$; there are three such choices for $G$. Finally if $G$ is the triangle then there is no way to map $G$ into the unit distance graph on $\R$ so $\Delta^G_{1, 3} = \emptyset$ which compactifies to the point at infinity. Therefore the following table stores the values of $\dim(E_1^{i,j})$
\begin{center}
\begin{tabular}{c|c|c|c}
2 & 6 & 0 & 0 \\ 
1 & 0 & 12 & 0 \\
0 & 3 & 3 & 1 \\ \hline
$j \uparrow, i \rightarrow$ & 0 & 1 & 2  \\
\end{tabular}
\end{center}
From this table we compute the Euler characteristic to be 19, so $b_2(\hat{\Delta}_{1, 3}) = 18$, from which it follows by duality that $b_0(\R^3 \setminus \Delta_{1, 3}) = 18 + 1 = 19$, recovering the number of labeled semiorders on $\{1, 2, 3\}$.
\end{example}

We could take the same approach to compute $b_0(\R^{1 \times n} \setminus \Delta_{1, n})$ for any $n$. For any graph $G$, $\Delta^G_{1, n}$ is always a product of a finite set $X := X(G)$, which is possibly empty, and $\R^{\beta_0(G)}$ so $\hat{\Delta}_{1, n}^G$ is a wedge of $|X|$ spheres of dimension $\beta_0(G)$ or the point at infinity if $X = \emptyset$. Moreover $X(G)$ will be empty if and only if there is no way to map $G$ into the unit distance graph on the real line, but this means that $X(G)$ is empty if and only if $G$ is not bipartite.

On the other hand if $G$ is bipartite we can compute $|X(G)|$ exactly, meaning that this approach could be used to compute the first page of the spectral sequence for $d = 1$ and any value of $n$. However, we should not expect to be able to do so in a reasonable amount of time for large $n$, assuming P $\neq$ NP as we explain below.

We have discussed $\Delta_{d, n}^G$ as the space of graph homomorphisms from $G$ into the unit distance graph on $\R^d$. but it turns out that given two graphs $G$ and $H$, it is in general NP-complete to decide if there are any graph homomorphisms from $G$ to $H$ by a result of Hell and Ne\v{s}et\v{r}il \cite{HellNesetril}. Indeed, \cite{HellNesetril} show that for any nonbipartite graph $H$ it is NP-complete to decide whether a graph $G$ admits any homomorphism into $H$. For us this means that we should not be able to even decide in general if $\Delta_{d,n}^G$ is empty or not if $d \geq 2$ as the unit distance graph on $\R^d$ for $d \geq 2$ is not bipartite. Now for $d = 1$, $\Delta_{d, n}^G$ will be computable, but a result Dyer and Greenhill \cite{DyerGreenhill} shows that the problem of enumerating graph homomorphism into a fixed bipartite graph $H$ is \#P-complete unless $H$ is a special type of bipartite graph, which does not include the unit distance graph on $\R$. Given these results, we should look for other ways to exactly count the number of isotopy classes given $d$ and $n$ than computing the full Mayer--Vietoris spectral sequence. To have at least some explicit examples for larger $d$ we work out the Betti numbers for $n = 3, 4$ and any $d \geq 2$, though the methods we use are rather ad hoc and doesn't generalize to higher values of $n$.
\begin{example}[$n = 3$, $d \geq 2$]
Here we explicitly compute the topology of each rigid isotopy class on 3 vertices. The space $W_{G, d}$ for $G$ the complete graph on 3 vertices is contractible, in fact it is easy to see that it is convex. 

For $G$ on two edges we observe that $W_{G, d}$ has the topology of $S^{d -1}$. To see this we consider $G$ as a path on 3 vertices, the first vertex on the path can go anywhere, the second vertex can go anywhere in the punctured ball of radius 1 around the first vertex; it must be punctured since only vertices with identical closed neighborhoods can map to the same point. After mapping the first two vertices the final vertex can be moved freely in some contractible subspace of $\R^d$ determined by the position of the first two vertices. 

Now for $G$ on 1 edge we observe have that $W_{G,d}$ has the topology of the configuration space of two points in $\R^d$, which is just $S^{d - 1}$. Finally if $G$ is the empty graph then $W_{G, d}$ is homotopy equivalent to the configuration space of 3 points in $\R^d$ which is known to have $b_0 = 1$, $b_{d - 1} = 3$ and $b_{2(d - 1)} = 2$. 

Putting this all together, we have that the empty graph contributes 1 to $b_0$, $3$ to $b_{d - 1}$ and $2$ to $b_{2(d - 1)}$, the three graphs on one edge each contribute $1$ to $b_0$ and $1$ to $b_{d - 1}$, the three graphs on two edges also each contribute $1$ to $b_0$ and $1$ to $b_{d - 1}$ and the complete graph contributes $1$ to $b_0$. So we have $8 + 9x^{d - 1} + 2x^{2d - 2}$ as the Poincar\'e polynomial for $\R^{d \times 3} \setminus \Delta_{d, 3}$.
\end{example}

\begin{example}[$n = 4$, $d \geq 2$]\label{ex:4vertices}
The case $n = 4$ is more interesting but again the computation is rather ad hoc. We examine each graph $G$ on 4 vertices and compute the Poincar\'e polynomial for $W_{G, d}$ summarized in Table \ref{tbl:PoincarePolynomials}. These Poincar\'e polynomials are essentially computed by inspection; we don't give the full details for the computations. For example the graph given by a path on 3 vertices and an isolated vertex is has the homotopy type of $S^{d - 1} \times S^{d - 1}$. With one vertex of the path fixed, the next vertex is free to go anywhere in the punctured ball around the first vertex. Next the final vertex of the path may be place freely inside some contractible set. So the path contributes a factor of $S^{d - 1}$ Finally after the path is placed in $\R^d$ the union of the balls around its vertices gives a contractible space in $\R^d$ and the isolated vertex may be place anywhere outside of this contractible space so this contributes the other $S^{d - 1}$ factor.  From Table \ref{tbl:PoincarePolynomials} we can determine $\beta(G)$ for any graph on $4$ vertices, and determine that the Poincare Polynomial of $\R^{d \times 4} \setminus \Delta_{d, 4}$ is 
\[64 + 7x^{d - 2} + 92x^{d - 1} + 7x^{2d - 3} + 35x^{2d - 2} + 6x^{3d - 3}\]
In particular in the case of the plane there are 71 rigid isotopy class of graphs on 4 vertices, while there are 64 labeled graphs on 4 vertices all of which can be realized as geometric graphs in the plane. We can also recover the number of rigid isotopy classes for 4 points on the real line with the observation that the 4-cycle and the star cannot be realized as geometric graphs in $\R$. Therefore in the case that $d = 1$ the 211 coming from evaluating the Poincare polynomial at $x = 1$ overcounts by the contribution of 4 for each 4-cycle and for each star, so the overcount is 28, bring the total number of chambers to 211 - 28 = 183.

\end{example}